\documentclass{article}

\usepackage{amsthm}
\usepackage{graphicx} 
\usepackage{amsfonts}
\usepackage{amsmath}
\usepackage{amssymb}
\usepackage{color}
\usepackage{url}
\usepackage{array}

\usepackage{subcaption}
\usepackage{tabularx}
\usepackage[ruled,linesnumbered,vlined]{algorithm2e}

\newtheorem{definition}{Definition}
\newtheorem{proposition}{Proposition}
\newtheorem{corollary}{Corollary}
\newtheorem{assumption}{Assumption}
\newtheorem{remark}{Remark}

\usepackage{tikz}
\usetikzlibrary{shapes.geometric, calc}
\usepackage{booktabs}
\usepackage{longtable}
\usepackage{adjustbox}
\usepackage{pdflscape}
\usepackage{float} 
\usepackage[colorlinks=true, linkcolor=blue, citecolor=blue, urlcolor=blue]{hyperref}
\hypersetup{
	pdftitle={Maximum Covering Network Design on Graphs with Low Connectivity: Dynamic Programming and Block-Cut Trees},
	pdfauthor={Felix Rauh, Jannik Matuschke, Hande Yaman},
	pdfkeywords={facility location, network design, dynamic programming, integer programming}
}

\newcommand{\suppref}[1]{Appendix~\ref{#1}}
\newcommand{\extraref}[1]{Appendix~\ref{#1}}

\usepackage[textsize=tiny,disable]{todonotes}
\usepackage{natbib}
\bibpunct[, ]{(}{)}{,}{a}{}{,}%

\title{Maximum Covering Network Design on Graphs with Low Connectivity:\\
Dynamic Programming and Block-Cut Trees}
\author{Felix Rauh, Jannik Matuschke, Hande Yaman}

\begin{document}
	\maketitle

	\begin{abstract}
		Planning accessible public services such as health care, emergency response, and schools
		often requires not only choosing where to open facilities but also improving the network
		that connects people to them, for example upgrading flood-prone roads in vulnerable regions.
		Most location models, however, take the network as fixed and the budget as given. We study
		the Maximum Covering Network Design Problem, in which a single budget is shared between
		opening facilities and upgrading weak links to maximize the population within a target travel
		distance of an open facility. The problem is hard even on the simplest networks, and planners
		usually want to see how coverage grows with the budget, not a single plan.

		We develop an exact dynamic-programming framework that exploits a property common to real
		road networks: their low connectivity, with many cut points whose removal disconnects the
		network. On trees, the recursion is self-contained: each state reduces to a few simple facility and
		upgrade choices that are fast to compute without a solver, giving predictable running times;
		for larger budgets and travel distances it outperforms solving the MILP formulation directly. On general low-connectivity networks, the framework
		decomposes the problem at the cut points and embeds a given MILP formulation to solve the
		resulting pieces, coordinating them through coverage conditions at the interfaces. This lets
		us compare a formulation on its own against the same formulation inside the framework: across
		306 test cases the framework matches or outperforms direct solving on more than 80\% of
		instances. Because it evaluates all budget levels in a single run, it also yields the full
		coverage-versus-budget curve at no extra cost, whereas direct solving must split its time
		across individual budgets.
	\end{abstract}


	\section{Introduction}
\label{sec:intro}

Infrastructure planning in health care, education, and emergency services involves fundamental trade-offs between facility investments and network improvements.
Opening new facilities increases service capacity, while improving connectivity can extend the reach of existing facilities.
This trade-off is particularly relevant in regions with underdeveloped or vulnerable infrastructure, where accessibility may be limited by poor road conditions (such as flooded gravel roads) rather than a lack of facilities;
investments may then include upgrading gravel roads or installing drainage systems~\citep{cookIntegratingDisasterGovernance2019,rozenbergRockyRoadSmooth2019}.
Despite the practical relevance of combining network design with facility location, most of the location literature assumes the underlying network is fixed.
A second limitation is that the budget is typically treated as a fixed input, whereas in practice it is often part of the planning question: decision-makers need to see how achievable coverage grows with the budget (the coverage-versus-budget curve) rather than the optimal plan for a single budget value.

We address both.
We study a model (the MCNDP, defined below) that jointly optimizes facility locations and network upgrades, and develop exact methods that exploit the low connectivity of real-world networks through a decomposition approach.
Because the decomposition evaluates many budget levels within a single solve, it yields the full budget curve at little additional cost.

\paragraph{The Maximum Covering Network Design Problem.}
We study the \emph{Maximum Covering Network Design Problem} (MCNDP) on an undirected graph $G = (N, E)$,
combining facility location and edge upgrade decisions under a joint budget constraint.
Certain edges are ``weak'' and require investment to become usable;
a coverage radius $\bar{d}$ specifies the maximum distance within which a demand node can be served by an open facility.

\begin{definition}[Maximum Covering Network Design Problem (MCNDP)]
	\label{def:mcndp}
	Given an undirected graph $G = (N, E)$ with demand weights $h_k \in \mathbb{Z}_{\ge 0}$ for $k \in N$, facility opening costs $c^F_j\in \mathbb{Z}_{\ge 0}$ for candidate sites $j \in F \subseteq N$,
	edge upgrade costs $c^E_e\in \mathbb{Z}_{\ge 0}$ and lengths $\ell_e\in \mathbb{Z}_{\ge 0}$ for edges $e \in E$,
	a coverage radius $\bar{d}\in \mathbb{Z}_{> 0}$, and a budget $\bar{b}\in \mathbb{Z}_{> 0}$,
	the MCNDP consists of choosing a set $Y \subseteq F$ of facilities to open and a set $X \subseteq E$ of edges to upgrade such that:
	\begin{enumerate}
		\item The total investment satisfies the budget: $\sum_{j \in Y} c^F_j + \sum_{e \in X} c^E_e \leq \bar{b}$.
		\item The total weight of covered nodes is maximized, where a node $k$ is covered if there exists $j \in Y$ with $d_{G[X]}(k, j) \leq \bar{d}$, where $G[X]=(N,X)$  is the subgraph with edge set $X$ and $d_{G[X]}(k, j)$ is the length of a shortest $k$-$j$-path in $G[X]$.
	\end{enumerate}
\end{definition}

Without loss of generality, we assume $\ell_e > 0$ for all $e \in E$, since any zero-length edge can be contracted in preprocessing.
Existing infrastructure is captured by zero costs: an already-usable edge has $c^E_e = 0$ and an already-open facility $c^F_j = 0$.

The MCNDP generalizes the classical Maximum Covering Location Problem (MCLP), which arises when all edges are already passable (zero upgrade cost).
Moreover, the MCNDP is closely related to the network design problem (NDP), or, more precisely, the multicommodity fixed-charge NDP~\citep{magnantiNetworkDesignTransportation1984}, which selects edges at minimum cost to route flow between origin--destination pairs.
The MCNDP differs from it in objective (maximizing coverage under a budget constraint rather than minimizing cost) and can be viewed as its single-source variant, with all facilities connected to a dummy super-source and all nodes acting as potential targets.

\paragraph{Complexity.}
The MCNDP is NP-hard even on trees: on a star graph with an open facility at the center,
selecting which leaf nodes to connect under a budget constraint corresponds to the 0-1 knapsack problem.
Moreover, even connecting a single demand node to a facility is NP-hard in general,
as finding a minimum-cost path of length at most $\bar{d}$ corresponds to a resource-constrained shortest path problem.
Consequently, no constant-factor approximation exists:
no polynomial-time algorithm can decide whether the optimal coverage is 0 or $h_k$ for an arbitrarily large weight $h_k$ (unless $\text{P}=\text{NP}$).

\paragraph{Focus on low-connectivity graphs.}
While the MCNDP is challenging on general graphs, many real-world networks exhibit low connectivity.
Road networks in rural or developing areas often have few alternative routes due to geographical constraints such as valleys, rivers, or islands;
\citet{tianArticulationPointsComplex2017} show that articulation points, whose removal disconnects the graph, comprise almost 20\% of nodes in three US state road networks.
We therefore focus on trees and, more generally, graphs whose decomposition at articulation points yields subproblems of smaller size
and enables efficient evaluation across many budget levels at once (exploited in Section~\ref{sec:bct} to produce full coverage-versus-budget curves).

\paragraph{Contributions and methodology.}
Our main contributions are as follows:
\begin{itemize}
	\item For tree graphs, we present an assignment-based formulation and a pseudo-polynomial dynamic programming (DP) algorithm that exploits the unique path structure (Section~\ref{sec:treedp}).
	\item For connected graphs with \emph{articulation points} (cut vertices whose removal disconnects the graph), we introduce a decomposition framework based on block-cut trees.
	The key insight is that subproblems on biconnected components can be solved with parameterized border conditions at articulation points, capturing both budget allocation and coverage interactions across blocks (Section~\ref{sec:bct}).
	\item The framework is independent of how the block subproblems are solved: blocks are coordinated only through selector variables and signed border conditions, so each can be handled by any suitable block-level MILP formulation.
	\item In our experiments, we compare our approaches against solving the mixed-integer linear programming (MILP) formulation directly (Section~\ref{sec:experiments}). Since the DP enumerates all budget levels inherently, it yields the complete coverage-versus-budget curve at no additional cost.
\end{itemize}

Beyond the MCNDP, the same framework could extend to other location and network design problems, such as uncapacitated facility location; we discuss possibilities for such extensions in Section~\ref{sec:conclusions}.

\paragraph{Outline.}
The remainder of this paper is organized as follows.
Section~\ref{sec:lit_review} reviews related literature on maximum covering location problems, network design, and decomposition techniques.
Section~\ref{sec:mips} presents a MILP formulation for general graphs and one for tree graphs.
Section~\ref{sec:dp} develops the DP decomposition framework, first applied to trees and then extended to block-cut tree decomposition.
Section~\ref{sec:experiments} reports computational results, and Section~\ref{sec:conclusions} concludes.

\section{Literature review and related problems}
\label{sec:lit_review}

This section reviews the literature on maximum covering location problems, network design, and their combination.
We also discuss solution approaches relevant to our methodology, particularly dynamic programming on trees and decomposition techniques based on graph connectivity.

\paragraph{Maximum Covering Location Problems.}
The MCLP, as introduced by \citet{churchMaximalCoveringLocation1974}, asks to open at most $p$ facilities to maximize covered demand within a distance radius; see \citet{laporteLocationScience2019} for an overview.
The problem is known to be well solvable with integer linear programming~\citep{snyderCoveringProblems2011}, and the state-of-the-art solver by \citet{cordeauBendersDecompositionVery2019a} handles instances with up to 15 million demand points via Benders decomposition.
On trees, polynomial-time DPs are known: \citet{megiddoMaximumCoverageLocation1983} achieve $\mathcal{O}(pn^2)$ for the MCLP, and \citet{tamirOpn2AlgorithmPmedian1996} the same for the related $p$-median problem, using a binary-tree transformation that we also adopt.
Our DP in Section~\ref{sec:treedp} extends these works by replacing the $p$-facility constraint with a knapsack-type budget and adding edge upgrade decisions, yielding pseudo-polynomial rather than polynomial complexity.

\paragraph{Network Design Problems.}
The network design problem (NDP), formalized by \citet{magnantiNetworkDesignTransportation1984} using flow-based formulations, selects edges to satisfy connectivity requirements at minimum cost; see \citet{crainicNetworkDesignApplications2021a} for an overview.
In the multicommodity fixed-charge NDP closest to our setting, a separate commodity is routed for each origin--destination pair (the \emph{multicommodity} aspect), and each selected edge incurs a fixed cost independent of the flow it carries (the \emph{fixed-charge} aspect).
Our flow formulation in Section~\ref{sec:flowmip} builds upon these classical flow-based formulations.
Exact methods for this problem are typically limited to hundreds or a few thousands of nodes and edges~\citep{zetinaExactAlgorithmsBased2019}.
Our framework is independent of the underlying connectivity model and could equally use cut-based formulations~\citep{arslanFlexibleNaturalFormulation2019,arslanBranchandCutAlgorithmAlternative2019}; we focus on the flow formulation throughout this paper.

\paragraph{Combined Facility Location and Network Design.}
An overview of combined facility location and network design problems is provided by \citet{contrerasGeneralNetworkDesign2012}.
\citet{melkoteIntegratedModelFacility2001,melkoteCapacitatedFacilityLocation2001} introduce the uncapacitated facility location/network design problem (UFLNDP) with a cost-minimization objective, studying the trade-off of improved network topology versus opening new facilities.
Structurally close to our setting is the network design problem with relays (NDPR) studied by \citet{leitnerExactApproachesNetwork2019}, in which signal regeneration devices are placed at nodes while additional links may be installed, so that the distance between consecutive devices along a path stays below a given threshold.
As in the MCNDP, node installation and edge installation therefore interact through a distance bound, and the authors likewise develop exact mixed-integer programming approaches (branch-and-price and branch-price-and-cut) rather than heuristics.
The NDPR differs from the MCNDP in that it minimizes total installation cost while serving all origin-destination pairs, whereas we maximize covered demand under a single budget shared between facility and upgrade decisions.

Most relevant to our work is \citet{bucareyBendersDecompositionNetwork2022}, who maximize covered origin-destination pairs subject to a budget constraint on network construction and a length bound on paths.
The MCNDP is a special case of their problem (with a single super-source).
In an earlier study, \citet{murawskiImprovingAccessibilityRural2009} study a similar coverage-maximizing network improvement model with fixed facility locations, solving a flow formulation on a case study in Ghana.
\citet{vanveggelRoadNetworkResilience} address the same model at larger scale with a greedy heuristic.
A related variant by \citet{baldomero-naranjoUpgradingEdgesMaximal2022,baldomero-naranjoComplexityUpgradingVersion2024} considers shortening existing edges rather than binary upgrades, with a complexity analysis on trees similar in spirit to our DP in Section~\ref{sec:treedp}.

\paragraph{Decomposition via Block-Cut Trees.}
Several optimization problems have been solved by decomposing graphs at articulation points into biconnected components (blocks).
\citet{baiouIntegralityFacilityLocation2009} show that the integrality of facility location polytopes can be established block by block: if constraints describe an integral polyhedron for each block, so does the combined system with interaction constraints at articulation points.
Block decomposition has also been applied to degree-dependent spanning tree problems~\citep{landeteDecompositionMethodsBased2017}, switch location~\citep{landeteLocatingSwitches2019},
the $d$-Minimum Branch Vertices problem~\citep{morenoExactHeuristicApproach2018}, the generalized vertex cover problem~\citep{correcherDecompositionTechniquesGeneralized2025},
and the generalized network dismantling problem~\citep{fengNovelAlgorithmGeneralized2024}.

A common property of these approaches is that subproblems on blocks are solved independently, with case distinctions or special constraints handling the interactions at articulation points.
Our approach differs in that we must allocate a shared budget across blocks, and the covering interactions at articulation points involve both directions and distances.
To the best of our knowledge, the MCNDP has not been studied with such a decomposition approach.

\section{MILP formulations for the MCNDP}
\label{sec:mips}
We present two formulations: a flow-based formulation for general graphs (Section~\ref{sec:flowmip}) and an assignment-based formulation for trees (Section~\ref{sec:treemip}).
The flow formulation uses multicommodity flow constraints for connectivity in the tradition of \citet{magnantiNetworkDesignTransportation1984}, and is closest to the flow formulation of \citet{bucareyBendersDecompositionNetwork2022}.

\paragraph{High-level formulation.}
In all formulations, we use binary variables $z_k\in\{0,1\}$ indicating coverage of demand node $k\in N$ and $x_e\in\{0,1\}$ for the decision to upgrade edge $e$ (edges that already exist have $c^E_e = 0$ and $x_e$ can be fixed to 1).
For the flow formulation, facility openings can be represented as edge decisions by introducing a source node $s$ and the facility-edge set $E^F := \{\{s,j\} \mid j \in F\}$, yielding the extended edge set $\overline{E} := E \cup E^F$.
These facility edges satisfy $\ell_{\{s,j\}} := 0$ and $c^E_{\{s,j\}} := c^F_j$.
We use this construction whenever a single flow formulation MILP is solved, on the full graph or on an individual block, but not across the decomposition, as a global source node $s$ would not preserve the tree and block-cut tree structure that the decomposition exploits.
The tree formulation (Section~\ref{sec:treemip}) keeps facilities as nodes of $G$, using explicit facility variables $y_j = x_{\{s,j\}} \in \{0,1\}$.
For a comprehensive summary of all notation, see \suppref{sec:appendix_notation}.

Using only the variables $(x, z)$, both formulations presented below can be written in the unified form
\begin{equation}
\label{eq:unified}
\max\Bigl\{\sum_{k\in N} h_k z_k \Bigm| (x, z) \in \mathcal{F},\;\; \sum_{e \in \overline{E}} c^E_e x_e \le \bar{b},\;\; x_e,z_k \in \{0,1\},\;\; e\in \overline{E},\;\; k \in N \Bigr\},
\end{equation}
where the feasibility set $\mathcal{F}$ encodes the connectivity and covering constraints.
The formulations differ only in how $\mathcal{F}$ is represented:
via flow variables (Section~\ref{sec:flowmip}) or assignments along unique paths on trees (Section~\ref{sec:treemip}).
In Section~\ref{sec:bct_mip}, we extend this to the block-cut tree setting,
where $\mathcal{F}$ is parameterized by border conditions at the articulation points.

\subsection{Flow formulation}
\label{sec:flowmip}
The flow formulation uses the source-node construction introduced above.
Let $\overline{N} := N \cup \{s\}$.
With this construction, a node $i \in N$ is covered if and only if there exists an $s$-$i$-path of length at most $\bar{d}$ via edges in the solution.

To describe coverage as a directed flow from $s$ to the demand nodes, let $\overline{A}$ denote the set of directed arcs containing both $(u,v)$ and $(v,u)$ for each $\{u,v\}\in\overline{E}$.
We use one flow commodity per demand node $k \in N$, and denote the respective flow variables by $f^{k}_a\in\{0,1\} $
for $a\in \overline{A}$.
For a node $v \in \overline{N}$, let $\delta^+(v)$ denote all outgoing arcs and $\delta^-(v)$ all incoming arcs of $v$ in $\overline{G}= (\overline{N},\overline{A})$.

Using the edge representation for facility openings as discussed above, we obtain:
{
	
	\begin{align}
		\max \quad & \sum_{k \in N} h_k z_k &&&& \label{eq:flow2_objective} \\[2pt]
		\text{s.t.} \quad & \sum_{e \in \overline{E}} c^E_e x_e
		&\;\le\;& \bar{b} && \label{eq:flow2_budget} \\
		& \sum_{a \in \delta^+(v)} f^k_a - \sum_{a \in \delta^-(v)} f^k_a
		&\;=\;& \begin{cases} z_k & v=s,\\ -z_k & v=k,\\ 0 & \text{otherwise}
		\end{cases}
		&\quad& \forall v \in \overline{N},\, \forall k \in N
		\label{eq:flow2_conservation} \\
		& f^k_{(u,v)} + f^k_{(v,u)}
		&\;\le\;& x_e
		&\quad& \forall e=\{u,v\} \in \overline{E},\, \forall k \in N
		\label{eq:flow2_link_edge} \\
		& \sum_{a \in \overline{A}} \ell_a f^k_a
		&\;\le\;& \bar{d}\,z_k
		&\quad& \forall k \in N \label{eq:flow2_length} \\
		& \rlap{$x_e \in \{0,1\}\ \forall e \in \overline{E}, \quad z_k \in \{0,1\}\ \forall k \in N, \quad f^k_a \in \{0,1\}\ \forall k \in N,\, a \in \overline{A}$} &&&& \label{eq:flow2_domain}
	\end{align}
	}
Objective \eqref{eq:flow2_objective} and budget constraint \eqref{eq:flow2_budget} are as in the high-level formulation.
Constraint \eqref{eq:flow2_conservation} enforces flow conservation with supply at $s$ and demand at $k$;
\eqref{eq:flow2_link_edge} links flow to edge decisions (including facility edges adjacent at $s$);
and \eqref{eq:flow2_length} bounds the path length for covered nodes by $\bar{d}$.
Theoretically, the integrality of the flow variables $f$ in~\eqref{eq:flow2_domain} can be relaxed, since given integral $x$ and $z$, an integral optimal flow exists by standard network flow arguments~\citep{bucareyBendersDecompositionNetwork2022}.
In our experiments, however, relaxing it does not improve solver performance.

\subsection{Assignment-based MILP formulation on trees}
\label{sec:treemip}
When $G$ is a tree, paths are unique: a node $k$ is covered by facility $j$ if and only if $j$ is open and all edges on the unique $j$-$k$-path of length at most $\bar{d}$ are built.

We use variables $x_e$, $y_j$, and $z_k$ as introduced before, and add assignment variables $r_{jk} \in \{0,1\}$ indicating that demand node $k$ is served by facility $j$.
Let $P_{jk} \subseteq E$ denote the set of edges on the unique path between $k\in N$ and facility $j\in F$.
For facilities $j$ that exceed the length bound to node $k$, i.e. $d(k,j)>\bar{d}$, we can set $r_{jk}=0$ and consider for each demand node $k$ its set of reachable facilities as $F_k \subseteq F$.
\begin{align}
	\max \quad & \sum_{k \in N} h_k z_k &&&& \label{eq:tree1_obj} \\[2pt]
	\text{s.t.} \quad & \sum_{j\in F}c^F_j y_j + \sum_{e \in E}c^E_e x_e
	&\;\leq\;& \bar{b} && \label{eq:tree1_budget} \\
	& r_{jk} &\;\leq\;& y_j &\quad& \forall j \in F_k,\, \forall k \in N \label{eq:tree1_open} \\
	& r_{jj} &\;=\;& y_j &\quad& \forall j \in F \label{eq:tree1_self} \\
	& r_{jk} &\;\leq\;& x_e &\quad& \forall e\in P_{jk},\, \forall j\in F_k,\, \forall k\in N \label{eq:tree1_path} \\
	& z_k &\;=\;& \sum_{j\in F_k} r_{jk} &\quad& \forall k \in N \label{eq:tree1_cover} \\
	& r_{jk} &\;\geq\;& 0 &\quad& \forall j\in F_k,\, \forall k\in N \label{eq:tree1_nonneg} \\
	& \rlap{$x_e \in \{0,1\}\ \forall e \in E, \quad y_j \in \{0,1\}\ \forall j \in F, \quad z_k \in \{0,1\}\ \forall k \in N.$} &&&& \label{eq:tree1_domain}
\end{align}

Objective, budget, and integrality constraints \eqref{eq:tree1_obj}--\eqref{eq:tree1_budget}, \eqref{eq:tree1_domain} are as in the high-level formulation.
Constraint \eqref{eq:tree1_open} requires the facility to be open for an assignment, \eqref{eq:tree1_self} ensures an open facility serves itself, \eqref{eq:tree1_path} enforces that all path edges are built, and \eqref{eq:tree1_cover} links assignments to coverage.
Note that, in \eqref{eq:tree1_domain}, the integrality of $r_{jk}$ has been relaxed to non-negativity by a $ \pm \epsilon$-argument that shows that $r_{jk}$ cannot be fractional at an extreme point.

\section{Dynamic programming approach}
\label{sec:dp}
We develop a dynamic programming (DP) decomposition framework for the MCNDP.
The core idea is to decompose the graph along a tree structure and solve subproblems bottom-up, combining their solutions via budget allocation and coverage interactions at interface points.
On trees (Section~\ref{sec:treedp}), each node constitutes a simple subproblem with a small number of binary decisions.
On general graphs with limited connectivity (Section~\ref{sec:bct}), the block-cut tree provides the decomposition, and each biconnected component becomes a subproblem solved via MILP.
Unlike the classic $p$-median DP on trees, our budget constraint is of knapsack type, leading to pseudo-polynomial time complexity \citep{martelloKnapsackProblemsAlgorithms1990}.
Table~\ref{tab:tree_vs_bct} summarizes the key differences between the two graph classes.

\begin{table}[ht]
	\centering
	\small
	\begin{tabular}{l|l|l}
		\hline
		\textbf{Component} & \textbf{DP on trees (Sec.~\ref{sec:treedp})} & \textbf{DP on BC-trees (Sec.~\ref{sec:bct})} \\
		\hline
		Subproblem & Single node $v$ & Biconnected component $B$ \\
		Tree structure & Rooted binary tree & Rooted block-cut tree \\
		Interface points & Every node & Articulation points \\
		Budget allocation & Split between children $v_1, v_2$ & Split between $B$ and children \\
		Relevant radii & $\mathcal{O}(\min(n, \bar{d}))$ values & $\mathcal{O}(\bar{d})$ values \\
		Local decisions & Open facility? Upgrade edges? & Solve MCNDP on $B$ (NP-hard) \\
		\hline
	\end{tabular}
	\caption{Comparison of the DP framework on trees and block-cut trees.
		On trees, each subproblem has a small number of binary decisions.}
	\label{tab:tree_vs_bct}
\end{table}

\subsection{Border conditions and DP states}
\label{sec:dp_framework}

We describe the DP abstractly, on a rooted \emph{decomposition tree} whose nodes are the subproblems.
It has two instantiations: the rooted binary tree of Section~\ref{sec:treedp}, where each subproblem is a single graph node $v$, and the rooted block-cut tree of Section~\ref{sec:bct}, where each subproblem is a block $B$.

A subproblem $S$ is therefore a set of vertices---a single vertex $\{v\}$ on trees, or the vertices of a block $B$ on BC-trees---equipped with a set of \emph{child subproblems} $\mathcal{C}_S$ and a single \emph{interface point} connecting it to its parent: the parent edge on trees, or the parent articulation point on BC-trees.
We write $T_S$ for the \emph{subtree rooted at $S$}, the subgraph induced by $S$ together with all its descendants in the decomposition tree.
On trees this is the usual subtree $T_v$; on BC-trees it is the subgraph formed by the union of all blocks in the subtree of the block-cut tree rooted at $S$.

The DP value $V(S, b, r)$ is the maximum coverage achievable within the \emph{entire} subtree $T_S$, under budget $b \in \{0, \ldots, \bar{b}\}$ and a \emph{signed covering condition} $r \in \{-\bar{d}, \ldots, -1, 0, 1, \ldots, \bar{d}\}$ at the interface point, whose sign indicates the direction of coverage flow and whose magnitude indicates its reach.
Local decisions inside $S$ (opening facilities, upgrading edges) must conform to the local budget and the covering condition, while the child subtrees enter recursively through the budget allocated to each child and the choice of child states.
The three sign cases of $r$ are:
\begin{itemize}
	\item $r = 0$ (\emph{BLANK state}): the two sides are decoupled at the interface; no coverage flows through it.
	At the root, this is the only relevant condition, and $V(\text{root}, \bar{b}, 0)$ yields the optimal value.
	\item $r > 0$ (\emph{OUT state}): an external facility at the parent side provides coverage with remaining radius $r$ at the interface point,
	i.e., nodes in $S$ within distance $r$ of the interface point are covered for free.
	\item $r < 0$ (\emph{INT state}): the subproblem must provide coverage to the parent side; some facility in $S$ must be within distance $\bar{d}-|r|$ of the interface point, or, equivalently, nodes on the parent side with maximum distance $|r|$ can be covered.
\end{itemize}
The three cases are illustrated in Figure~\ref{fig:dp_border_conditions} (drawn for the tree case; the block-cut tree case is analogous, with the node $v$ replaced by a block $B$).
Intuitively, large positive $r$ makes abundant external coverage available to $S$, whereas large negative $r$ imposes a tight outward-coverage obligation that may not always be feasible.
This generalizes the two-state framework of \cite{megiddoMaximumCoverageLocation1983}, which already distinguishes OUT and INT states (there, for the maximum covering problem with $p$ facilities); here, the signed index unifies all three cases into a single function $V$ that is monotone in $r$.
In what follows, we use the sign of $r$ and the state names interchangeably, preferring the names when it aids readability (e.g., ``the OUT case'' for $r>0$).


\begin{figure}[ht]
\centering
\tikzset{
    itria/.style={
        draw, dashed, shape border uses incircle,
        isosceles triangle, shape border rotate=90, anchor=apex
    }
}

\begin{tabular}{ccc}
\begin{tikzpicture}[scale=.55, sibling distance=4cm, 
    level 2/.style={sibling distance=2.5cm}]
    \node[circle,draw] {\tiny $v_r$}
    child{node[circle,draw] {}{ node[draw, itria, scale=.5] {\phantom{Q1}}}}
    child{ node[circle, draw] {}
        child[thick, red, dashed]{ node[circle, draw, solid, black, fill=blue!20] {$v$}
            child[draw, solid, black]{node[circle,draw] {}
                {node[black, itria, scale=.35] {\phantom{Q1}}}}
            child[draw, solid, black]{node[circle,draw] {}
                {node[black, itria, scale=.35] {\phantom{Q1}}}}
        }
        child[]{ node[circle, draw] {}}
    };
\end{tikzpicture}
&
\begin{tikzpicture}[scale=.55, sibling distance=4cm, 
    level 2/.style={sibling distance=2.5cm}]
    \node[circle,draw] {\tiny $v_r$}
    child{node[circle,draw] {}{ node[draw, itria, scale=.5] {\phantom{Q1}}}}
    child{ node[circle, draw] {}
        child[thick, green!50!black, ->]{ node[circle, draw, fill=blue!20] {$v$}
            child[draw, black, -]{node[circle,draw] {}
                {node[black, itria, scale=.35] {\phantom{Q1}}}}
            child[draw, black, -]{node[circle,draw] {}
                {node[black, itria, scale=.35] {\phantom{Q1}}}}
        }
        child[thick, green!50!black, <-]{ node[rectangle, fill=green!30, 
            draw=green!50!black, text=black] {$j$}}
    };
\end{tikzpicture}
&
\begin{tikzpicture}[scale=.55, sibling distance=4cm, 
    level 2/.style={sibling distance=2.5cm}]
    \node[circle,draw] {\tiny $v_r$}
    child{node[circle,draw] {}{ node[draw, itria, scale=.5] {\phantom{Q1}}}}
    child{ node[circle, draw] {}
        child[thick, green!50!black, <-]{ node[circle, draw, fill=blue!20] {$v$}
            child[draw, black, <-]{node[circle, draw, fill=green!30] {}
                {node[black, itria, scale=.4, fill=green!20,
                    append after command={
                        node[rectangle, fill=green!30, draw=green!50!black, inner sep=1pt]
                        at (\tikzlastnode.center) {\scriptsize $j$}
                    }] {\phantom{111}}}}
            child[draw, black, -]{node[circle,draw] {}
                {node[black, itria, scale=.35] {\phantom{Q1}}}}
        }
        child[thick, green!50!black, ->]{ node[circle, fill=gray!5, draw] {\scriptsize$w$}}
    };
\end{tikzpicture}
\\[1ex]
\small
\textbf{(a)} $V(v,b,0)$ &
\textbf{(b)} $V(v,b,r),\ r>0$ &
\textbf{(c)} $V(v,b,r),\ r<0$
\end{tabular}

\caption{The three cases of the signed-radius DP value $V(v,b,r)$ on trees.
    \textbf{(a)}~$r=0$: no covering interaction (BLANK).
    \textbf{(b)}~$r>0$: an external facility $j$ (OUT).
    \textbf{(c)}~$r<0$: an internal facility in $T_v$ (INT) can cover external nodes     $w$.
    }
\label{fig:dp_border_conditions}
\end{figure}
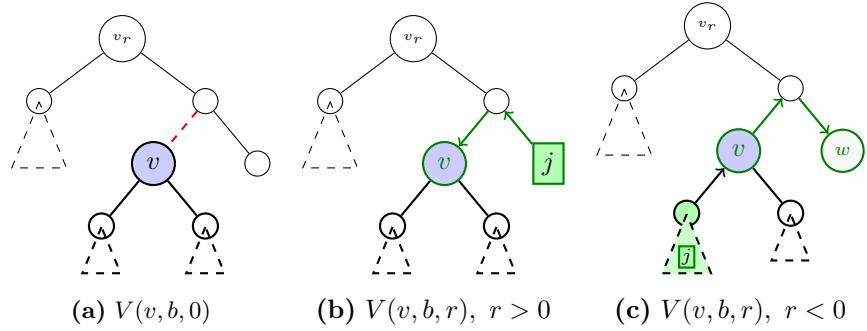

Not every value $r \in \{-\bar{d}, \ldots, \bar{d}\}$ needs to be considered: we denote by $\mathcal{R}_S$ the set of \emph{relevant signed radii} at the interface of $S$.
This set is obtained by pre-processing with the goal to only contain the distances at which coverage can actually be provided or received (for instance, $r$ is not required in $\mathcal{R}_S$ if no path of length $|r|$ from the interface point to another node exists).

\paragraph{High-level recursion.}
For a subproblem $S$ with child problems $\mathcal{C}_S$ and parent radius $r$, the DP optimizes over two types of decisions:
$x$, encoding local decisions inside $S$, and for each child $C \in \mathcal{C}_S$ a selection $\lambda^C = (b_C, r_C) \in \Lambda_C$, determining the budget allocation and covering condition imposed on $C$.
Here, $\Lambda_C \subseteq \{0,\ldots,\bar{b}\} \times \mathcal{R}_C$ denotes the set of feasible child states.
The recursion is:
\begin{equation}
	\label{eq:dp_highlevel}
	V(S, b, r) \;=\; \max_{x, \, \lambda} \Biggl(
	\underbrace{h_S(x,\lambda)}_{\text{coverage inside } S}
	\;+\; \underbrace{\sum_{C \in \mathcal{C}_S} V(C, b_C, r_C)}_{\text{child contributions}}
	\Biggr),
\end{equation}
where the child terms $V(C, b_C, r_C)$ are precomputed values that themselves represent the optimal coverage of their entire respective subtrees.
Thus, $V(S, b, r)$ accounts for coverage across the full subtree rooted at $S$, even though only the decisions in $S$ and the allocation to children are optimized at this level.
Note also that the local contribution $h_S(x,\lambda)$ depends on the child radii as well,
since a child with $r_C < 0$ can provide coverage to $S$ whereas a child with $r_C > 0$ means that $S$ needs to provide coverage.

\paragraph{Monotonicity.}
The DP values are monotonically non-decreasing in both arguments:
\[
V(S, b, r) \;\leq\; V(S, b', r')
\qquad \text{whenever } b \leq b' \text{ and } r \leq r'.
\]
Intuitively, a larger budget cannot hurt, and a larger signed radius either relaxes an obligation (when $r<0$) or provides more free coverage (when $r>0$).
Monotonicity allows pruning the state set $\Lambda_C$: any state $(b_C, r_C)$ that is dominated, i.e., another state $(b_C', r_C') \in \Lambda_C$ achieves at least the same DP value while consuming no more budget ($b_C' \le b_C$) and imposing a weaker covering condition ($r_C' \le r_C$, hence more favorable to the parent), can be removed without loss of optimality, since it would never be required in an optimal parent solution.

\subsection{DP on trees}
\label{sec:treedp}

\paragraph{Rooted tree and binary tree transformation.}
We root the tree $G$ at an arbitrary node $v_r$ and process subtrees bottom-up.
For simplicity of exposition, we assume that $G$ is a binary tree (each node has at most two children).
Any tree can be transformed into a binary tree by inserting at most $n$ dummy nodes without facility and demand (i.e., $c^F_v=\infty$ and demand weight $h_v=0$ for such a node $v$); this transformation preserves optimality \citep{tamirOpn2AlgorithmPmedian1996}.

\paragraph{DP states.}
For each node $v$, the subproblem is the subtree $T_v$, and the interface point is the parent edge $e$.
More specifically, a positive $r>0$ means the remaining coverage distance entering $e$ from the parent side, while a negative $r<0$ means the remaining coverage entering $e$ from $v$.
The set of considered signed radii at node $v$ derives from the distances to reachable nodes outside $T_v$,
\[
\mathcal{R}_v \; = \; \bigl\{\, \pm(\bar{d} - d_G(v,u)) \;\bigm|\; u \in N\setminus T_v,\; d_G(v,u) < \bar{d} \,\bigr\} \,\cup\, \{0\},
\]
with $|\mathcal{R}_v| \in \mathcal{O}(\min(|N|, \bar{d}))$.

\paragraph{Base case (leaves).}
At a leaf $v$ with a parent edge $e$ with length $\ell_e$, the only decision is whether to open a facility.
If $r \geq \ell_e$, then $v$ is already covered by the parent side for free.
Infeasible states (no facility or insufficient budget) in the $r<0$ case evaluate to $-\infty$.
\begin{align}
	V(v,b,r) \;=\; \begin{cases}
		h_v & \text{if } r \geq \ell_e, \\[2pt]
		\begin{cases} h_v & \text{if } v \in F \text{ and } b \ge c^F_v, \\ 0 & \text{otherwise,} \end{cases} & \text{if } 0 \leq r < \ell_e, \\[10pt]
		\begin{cases} h_v & \text{if } v \in F \text{ and } b \ge c^F_v, \\ -\infty & \text{otherwise,} \end{cases} & \text{if } r < 0.
	\end{cases}
\end{align}

\paragraph{Recursive case (inner nodes).}
Let $v$ be an inner node with children $v_1, v_2$ and child edges $e_1=\{v,v_1\}$, $e_2=\{v,v_2\}$ with upgrade costs $c^E_{e_1}, c^E_{e_2}$.
The local decisions are: open a facility at $v$, upgrade edge $e_1$, upgrade edge $e_2$.

We distinguish four coverage situations for $v$ (not covered; covered by a facility at $v$ or outside $T_v$; covered by the left subtree; covered by the right subtree), and capture each by a helper function that optimizes the budget split and edge-upgrade decisions (see Figure~\ref{fig:dp_decisions}).
Any helper or child state evaluated with negative remaining budget returns $-\infty$.


\begin{figure}[ht]
\centering
\tikzset{
    itria/.style={
        draw, dashed, shape border uses incircle,
        isosceles triangle, shape border rotate=90, anchor=apex
    }
}

\begin{tikzpicture}[scale=.7, sibling distance=5cm, 
    level 2/.style={sibling distance=3cm}]
    \node[circle,draw] {\tiny $r$}
    child{node[circle,draw] {}{ node[draw, itria, scale=.6] {\phantom{Q1}}}}
    child{ node[circle, draw] {}
        child[]{ node[circle, draw, fill=blue!20](v) {$v$}
            child[draw, black]{node[circle,draw](a) {}
                {node[black, itria, scale=.4] (b1) {\phantom{Q1}}}}
            child[draw, black]{node[circle,draw](c) {}
                {node[black, itria, scale=.4] (b2) {\phantom{Q1}}}}
        }
        child[]{ node[circle, draw](w) {}}
    };
    
    \node[draw, circle, blue, thick, minimum size=3.3cm] at ($(a)!0.5!(c)$) {};
    
    \draw[thick, red!70!black, dashed](v)--(a);
    \node[red!70!black, font=\small] at ($(v)!0.5!(a) + (-0.425,0.25)$) {$x_{e_1}$?};
    
    \draw[thick, red!70!black, dashed](v)--(c);
    \node[red!70!black, font=\small] at ($(v)!0.5!(c) + (0.425,0.25)$) {$x_{e_2}$?};
    
    \node[rectangle, fill=green!30, draw=green!50!black, text=black, font=\small] 
        at ($(v) + (1.15,0.33)$) {$y_v$?};
    
    \node[fill=blue!10, text=black, font=\small, rounded corners] at (a) {$v_1$};
    \node[fill=blue!10, text=black, font=\small, rounded corners] at (c) {$v_2$};
    
    \node[font=\footnotesize, text=green!50!black] at ($(a) + (-0.6, 0.3)$) {$b_1$};
    \node[font=\footnotesize, text=green!50!black] at ($(c) + (0.6, 0.3)$) {$b_2$};
    
    \node[blue, font=\small] at ($(a)!0.5!(c) + (0, -1.8)$) {Subtree $T_v$};
\end{tikzpicture}

\caption{Decisions at an inner node $v$ with children $v_1$ and $v_2$.
    The DP must decide:
    (i)~whether to open a facility at $v$ ($y_v \in \{0,1\}$),
    (ii)~whether to upgrade edges $e_1$ and $e_2$ ($x_{e_1}, x_{e_2} \in \{0,1\}$), and
    (iii)~how to allocate the remaining budget between subtrees ($b_1 + b_2 \le b$).
    The blue circle indicates the subtree $T_v$ rooted at $v$.
    }
\label{fig:dp_decisions}
\end{figure}
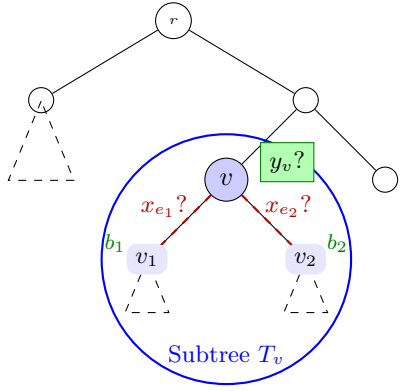

\begin{enumerate}

	\item $\mathcal{A}(v,b)$: No coverage at $v$; no edges upgraded; full budget $b$ distributed.
	\begin{equation}
		\mathcal{A}(v,b) := \max_{b_1+b_2\leq b} \Bigl[ V(v_1, b_1, 0) + V(v_2, b_2, 0) \Bigr]
	\end{equation}
	
	\item $\mathcal{B}(v,b,r)$: $v$ is covered by a facility at $v$ or through the parent edge and can pass a remaining coverage of $r$ to the children.
	\begin{equation}
		\mathcal{B}(v,b,r) := \max_{\substack{b_1+b_2\leq b \\ x_{e_1},x_{e_2}\in\{0,1\}}}
		\Bigl[ V(v_1, b_1- c^E_{e_1} x_{e_1}, r x_{e_1}) + V(v_2, b_2- c^E_{e_2} x_{e_2}, r x_{e_2}) \Bigr]
	\end{equation}
	Depending on the edge upgrade decisions, $\mathcal{B}$ calls $V$ for the child trees with adjusted budget and with covering condition either $r$ (if the edge is upgraded) or $0$ (if not), since coverage can only flow through an upgraded edge.

	\item $\mathcal{C}^1(v,b,r)$: Left subtree provides coverage of $ \rho \geq r + \ell_{e_1} $ through $e_1$ into $v$ such that a remaining coverage of at least $r$ to the parent side or right subtree is available.
	\begin{equation}
		\mathcal{C}^1(v,b,r) :=
		\max_{\substack{b_1+b_2\leq b-c^E_{e_1} \\ \rho \in \mathcal{R}_{v_1},\ \rho \geq r + \ell_{e_1} \\ x_{e_2}\in\{0,1\}}}
		\Bigl[ V(v_1, b_1, -\rho) + V(v_2, b_2- c^E_{e_2} x_{e_2}, (\rho - \ell_{e_1}) x_{e_2}) \Bigr] \label{eq:C}
	\end{equation}
	The left edge is necessarily upgraded, so the maximum is taken only over the right edge upgrade decision as well as the budget and the child radius $\rho$ that the left subtree can provide.
	\item $\mathcal{C}^2(v,b,r)$: Symmetric to $\mathcal{C}^1$.
\end{enumerate}

\paragraph{Resulting recursions.}
The DP values follow by maximizing over the coverage situations, each mirroring the structure of~\eqref{eq:dp_highlevel} with $h_v$ as the local coverage and the helpers as the child contributions.
Depending on whether the parent radius $r$ is sufficient to cover $v$ or instead imposes a covering obligation, a different subset of helpers is active:
\begin{align}
\label{eq:dp_recursion}
V(v,b,r) \;=\; \begin{cases}
\begin{aligned}[t]
\max\bigl\{\; & h_v + \mathcal{B}(v, b-c^F_v, \bar{d}),\;
	h_v + \mathcal{B}(v, b, r-\ell_e), \\
& h_v + \mathcal{C}^1(v, b, 0),\;
	h_v + \mathcal{C}^2(v, b, 0) \;\bigr\}
\end{aligned} & \text{if } r \geq \ell_e, \\[10pt]
\begin{aligned}[t]
\max\bigl\{\; & \mathcal{A}(v, b),\;
	h_v + \mathcal{B}(v, b-c^F_v, \bar{d}), \\
& h_v + \mathcal{C}^1(v, b, 0),\;
	h_v + \mathcal{C}^2(v, b, 0) \;\bigr\}
\end{aligned} & \text{if } 0 \leq r < \ell_e, \\[10pt]
\begin{aligned}[t]
\max\bigl\{\; & h_v + \mathcal{B}(v, b-c^F_v, \bar{d}), \\
& h_v + \mathcal{C}^1(v, b, |r|),\;
	h_v + \mathcal{C}^2(v, b, |r|) \;\bigr\}
\end{aligned} & \text{if } r < 0.
\end{cases}
\end{align}
All cases share the first term, which corresponds to opening a facility at $v$.
Coverage from the children via $\mathcal{C}^1$ and $\mathcal{C}^2$ is available in all cases, but only for $r<0$ must it reach the parent side.
Not building a facility and not using child coverage is only possible for $r\geq 0$ and corresponds to the cases $\mathcal{B}(v, b, r-\ell_e)$ and $\mathcal{A}(v, b)$.

\paragraph{Complexity.}
Let $R := \max_{v \in N} |\mathcal{R}_v| \le \min(n, \bar{d}+1)$ denote the maximum number of relevant radii at any node.
The DP table for $V$ has $\mathcal{O}(n \cdot R \cdot \bar{b})$ entries,
and each entry requires $\mathcal{O}(\bar{b} \cdot R)$ time for enumerating budget allocations and radii in the helper functions.
The total time complexity is therefore
\[
\mathcal{O}(n \cdot R \cdot \bar{b}) \times \mathcal{O}(\bar{b} \cdot R)
= \mathcal{O}(n \cdot \bar{b}^2 \cdot R^2),
\]
which is bounded by $\mathcal{O}(n^3 \cdot \bar{b}^2)$ and $\mathcal{O}(n \cdot \bar{b}^2 \cdot \bar{d}^2)$, i.e., pseudo-polynomial.

\paragraph{Implementation}
In practice, the DP can further be improved by memoizing helper function results and pruning infeasible budget/radius combinations early.
A pseudocode summary of the tree DP is given in \suppref{app:pseudocode_tree}.
In Section \ref{sec:experiments}, we report computational results of our DP implementation, and show that it is competitive with an MILP solver especially for larger budget and coverage radius values.

\subsection{Block-cut tree decomposition}
\label{sec:bct}

We now extend the DP framework to graphs that are connected but not biconnected.

\subsubsection{From trees to graphs with articulation points}
\label{sec:bct_motivation}

We now generalize the tree DP from Section~\ref{sec:treedp} to graphs that contain \emph{articulation points} (APs), also known as \emph{cut vertices}: nodes whose removal disconnects the graph.
Trees are the extreme case: every internal node is an articulation point, and the per-block subproblems of this section reduce to the single-node subproblems of the tree DP in Section~\ref{sec:treedp}.
The block-level solutions are combined via the general DP recursion~\eqref{eq:dp_highlevel} on the block-cut tree defined below.

\paragraph{Block-cut tree structure.}
A graph $G = (N, E)$ can be decomposed into its \emph{biconnected components}, called \emph{blocks}:
maximal subgraphs that remain connected after removing any single node.
The blocks are connected to each other only through articulation points,
and this relationship forms a tree structure called the \emph{block-cut tree} (BC-tree),
which can be computed in linear time~\citep{hopcroftAlgorithm447Efficient1973}.

Formally, let $\mathcal{B}$ denote the set of blocks and $\mathcal{A}$ the set of articulation points.
The BC-tree $\mathcal{T}$ has vertex set $\mathcal{B} \cup \mathcal{A}$,
where block $B$ and AP $a$ are adjacent in $\mathcal{T}$ if and only if $a$ belongs to $B$.
We work directly on $\mathcal{T}$; however, DP subproblems are only necessary on the block nodes, since these include their incident articulation points.

We root $\mathcal{T}$ at an arbitrary block $B_r$, which induces parent-child relationships among blocks.
For a non-root block $B$, we denote by $a_B$ its \emph{parent AP} (the unique AP connecting $B$ to its parent block) and by $\mathcal{C}_B$ its set of \emph{child blocks}.
For a child block $C \in \mathcal{C}_B$, its parent AP $a_C$ is also called a \emph{child AP} of $B$, and $B \cap C  = \{a_C \}$.
For the demand accounting, each articulation point $a_C$ is assigned to the parent block of $C$; that is, $h_{a_C}$ is counted in the DP value of the parent (consistent with the sum over $N_B \setminus \{a_B\}$ in the block-level flow formulation).

\subsubsection{Border conditions at articulation points}
\label{sec:bct_border}

A facility in one block can cover nodes in another, but only through the connecting AP; the interaction therefore reduces to the coverage radius remaining at that AP.
The signed-radius DP of Section~\ref{sec:dp_framework} therefore instantiates on the BC-tree by taking blocks instead of single nodes as subproblems: the interface point of a block $B$ is its parent AP $a_B$, and the signed radius $r$ is exactly the coverage radius remaining at $a_B$.
Consequently, $V(B,b,r)$ is the maximum coverage achievable in the subtree of $\mathcal{T}$ rooted at $B$.

A block $B$ interacts with its neighbors through its parent AP $a_B$ and its child APs $\{a_C\}_{C \in \mathcal{C}_B}$.
Recall from Section~\ref{sec:dp_framework} that the \emph{parent radius} $r$ is a fixed parameter of the DP state, while each \emph{child selection} $\lambda^C = (b_C, r_C) \in \Lambda_C$ is a decision variable.
At $a_B$, $r$ is the covering condition imposed on $B$ by its parent block.
At each child AP $a_C$, $\lambda^C$ specifies the budget $b_C$ allocated to $C$ together with $r_C$, the covering condition at $a_C$ that simultaneously serves as $C$'s parent radius (in $V(C, b_C, r_C)$) and constrains $B$'s own coverage flow through $a_C$.
These selections are chosen jointly with the local decisions inside $B$.

The sign of the radius encodes a duality in its effect on $B$:
a positive parent radius ($r>0$) and a negative child radius ($r_C<0$) both represent free external coverage arriving at the respective AP (with remaining radius $r$ or $|r_C|$, respectively),
whereas a negative parent radius ($r<0$) and a positive child radius ($r_C>0$) both impose an obligation on $B$ to provide coverage outward through that AP.

\paragraph{Auxiliary-node interpretation.}
It is useful to represent each signed-radius border condition as an auxiliary node at the articulation point, following the duality above:
a condition providing free external coverage into $B$ (parent $r>0$, child $r_C<0$) is represented by a \emph{virtual facility} at distance $\bar{d}-r$ resp.\ $\bar{d}-|r_C|$ from the AP;
a condition obliging $B$ to cover outward (parent $r<0$, child $r_C>0$) by a \emph{virtual demand node} at distance $|r|$ resp.\ $r_C$, which $B$ must cover;
the BLANK conditions ($r=0$, $r_C=0$) add no node, decoupling the blocks at the AP.
In all cases, the node is placed so that the magnitude of the radius is exactly the coverage distance remaining at the AP.
Since a child radius is signed from $C$'s perspective (it is $C$'s parent radius), the same sign yields opposite node types at the parent and child APs.
This interpretation is illustrated in Figure~\ref{fig:bc_auxiliary_nodes}.
Adding the auxiliary nodes of the active border conditions to the block instance yields the \emph{auxiliary instance} of $B$; Table~\ref{tab:aux_nodes} lists the auxiliary node for each of the four non-BLANK conditions.
This construction directly guides the MILP extensions in Section~\ref{sec:ext_flowmip}.

\begin{table}[htbp]
\centering
\caption{The auxiliary instance: one added node with its connecting arc per non-BLANK border condition (the BLANK cases $r=0$, $r_C=0$ add nothing).
The parent nodes exist only in the DP states of matching sign, since the parent radius $r$ is a fixed parameter; the child nodes are present for every candidate state and activated by the child selection $\lambda^C = (b_C, r_C)$.
For the child OUT condition, a single node $d_C$ per child suffices, its arc carrying the selected radius as its length (Section~\ref{sec:ext_flowmip}).}
\label{tab:aux_nodes}
\small
\begin{adjustbox}{max width=\textwidth}
\begin{tabular}{lllll}
\toprule
Border condition & Auxiliary node & Connecting arc & Length & Exists / active \\
\midrule
parent OUT ($r > 0$)  & virtual facility $w_B$      & $(w_B, a_B)$      & $\bar{d} - r$         & parent states $r>0$ only; free to use \\
child INT ($r_C < 0$) & virtual facility $w_C[r_C]$ & $(w_C[r_C], a_C)$ & $\bar{d} - |r_C|$     & usable iff $\lambda^C$ selects radius $r_C$ \\
\midrule
child OUT ($r_C > 0$) & virtual demand node $d_C$   & $(a_C, d_C)$      & selected radius $r_C$ & covered iff $\lambda^C$ selects $r_C > 0$ \\
parent INT ($r < 0$)  & virtual demand node $d_B$   & $(a_B, d_B)$      & $|r|$                 & parent states $r<0$ only; must be covered \\
\bottomrule
\end{tabular}
\end{adjustbox}
\end{table}

\begin{figure}[ht]
	\centering
	\begin{minipage}[c]{0.55\textwidth}
		\centering
		\includegraphics[width=\linewidth]{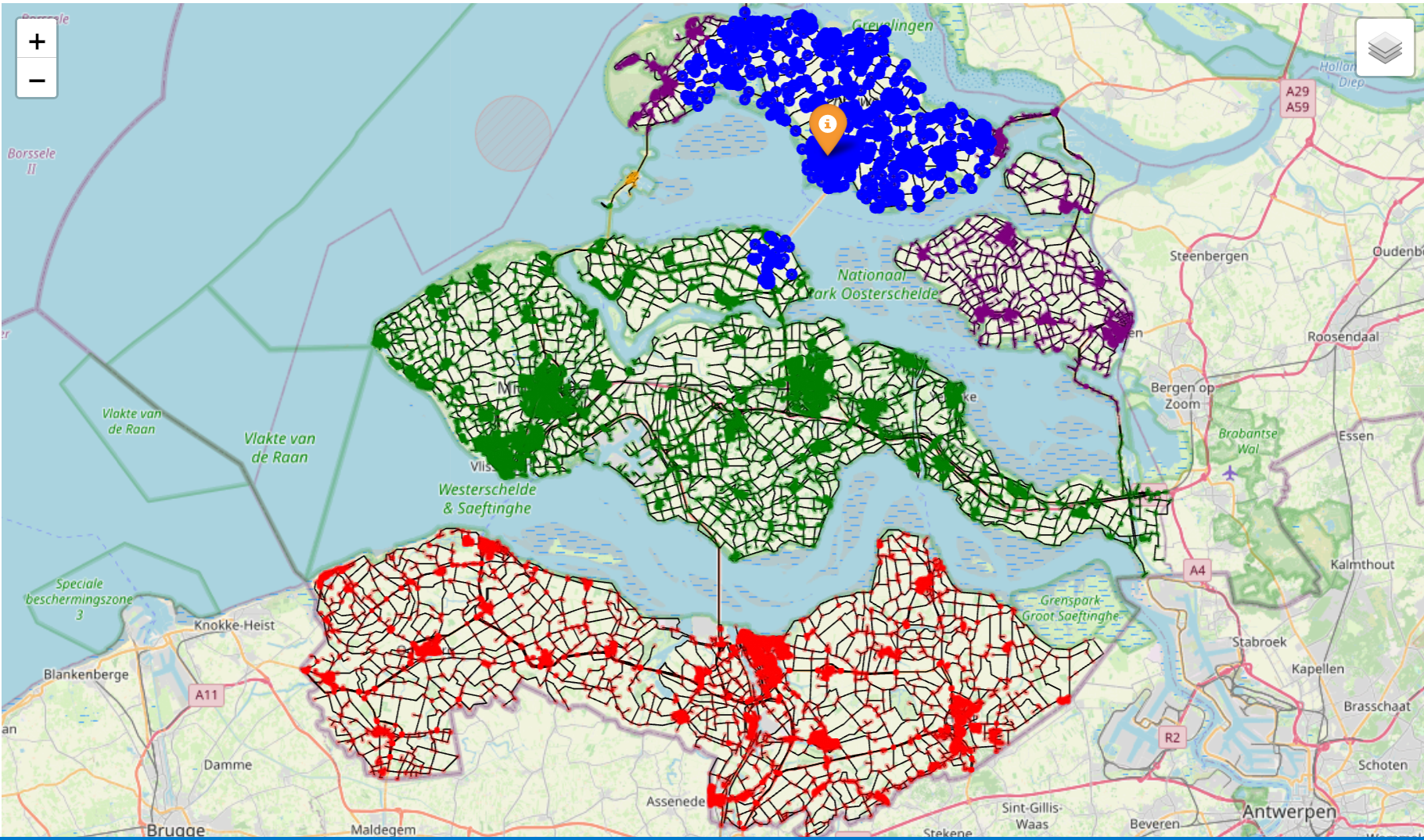}
	\end{minipage}%
	\hfill
	\begin{minipage}[c]{0.4\textwidth}
		\centering
		\adjustbox{max width=\linewidth}{%
		\begin{tikzpicture}[>=stealth, node distance=0.5cm, scale=0.85]
			
			\node[circle, draw, minimum size=1.8cm, inner sep=0pt,
			path picture={
				\node at (path picture bounding box.center){
					\includegraphics[width=2.1cm]{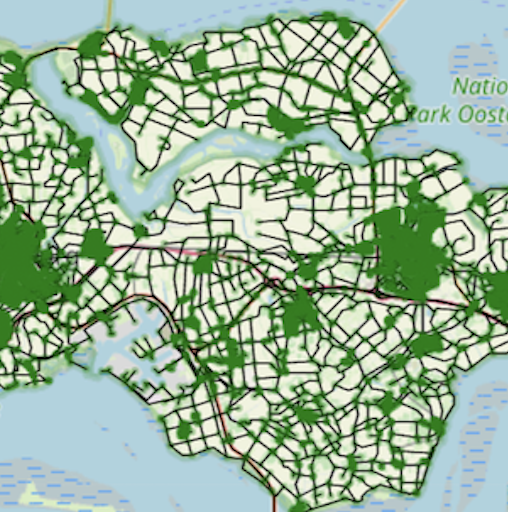}
				};
			}] (center) {};
			
			
			\node[circle, draw=green!40!black, fill=green!30!black, minimum size=0.15cm, inner sep=0pt] (ap_parent) at (center.north) {};
			\node[above=0.02cm of ap_parent, font=\tiny] {$\mathbf{a_B}$};
			
			\node[rectangle, draw=green!70!black, fill=green!60, minimum size=0.35cm,
			above left=0.8cm and 0.25cm of center] (out_parent) {};
			
			\draw[<-] (ap_parent) -- node[pos=0.5, left, font=\tiny] {$\bar{d}{-}r$} (out_parent);
			
			\node[above left=-0.5cm and 0.025cm of out_parent, font=\scriptsize] {$V(B,b,r),\ r>0$};
			
			\node[circle, draw=cyan!70!black, fill=cyan!40, minimum size=0.35cm,
			above right=0.8cm and 0.25cm of center] (int_parent) {};
			
			\draw[->] (ap_parent) -- node[pos=0.5, right, font=\tiny] {$|r|$} (int_parent);
			
			\node[above right=-0.5cm and 0.025cm of int_parent, font=\scriptsize] {$V(B,b,r),\ r<0$};
			
			\node[circle, draw=gray, dashed, fill=gray!20, minimum size=0.35cm,
			above=0.8cm of center] (blank_parent) {};
			
			\node[above=0.03cm of blank_parent, font=\scriptsize] {$V(B,b,0)$};
			
			
			\node[circle, draw=green!40!black, fill=green!30!black, minimum size=0.15cm, inner sep=0pt] (ap_child) at (center.south) {};
			\node[below=0.01cm of ap_child, font=\tiny] {$\mathbf{a_C}$};
			
			\node[circle, draw=red, fill=red!40, minimum size=0.25cm,
			below left=0.6cm and 0.25cm of center] (bleft1) {};
			\node[circle, draw=red, fill=red!65, minimum size=0.25cm,
			below=0.25cm of bleft1] (bleft2) {};
			\node[circle, draw=red, fill=red!90, minimum size=0.25cm,
			below=0.25cm of bleft2] (bleft3) {};
			
			\draw[->] (ap_child) -- (bleft1);
			\draw[->] (bleft1) -- (bleft2);
			\draw[->] (bleft2) -- (bleft3);
			
			\node[left=0.1cm of bleft3, font=\scriptsize, red] {$r_C>0$};

			\node[rectangle, draw=orange, fill=orange!80, minimum size=0.2cm,
			below right=0.75cm and 0.15cm of center] (bsq1) {};
			\node[rectangle, draw=orange, fill=orange!80, minimum size=0.2cm,
			below right=1.25cm and -0.1cm of center] (bsq2) {};
			\node[rectangle, draw=orange, fill=orange!80, minimum size=0.2cm,
			below right=1.8cm and -0.25cm of center] (bsq3) {};
			
			\draw[->] (bsq1) -- (ap_child);
			\draw[->] (bsq2) -- (ap_child);
			\draw[->] (bsq3) -- (ap_child);
			
			\node[right=0.1cm of bsq3, font=\scriptsize, orange] {$r_C<0$};
			
			\node[circle, draw=gray, dashed, fill=gray!20, minimum size=0.25cm]
			(blank_child) at ($(bleft3)!0.5!(bsq3)$) {};
			
			\node[below=0.05cm of blank_child, font=\scriptsize, gray] {$r_C=0$};

		\end{tikzpicture}}
	\end{minipage}
	\caption{Border conditions as auxiliary nodes.
		\emph{Left:} The Dutch province of Zeeland decomposed into three blocks; a facility in the northern block (orange pin, blue coverage area) reaches the central block through the AP, yielding a parent condition $r>0$ there.
		\emph{Right:} The auxiliary nodes at the parent AP $a_B$ (top) and a child AP $a_C$ (bottom), cf.\ Table~\ref{tab:aux_nodes}: a virtual facility for coverage into the block ($r>0$; $r_C<0$), a virtual demand node for an outward obligation ($r<0$; $r_C>0$), and no node for BLANK ($r=0$; $r_C=0$).
		The parent radius $r$ is fixed per DP state, whereas at each child AP the candidate radii $r_C \in \mathcal{R}_C$ appear as alternative nodes, of which the selection $\lambda^C$ activates one.
		Map data from OpenStreetMap, available under the Open Database License; see \url{https://www.openstreetmap.org/copyright}.
	}
	\label{fig:bc_auxiliary_nodes}
\end{figure}
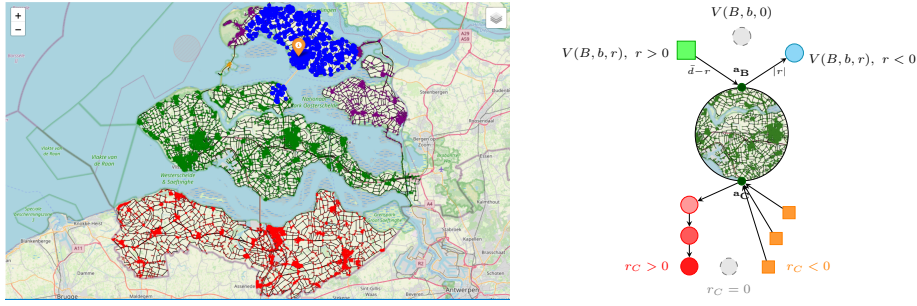

\subsubsection{Computing block DP values via MILP}
\label{sec:bct_mip}

The DP recursion~\eqref{eq:dp_highlevel} requires, for each block $B$, optimizing over the local decisions inside $B$ jointly with the choice of budget and covering condition for every child block.
Since the number of child state combinations grows exponentially in the number of children, explicit enumeration and computation would be brute force and unlikely to perform well.
Instead, we formulate the recursion~\eqref{eq:dp_highlevel} as a MILP that incorporates the unified formulation~\eqref{eq:unified} for block $B$, extended with \emph{selector variables} that pick a precomputed child state per child block and with a parameterized feasibility set $\mathcal{F}$ encoding the border conditions at the articulation points.

\paragraph{Block-level MILP formulation.}
Let $N_B$, $E_B$, and $\overline{E}_B = E_B \cup \{\{s,j\} \mid j \in F_B\}$ denote the nodes, edges, and extended edges (including facility edges) of block $B$.
For each child $C \in \mathcal{C}_B$, we introduce \emph{selector variables} $\lambda^C_{b',r'} \in \{0,1\}$ indexed over the (pruned) state set $\Lambda_C \subseteq \{0,\ldots,\bar{b}\} \times \mathcal{R}_C$ from Section~\ref{sec:dp_framework}.
Setting $\lambda^C_{b',r'} = 1$ selects the precomputed DP value $V(C, b', r')$.

The block-level MILP for computing $V(B, b, r)$ is:
{
\begin{align}
	\max \quad & \sum_{k \in N_B \setminus \{a_B\}} h_k z_k \;+\;
	\sum_{C \in \mathcal{C}_B} \sum_{(b', r') \in \Lambda_C} \lambda^C_{b',r'} \cdot V(C, b', r')
	\label{eq:bct_mip_obj2} \\[2pt]
	\text{s.t.} \quad & \sum_{e \in \overline{E}_B} c^E_e x_e \;+\;
	\sum_{C \in \mathcal{C}_B} \sum_{(b', r') \in \Lambda_C} b' \cdot \lambda^C_{b',r'}
	\;\leq\; b \label{eq:bct_mip_budget2} \\
	& \sum_{(b', r') \in \Lambda_C} \lambda^C_{b',r'}
	\;=\; 1
	\hspace{4em} \forall\, C \in \mathcal{C}_B \label{eq:bct_mip_select2} \\
	& (x, z)
	\;\in\; \mathcal{F}(r, \{\lambda^C\}_{C \in \mathcal{C}_B})
	\label{eq:bct_mip_feas2} \\
	& x_e
	\;\in\; \{0,1\}
	\hspace{4em} \forall\, e \in \overline{E}_B,\\
	& z_k
	\;\in\; \{0,1\}
	\hspace{4em} \forall\, k \in N_B \\
	& \lambda^C_{b',r'}
	\;\in\; \{0,1\}
	\hspace{4em} \forall\, C \in \mathcal{C}_B,\, (b',r') \in \Lambda_C \label{eq:lambda_binary}
\end{align}
}

The objective \eqref{eq:bct_mip_obj2} maximizes coverage inside $B$ (excluding the parent AP $a_B$, whose demand is accounted for in the parent block) plus the precomputed DP values of the selected child states.
Constraint \eqref{eq:bct_mip_budget2} ensures the total budget (block plus children) does not exceed $b$.
Constraint \eqref{eq:bct_mip_select2} ensures exactly one state is selected per child block.
The feasibility set $\mathcal{F}(r, \{\lambda^C\})$ in \eqref{eq:bct_mip_feas2} extends~\eqref{eq:unified} by encoding the parent radius $r$ and the selected child radii at the articulation points; we describe its implementation in Section \ref{sec:ext_flowmip}.

\begin{remark}[Selector subproblem is small and fast to solve]
\label{par:child_alloc_assumption}
Fixing the local design to incumbent values $(\hat{x}, \hat{z})$ collapses the border-condition encoding into a feasibility filter on the child states, excluding those incompatible with the chosen design and typically leaving only a handful of candidates per child.
The residual problem in $\lambda$ is essentially
\begin{equation}
	\label{eq:selector_subproblem}
	\max_{\lambda} \;\Bigl( h_B(\hat{x}, \lambda) \;+\; \sum_{C \in \mathcal{C}_B} l(C, \lambda^C) \Bigr),
\end{equation}
where $h_B(\hat{x}, \lambda)$ is the local coverage of the incumbent at child allocation $\lambda$ and $l(C, \lambda^C)$ a lower bound on $V(C, \lambda^C)$ (Section~\ref{sec:propagating_bounds}).
This is a \emph{multiple-choice knapsack problem} (MCKP) \citep{martelloKnapsackProblemsAlgorithms1990} with at most $\bar{d}\cdot\bar{b}$ candidate states per child, hence small and fast to solve compared with the network-design part.
\end{remark}

Preliminary experiments support this: across a sample of block solves, re-solving the MCKP alone from the incumbent's $(\hat{x}, \hat{z})$ recovered that incumbent's $\lambda$ in well under a second whenever the block MILP had stopped at its time limit.
We adopt this as an assumption in our bound analysis (Section~\ref{sec:propagating_bounds}; details in \suppref{app:bound_propagation}).


\subsubsection{Extending the flow formulation}
\label{sec:ext_flowmip}

We realize the feasibility set $\mathcal{F}(r, \{\lambda^C\})$ in~\eqref{eq:bct_mip_feas2} by solving the block MILP on the auxiliary instance of Section~\ref{sec:bct_border} (Table~\ref{tab:aux_nodes}).
Since a node is covered if and only if a coverage path of length at most $\bar{d}$ connects it to an opened facility via built edges (Section~\ref{sec:flowmip}), the auxiliary nodes extend exactly the two sides of this characterization: virtual facilities are further origins of coverage paths, virtual demand nodes further destinations that must be reached.
The covering constraints therefore apply unchanged on the auxiliary instance; this holds for the flow formulation as for any other block-level formulation stated on the block instance.

What remains is to couple the auxiliary nodes to the DP states.
The parent condition is fixed per block MILP, so its node is simply present or absent; only the child choice needs encoding.
The MILP carries the auxiliary nodes of \emph{all} candidate child states $(b', r') \in \Lambda_C$, and linking constraints tie them to the selectors $\lambda^C_{b',r'}$: a virtual facility is usable only if its state is selected, and a virtual demand node must be covered exactly if an OUT state of its child is selected.
In the flow formulation, these linking constraints amount to a handful of extra flow variables and variable upper bounds.

Finally, the flow formulation spares us from constructing the auxiliary instance in its entirety.
Conceptually, each candidate OUT radius of a child contributes its own virtual demand node; since exactly one state is selected per child, they merge into the single node $d_C$, the length of its connecting arc becoming a linear expression in the selectors $\lambda^C$ that evaluates to the selected radius.
The block MILP therefore grows with the number of children rather than with the number of admissible child radii.
\suppref{app:flow_border} makes this construction precise and derives the resulting master MILP.

\subsubsection{Propagating bounds}
\label{sec:propagating_bounds}

Under time limits, not every block MILP is solved to optimality.
For each block~$B$ and state $(b, r)$ we therefore maintain bounds $l(B,b,r) \le V(B,b,r) \le u(B,b,r)$ with gap $\Delta(B,b,r) := u(B,b,r) - l(B,b,r)$, and propagate them up the block-cut tree through the recurrence~\eqref{eq:dp_highlevel}.

The absolute gap necessarily grows up the tree, since each block's bound aggregates those of its children and the subproblem itself grows with the subtree.
The gap \emph{relative} to the subproblem size, however, does not accumulate.
Assuming the selector subproblem~\eqref{eq:selector_subproblem} is solved exactly (Remark~\ref{par:child_alloc_assumption}), an inductive argument shows that if each block returns a primal design within factor $\alpha_B \in (0, 1]$ of its local optimum, the global approximation ratio is at least $\min_B \alpha_B$:
the relative error is dictated by the weakest local solve, independent of tree depth. Formal statements (Propositions~\ref{prop:additive_gap}, \ref{prop:relative_gap}, Corollary~\ref{cor:global_ratio}) and proofs are in \suppref{app:bound_propagation}.

\paragraph{Bounds from neighboring states.}
Monotonicity (Section~\ref{sec:dp_framework}) gives a second, cheap source of bounds. When only the BLANK state $V(B, b, 0)$ has been solved, the OUT state at radius $r > 0$ exceeds the BLANK by at most the demand reachable from the parent AP~$a_B$ within distance~$r$,
\begin{equation*}
	V(B, b, r) \;\le\; V(B, b, 0) + w^B(a_B, r),
\end{equation*}
where $w^B(a_B, r)$ denotes that demand weight. The INT bound for $r < 0$ is analogous and given in \suppref{app:bound_propagation}.

A pseudocode summary of the BC-tree DP, including the warm-starting, dominance pruning, and bound-propagation refinements introduced above, is given in \suppref{app:pseudocode_bct}.

\section{Computational experiments}
\label{sec:experiments}

We evaluate the proposed decomposition approach in two settings of increasing complexity.
First, on tree graphs where the DP solves the MCNDP exactly in pseudo-polynomial time (Section~\ref{sec:results_trees}),
we compare against the assignment-based MILP to validate the approach in the simplest setting.
Second, on general graphs with articulation points (Section~\ref{sec:results_bct}),
we investigate whether embedding a MILP formulation within the BC-tree DP framework improves performance
compared to solving the MILP on the full graph.
We focus on the flow formulation (Section~\ref{sec:flowmip}), solved both directly and within the DP.

\subsection{Experimental setup}
\label{sec:setup}

The experiments were performed in a Linux computing environment (64-bit).
The tree DP experiments ran on Intel Xeon Platinum 8360Y CPUs at 2.40\,GHz,
the BC-tree experiments on Intel Xeon Gold 6240 CPUs at 2.60\,GHz.
All runs are single-threaded (Gurobi threads=1) with a 16\,GB solver memory limit and a 3600-second time limit.
The algorithms are implemented in Python (3.11 for the tree DP, 3.13 for the BC-tree experiments)
using NetworkX~3.5 for graph algorithms
and Gurobi~12.0 as the MILP solver~\citep{gurobi}.
The resources and services used in this work were provided by the VSC (Flemish Supercomputer Center),
funded by the Research Foundation Flanders (FWO) and the Flemish Government.

\subsection{Instances}
\label{sec:instances}

There are no standard benchmark instances for network design problems on graphs with low connectivity.
We describe the instance classes for the tree DP and BC-tree DP experiments separately.

\paragraph{Tree instances.}
For the tree DP experiments, we sample trees on $n$ nodes uniformly at random among all labeled trees with the \texttt{random\_labeled\_tree} function of NetworkX~\citep{SciPyProceedings_11}.
Edge distances and upgrade costs are drawn independently from $\{1, \ldots, 10\}$,
facility opening costs from $\{3, 6, \ldots, 30\}$,
and demand weights from $\{1, \ldots, 10\}$.

\paragraph{BC-tree instances.}
The Sevilla network~\cite{garcia-archillaGRASPAlgorithmsRobust2013} (49 nodes, 119 edges)
and Sioux Falls network~\citep{leblancAlgorithmDiscreteNetwork1975} (24 nodes, 76 edges)
are standard transportation benchmarks but are biconnected,
so the BC-tree decomposition yields a single block with no computational benefit.
To obtain instances with controlled block-cut tree structure,
we construct synthetic graphs by merging biconnected building blocks:
\begin{enumerate}
	\item \textbf{Select base blocks:}
	full copies or connected subgraphs of Sioux Falls,
	Sevilla, and grid graphs ($4{\times}4$, $5{\times}5$, $6{\times}6$).

	\item \textbf{Sample a tree topology $\Theta$ on $k$ blocks}
	(star, path, balanced tree, or random tree),
	where each node represents a block taken from base blocks and each edge indicates a shared articulation point.

	\item \textbf{Merge via node identification:}
	for each edge in $\Theta$,
	identify (i.e.,~shrink) one randomly chosen node from each adjacent block, creating an articulation point.
\end{enumerate}
Table~\ref{tab:instances} in \suppref{sec:appendix_instances} summarizes the resulting 34 instances
(51--337 nodes, 76--833 edges, 2--34 blocks).
Figure~\ref{fig:instance_examples} illustrates a representative instance.

\begin{figure}[htbp]
	\centering
	\includegraphics[width=0.95\textwidth]{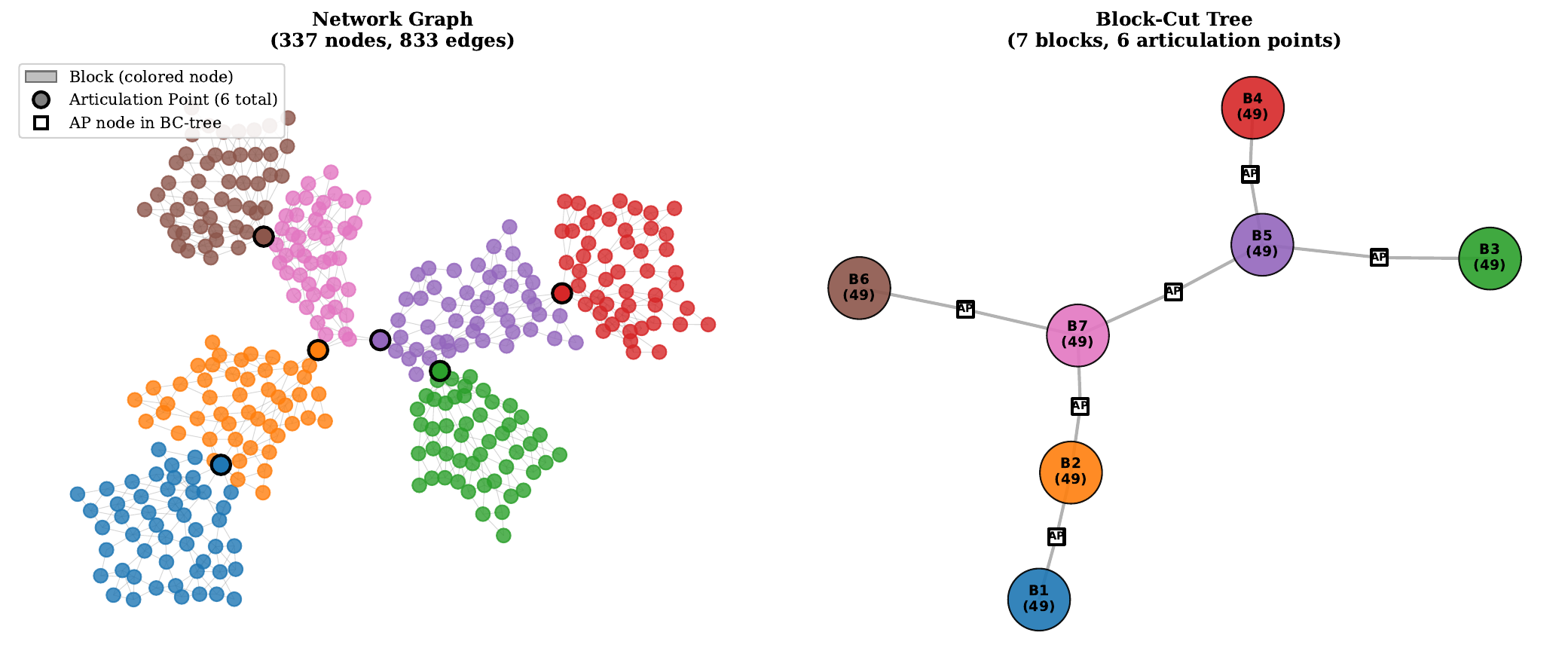}
	\caption{Example BC-tree instance with balanced tree topology: 7 blocks of 49 nodes each, connected in a path-like structure (Max AP$^\circ = 2$).
		The network graph (left) and block-cut tree structure (right) are shown, with block nodes colored consistently.}
	\label{fig:instance_examples}
\end{figure}

\paragraph{Instance parameter generation.}
Edge distances are uniform (all equal to~1).
Demand weights are drawn from $\{0, 1, 2, 3, 4\}$ with a fixed seed.
Facility opening costs are drawn from $\{5, 10, 15\}$,
and edge upgrade costs from $\{1, 2, 3, 4, 5\}$.
The budget $\bar{b} = \beta \cdot b_{\mathrm{ref}}$ is a fraction of a reference budget estimated to cover the entire graph,
with budget factor $\beta \in \{0.3, 0.5, 0.7\}$.
The coverage radius is
\begin{equation}\label{eq:coverage_radius}
	\bar{d} = \lfloor \gamma \cdot m \cdot \bar{\ell} \rfloor,
\end{equation}
where $m$ is the number of edges,
$\bar{\ell}$ the average edge length,
and $\gamma \in \{0.05, 0.10, 0.15\}$ the coverage radius factor.
Intuitively, this is the fraction of the network's total edge length
that a single facility-to-demand path may traverse.
This yields coverage radii ranging from 3 to 124 across instances.

\subsection{Results on trees}
\label{sec:results_trees}

We evaluate the tree DP algorithm from Section~\ref{sec:treedp} against solving the MILP formulation directly with Gurobi.
For each $(\bar{b}, \bar{d}) \in \{10, 20, \ldots, 60\}^2$ we generate 9 random 20-node trees (324 test cases), and additionally study scaling at $n \in \{20, 40, \ldots, 160\}$ for three fixed parameter settings, with eight trees per size (192 further test cases).
Both algorithms are exact, and on all 516 test cases they return the same objective value, which confirms the correctness of the DP implementation.

\begin{figure}[htbp]
	\centering
	\begin{subfigure}[b]{0.48\textwidth}
		\centering
		\includegraphics[width=\textwidth]{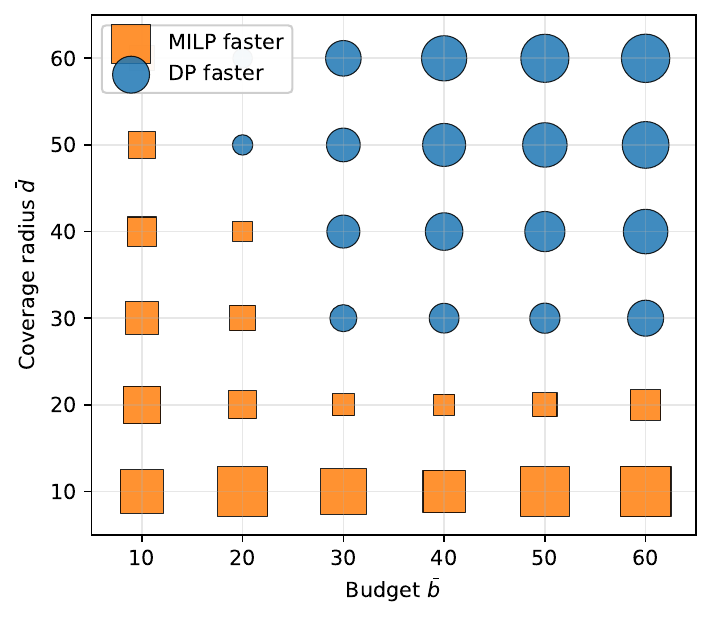}
		\caption{Faster algorithm by $(\bar{b}, \bar{d})$ on 20-node trees; marker area scales with $|\log_2(t_{\mathrm{MILP}}/t_{\mathrm{DP}})|$.}
		\label{fig:tree_dp_faster}
	\end{subfigure}
	\hfill
	\begin{subfigure}[b]{0.48\textwidth}
		\centering
		\includegraphics[width=\textwidth]{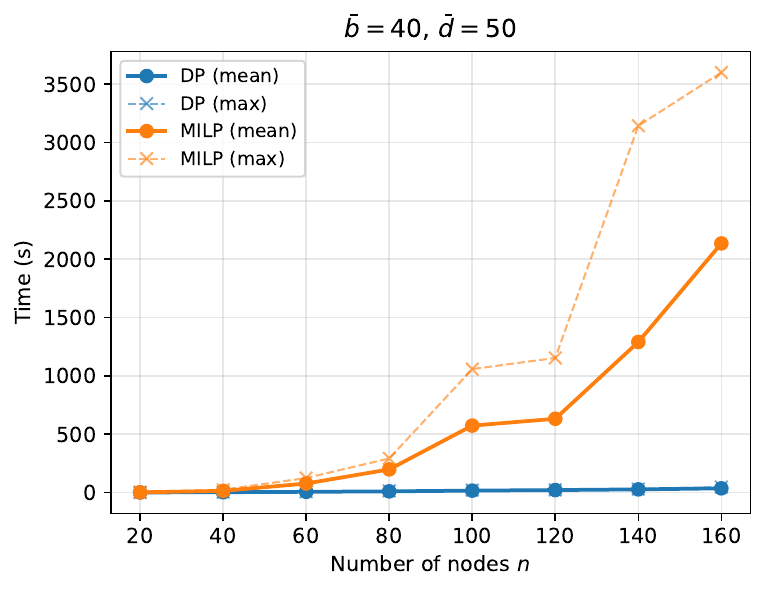}
		\caption{Scaling at $\bar{b}=40,\bar{d}=50$: at $n = 160$ the MILP averages over 2100\,s, with one of the eight instances unsolved at the 3600\,s limit, while the DP completes in about 36\,s (max 46\,s).}
		\label{fig:tree_dp_n_high}
	\end{subfigure}
	\caption{Tree DP vs.\ direct MILP. The DP dominates once both $\bar{b}\geq 30$ and $\bar{d}\geq 40$ (left), and the advantage widens with instance size (right).}
	\label{fig:tree_dp}
\end{figure}

For moderate parameters the MILP is faster, though both algorithms finish quickly there: over the eighteen parameter combinations that the MILP wins, it averages $0.11$\,s against the DP's $0.21$\,s, and no solve exceeds $1.5$\,s.
As $\bar{b}$ and $\bar{d}$ grow, the DP's pseudo-polynomial $\mathcal{O}(n \bar{b}^2 \bar{d}^2)$ runtime becomes competitive: once $\bar{b} \geq 30$ and $\bar{d} \geq 40$, the DP is consistently faster (Figure~\ref{fig:tree_dp_faster}), by a median factor of $3.0$.
The margin widens with instance size (Figure~\ref{fig:tree_dp_n_high}): at $\bar{b}=40,\, \bar{d}=50$ it reaches a median factor of $66$ at $n=160$, where one MILP run hit the $3600$\,s limit unsolved on a tree that the DP had settled in $46$\,s.
The MILP also exhibits high variance across instances, with a coefficient of variation of $0.55$ over the nine trees of a parameter combination against the DP's $0.21$.
This predictability is a practical advantage beyond raw speed.
Detailed sensitivity at $n=20$ and the two lower-parameter scaling settings ($\bar{b}=20,\, \bar{d}=40$ and $\bar{b}=30,\, \bar{d}=30$) appear in \extraref{sec:additional_tree_results}.
These findings motivate the extension to graphs with low connectivity via the block-cut tree decomposition (Section~\ref{sec:bct}), where the DP is applied recursively to biconnected components.

\subsection{Results on graphs with articulation points}
\label{sec:results_bct}

We evaluate the BC-tree DP framework on graphs with articulation points, where it coordinates MILP solves on individual blocks, instead of solving a single MILP on the full graph.

\subsubsection{Implementation choices}
\label{sec:bct_prelim}

Several implementation choices affect the performance of the BC-tree DP framework.
Compared with solving only the direct MILP, the decomposition offers more opportunities
for performance improvements through careful design.
We summarize the key decisions based on preliminary experiments
and refer to \suppref{app:implementation} for details.

\paragraph{Warm-starting and state ordering.}
Each block MILP is initialized with a greedy heuristic solution.
We further exploit monotonicity across neighboring DP states: 
a solution for state $V(B,b, r)$ provides a valid warm start for $V(B,b+1, r)$ and $V(B,b, r+1)$.
States are processed with budget increasing in the outer loop and radius in the inner loop.

\paragraph{Structural choices.}
The BC-tree is rooted at the largest block, which is solved only for the BLANK condition with full budget $\bar{b}$.
Adjacent blocks with at most three vertices are merged to reduce DP overhead.
We also filter clearly infeasible budget-distance combinations using combinatorial properties of edge costs and path lengths.

\paragraph{Time allocation.}
When a total time limit is imposed, time is distributed across blocks proportionally to edge count,
then allocated adaptively across DP states within each block so that unused time from faster solves accumulates for harder states.
States receiving insufficient time are skipped, with bounds constructed from dominating states and heuristics (Section~\ref{sec:propagating_bounds}).

\subsubsection{Computational comparison}
\label{sec:bct_comparison}

We evaluate whether the BC-tree decomposition yields computational benefits compared to solving a single MILP on the full graph.
We use the 34 multi-block instances from Table~\ref{tab:instances},
each solved with three budget factors ($\beta \in \{0.3, 0.5, 0.7\}$) and three coverage radius factors ($\gamma \in \{0.05, 0.10, 0.15\}$),
yielding $34 \times 3 \times 3 = 306$ test cases.
All algorithms use Gurobi 12.0.3 with a one-hour time limit.

\paragraph{Algorithm selection via performance profiles.}
To select the MILP formulation presented here, we compared two MILP formulations for the MCNDP subproblems: the flow-based model (Section~\ref{sec:flowmip}) and an alternative based on length-bounded cuts \citep{arslanFlexibleNaturalFormulation2019}.

We compare four configurations: the direct MILP with warm-starting using either formulation (FlowMIP, LBCutMIP), and the BC-tree DP with either formulation (DP-FlowMIP, DP-LBCutMIP).
We use performance profiles \citep{dolanBenchmarkingOptimizationSoftware2002} on solution quality to compare the configurations: for each algorithm, the profile reports the fraction of instances solved within a factor~$\tau$ of the best objective found by any algorithm.
Figure~\ref{fig:performance_profile} shows the result.

\begin{figure}[htbp]
	\centering
	\includegraphics[width=0.48\textwidth]{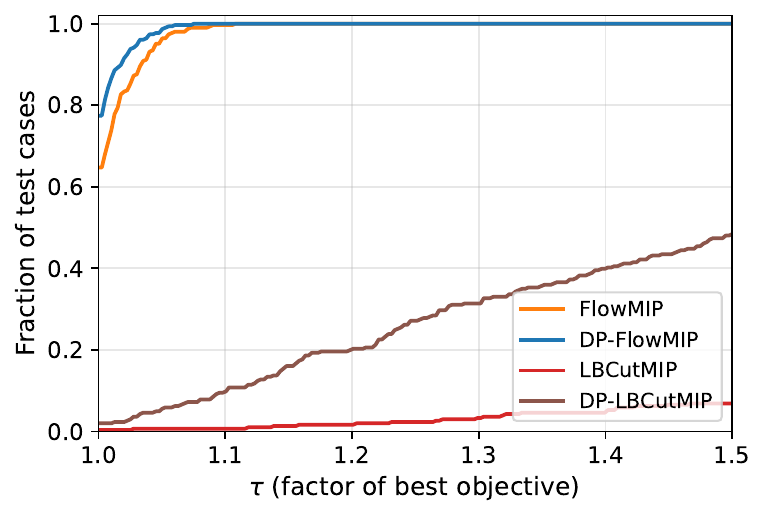}
	\caption{Solution quality performance profile.
		At $\tau = 1$: DP-FlowMIP achieves best in 75.2\%, FlowMIP
		in 62.7\%.}
	\label{fig:performance_profile}
\end{figure}

FlowMIP and DP-FlowMIP clearly dominate and will be the focus in the following: at $\tau = 1.05$, both solve nearly all instances well (96.4\% and 98.7\%).
The length-bounded cut formulation is clearly outperformed by the flow formulation; nonetheless, embedding it in the BC-tree DP (DP-LBCutMIP) still improves over the direct LBCutMIP, confirming the value of the decomposition.

\paragraph{Overall performance.}
Across all 306 test cases, the DP finds better solutions in 98 cases (32.0\%), the MILP in 57 (18.6\%), with 151 equivalent (49.3\%).
The mean objective difference is 0.43\% in favor of the DP.

In solution quality, the DP shows robustness: when it is worse than FlowMIP, it remains within 2\% in most instances,
exceeding this threshold in only 26 cases compared to 51 for FlowMIP (see Figure~\ref{fig:bct_overview}(a)).
FlowMIP solves 26.1\% of instances to optimality versus 17.6\% for DP-FlowMIP
(see also the time-based performance profile in \extraref{sec:additional_bct_results}).
In the 54 cases where both are optimal, FlowMIP is faster on average (307 seconds versus 521 seconds for the DP).
However, in the 252 instances not solved to optimality, the DP uses an average of 2983 seconds,
leaving approximately 617 seconds of the 3600-second limit unused.

We compared performance across several metrics (BC-tree topology, largest block size, network size) but found no strong differentiating factors; the DP's solution quality advantage is broadly consistent, suggesting it stems from warm-starting and budget allocation rather than specific structural properties.
A more nuanced picture emerges from the sensitivity to budget and coverage radius parameters, discussed next.

\paragraph{Sensitivity to budget and coverage radius.}
\begin{table}[htbp]
	\centering
	\small
	\begin{tabular}{crrrr}
		\toprule
		$\beta$ & DP wins & FlowMIP wins & Equiv 
		  & Mean DP adv.\ (\%) \\
		\midrule
		0.3 & 26 & 28 & 48 & $-0.13$ \\
		0.5 & 35 & 15 & 52 & $+0.82$ \\
		0.7 & 37 & 14 & 51 & $+0.59$ \\
		\midrule
		$\gamma$ & DP wins & FlowMIP wins & Equiv
		  & Mean DP adv.\ (\%) \\
		\midrule
		0.05 & 30 & 26 & 46 & $+0.03$ \\
		0.10 & 36 & 16 & 50 & $+0.68$ \\
		0.15 & 32 & 15 & 55 & $+0.58$ \\
		\bottomrule
	\end{tabular}
	\caption{Algorithm comparison by budget factor $\beta$ (top) and 
		coverage radius factor $\gamma$ (bottom).
		Larger values of both parameters favor the BC-tree DP.}
	\label{tab:sensitivity_params}
\end{table}

At tight budget ($\beta = 0.3$), the MILP and DP are roughly tied, while at $\beta \geq 0.5$ the DP shows a consistent advantage of ${\sim}0.7\%$ (Table~\ref{tab:sensitivity_params}).
Similarly, higher coverage radius favors the DP (32 vs.\ 15 wins at $\gamma = 0.15$), consistent with the tree results from Section~\ref{sec:results_trees}:
longer radii mean more paths through articulation points, making border condition aggregation more effective.
Figure~\ref{fig:bct_overview}(b) shows the same pattern as a heatmap over $\beta \times \gamma$.

\begin{figure}[htbp]
	\centering
	\begin{subfigure}[b]{0.48\textwidth}
		\centering
		\includegraphics[width=\textwidth]{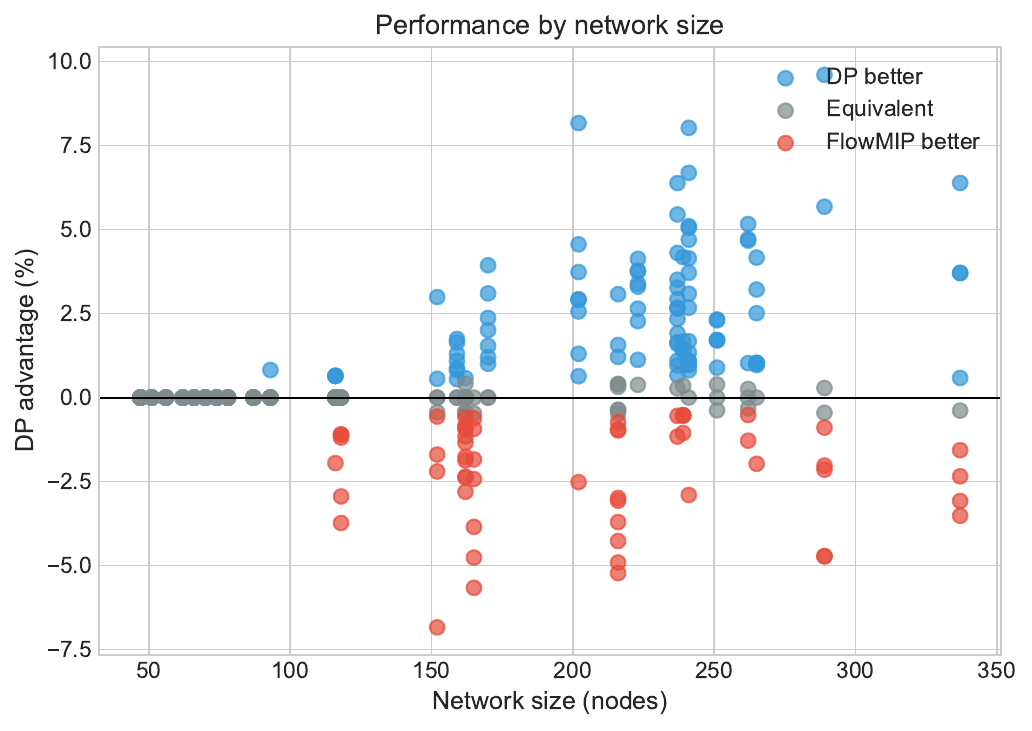}
	\end{subfigure}
	\hfill
	\begin{subfigure}[b]{0.48\textwidth}
		\centering
		\includegraphics[width=\textwidth]{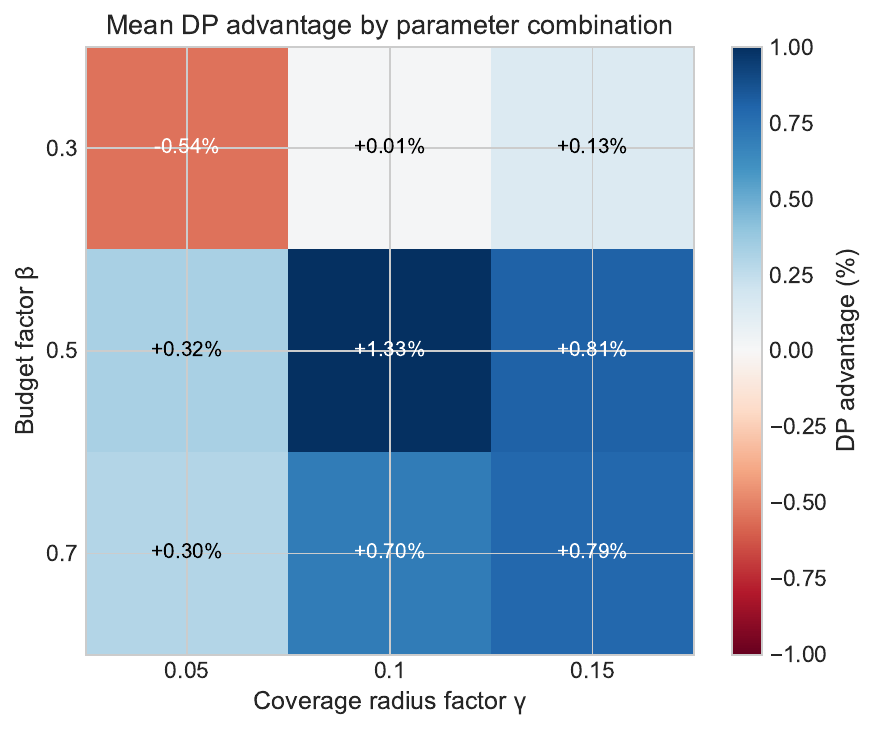}
	\end{subfigure}
	\caption{(a) DP advantage over FlowMIP by network size: red indicates DP advantage $>$0.5\%, blue FlowMIP advantage, gray equivalence.
		(b) Mean DP advantage by budget factor and coverage radius factor.}
	\label{fig:bct_overview}
\end{figure}

\paragraph{Practical solve scenarios.}
Two settings common in practice highlight the BC-tree DP's structural advantages.
We evaluate both on the five largest instances that did not reach optimality within one hour, since these are the cases where solver choice actually matters; full experimental details are reported in \suppref{sec:appendix_practical_solve}.

First, decision-makers often need a full coverage-versus-budget curve rather than a solution for a single budget value.
Because the BC-tree DP enumerates budget states intrinsically, all curve points emerge from a single solve, so each budget value benefits from the full time budget.
FlowMIP, by contrast, must be re-solved for each budget level, and the available time has to be split across points.
We compare three time-allocation strategies that distribute time proportionally to the logarithm (\emph{FlowMIP-LogTime}), linear value (\emph{FlowMIP-LinTime}), or square (\emph{FlowMIP-QuadTime}) of the budget.
Figure~\ref{fig:budget_and_time_sensitivity}(a) shows the resulting curves at $\gamma = 0.10$: the BC-tree DP dominates the FlowMIP variants at every budget value, with QuadTime degrading sharply at small budgets, where the time allotted per point is exhausted before a useful incumbent is found.

Due to its advantage on small budgets, FlowMIP-LogTime is the most competitive allocation strategy, yet is still outperformed by the DP by 1.9\% on average across all budget levels.
Restricting the analysis to the practically more relevant upper half of the budget range, a more consistent advantage emerges: the DP achieves mean relative gains of 4.5--6.7\% against all three strategies (6.7\% vs.\ LogTime, 4.5\% vs.\ LinTime, 5.4\% vs.\ QuadTime).

Second, decision-making scenarios differ widely in how much computation time they allow, from minutes for time-sensitive choices to hours for long-term planning.
We therefore evaluate solution quality across a range of total time budgets, rather than only at a single time limit.
Figure~\ref{fig:budget_and_time_sensitivity}(b) reports this \emph{time-limit sensitivity} at the default budget.
The BC-tree DP delivers better solutions than FlowMIP across nearly the entire range, with one narrow sweet spot at very small total runtimes: there, the MILP subproblems are small enough that FlowMIP finds a decent first incumbent before the DP has amortized its state-enumeration overhead.
Beyond roughly ten minutes, the DP overtakes and the quality gap widens steadily.

\begin{figure}[htbp]
	\centering
	\begin{subfigure}[b]{0.48\textwidth}
		\centering
		\includegraphics[width=\textwidth]{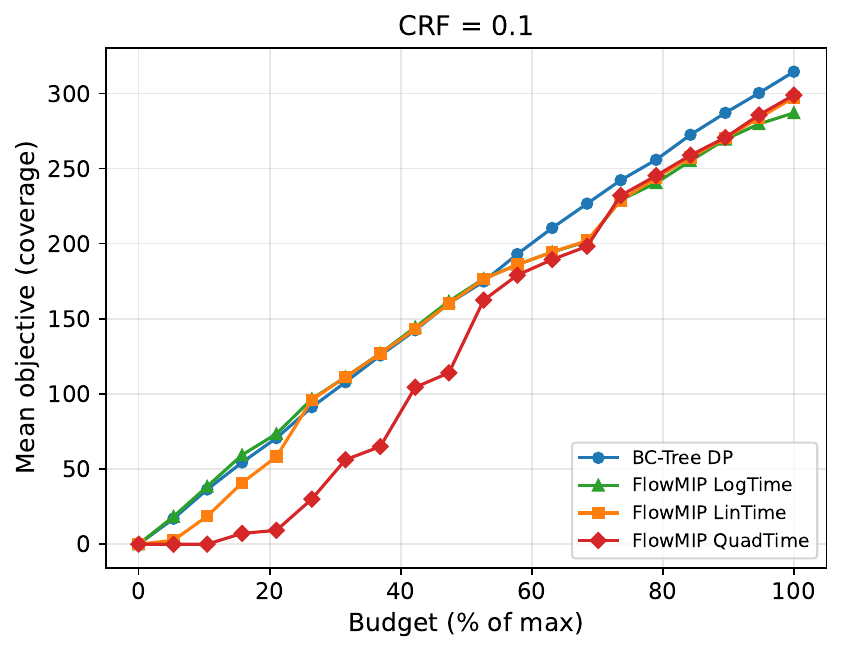}
		\caption{Coverage--budget curve at $\gamma = 0.10$.}
		\label{fig:budget_curve}
	\end{subfigure}
	\hfill
	\begin{subfigure}[b]{0.48\textwidth}
		\centering
		\includegraphics[width=\textwidth]{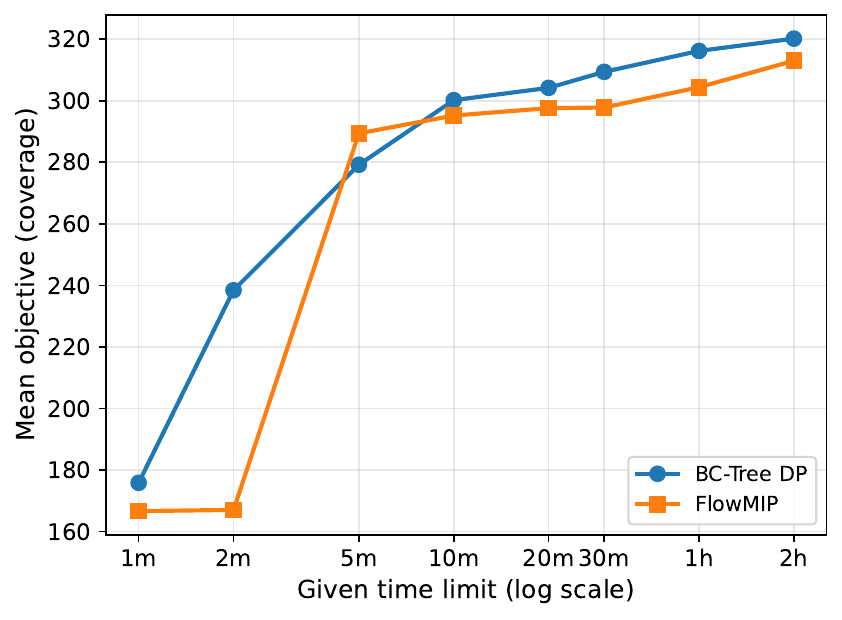}
		\caption{Time-limit sensitivity at the default budget.}
		\label{fig:time_sensitivity}
	\end{subfigure}
	\caption{Two practical solve scenarios on the five largest instances not solved to optimality within one hour.
		The BC-tree DP enumerates all budgets within a single solve, whereas the FlowMIP variants split the time limit across budget points by a log-, linear-, or quadratic-in-budget allocation.}
	\label{fig:budget_and_time_sensitivity}
\end{figure}


\section{Conclusions}
\label{sec:conclusions}

We introduced a decomposition framework for the Maximum Covering Network Design Problem on graphs with low connectivity, centered on a block-cut tree dynamic program that exploits articulation points to decompose the problem into independent subproblems on biconnected components.

For tree graphs the resulting pseudo-polynomial DP solves the MCNDP exactly and outperforms direct MILP formulations on larger budgets and coverage radii, with the added advantage of consistent runtimes without outliers.
For graphs with articulation points, the BC-tree DP coordinates MILP solves on the individual blocks and matches or outperforms direct MILP on over 80\% of our test instances, with an advantage that grows with the budget and coverage radius factors;
where it does fall behind, it typically stays within 2\% of the direct MILP, exceeding that threshold in only 26 cases versus 51 for FlowMIP.
It is particularly well suited to two practical settings, producing full coverage--budget curves at essentially no extra cost over a single solve and dominating FlowMIP across a wide range of total time limits.
Its key practical strength is reducing problem sizes through decomposition, and further improvements in bounding and state iteration could yield larger advantages.

\paragraph{Limitations.}
On 2-connected graphs without articulation points the BC-tree method reduces to a single-block solve, equivalent to direct MILP with neither advantage nor disadvantage. On less connected instances the per-instance gain may be modest, but the computational study shows that the BC-tree DP is preferable in the majority of cases.

\paragraph{Future directions.}
The framework extends along two axes.
First, it is not specific to the MCNDP: the same selector-and-border-condition scheme applies to other location and network design problems in which the interaction across an articulation point reduces to a facility's remaining reach and a budget allocation, such as uncapacitated facility location and its network design variant (the UFLNDP; Section~\ref{sec:lit_review}).
Second, it extends to richer graph structure: SPQR-tree decompositions would lift it to 2-connected graphs, and standard graph partitioning could handle larger separators~\citep{dellingGraphPartitioningNatural2011}.
Algorithmically, block subproblems are naturally multi-objective in budget, coverage, and border conditions, suggesting connections to Lagrangian relaxation~\citep{larssonLagrangianBoundingHeuristic2024} or resource-directed decomposition~\citep{bodurDecompositionLooselyCoupled2022,yildizDecompositionBranchingMixed2022}.
Moreover, a Benders-like scheme could compute child states on demand, and sibling-block independence makes parallelization straightforward, yielding significant performance improvement potential on modern multi-core hardware.
Finally, an instance-space analysis~\citep{smith-milesInstanceSpaceAnalysis2023} could characterize when BC-tree decomposition pays off in practice and guide method selection.

	\paragraph*{Code and data availability.}
	The implementation, the instance generator and instance files, and the per-case result data
	referenced in the appendix are available in the authors' public GitHub repository.

	\paragraph*{Acknowledgments.}
	This work was supported by the special research fund of KU Leuven (project C14/22/026).
	The resources and services used in this work were provided by the VSC (Flemish Supercomputer Center), funded by the Research Foundation - Flanders (FWO) and the Flemish Government.

	\begingroup
	\sloppy
	\bibliographystyle{informs2014}
	\bibliography{ResilientMaxCovering}
	\endgroup

\appendix
\numberwithin{figure}{section}
\numberwithin{table}{section}

\section{Notation reference}
\label{sec:appendix_notation}

Table~\ref{tab:notation_full} provides a comprehensive overview of notation used in the paper.

\begin{table}[H]
	\centering
	\footnotesize
	\renewcommand{\arraystretch}{0.88}
	\caption{Complete notation reference.}
	\label{tab:notation_full}
	\makebox[\textwidth][c]{%
	\begin{tabular}[t]{ll}
		\toprule
		Symbol & Meaning \\
		\midrule
		\multicolumn{2}{l}{\textbf{Graphs and sets}}\\
		$G=(N,E)$ & Undirected road network. \\
		$N$ & Demand nodes and candidate sites. \\
		$F$ & Facility sites, $F\subseteq N$. \\
		\midrule
		\multicolumn{2}{l}{\textbf{Parameters (nonneg.\ integers or $+\infty$)}}\\
		$\bar{b}$ & Investment budget. \\
		$\bar{d}$ & Coverage radius / length bound. \\
		$h_k$ & Demand weight at node $k\in N$. \\
		$\ell_e$ & Length of edge $e\in E$ ($0$ for facility edges). \\
		$c^E_e$ & Upgrade cost for $e\in E$ ($0$ for existing edges). \\
		$c^F_j$ & Facility opening cost for $j\in F$. \\
		\midrule
		\multicolumn{2}{l}{\textbf{Decision variables}}\\
		$z_k$ & Demand node $k$ is covered. \\
		$x_e$ & Build/upgrade edge $e\in \overline{E}$. \\
		$y_j$ & Open facility at $j\in F$ ($\equiv x_{\{s,j\}}$). \\
		$f^k_a$ & Commodity-$k$ flow on arc $a\in\overline{A}$. \\
		\bottomrule
	\end{tabular}
	\quad
	\begin{tabular}[t]{ll}
		\toprule
		Symbol & Meaning \\
		\midrule
		\multicolumn{2}{l}{\textbf{Extended graph (super-source construction)}}\\
		$\overline{G}$ & Extended graph with source $s$, facility edges. \\
		$\overline{N}$ & Extended node set, $\overline{N}=N\cup\{s\}$. \\
		$E^F$ & Facility edges, $E^F=\{\{s,j\}\mid j\in F\}$. \\
		$\overline{E}$ & Extended undirected edges, $\overline{E}=E\cup E^F$. \\
		$\overline{A}$ & Bidirected arcs of $\overline{E}$. \\
		\midrule
		\multicolumn{2}{l}{\textbf{Block-cut tree notation}}\\
		$\mathcal{B}$ & Set of blocks (biconnected components). \\
		$\mathcal{A}$ & Set of articulation points (APs). \\
		$N_B, E_B$ & Node set and edge set of block $B$. \\
		$a_C$ & Parent AP of child block $C$. \\
		$\mathcal{C}_B$ & Set of child blocks of block $B$. \\
		$T_B$ & Subtree rooted at block $B$. \\
		$\mathcal{R}_B$ & Set of relevant radii for block $B$. \\
		\midrule
		\multicolumn{2}{l}{\textbf{DP states and selectors}}\\
		$V(B,b,r)$ & DP value at budget $b$, signed parent radius $r$. \\
		$\lambda^C_{b',r'}$ & Select child state $(b',r') \in \Lambda_C$. \\
		\bottomrule
	\end{tabular}}

	\vspace{0.6ex}
	{\footnotesize The signed radius $r\in\{-\bar{d},\dots,\bar{d}\}$ encodes the parent
	interaction: $r=0$ no interaction (decoupled), $r>0$ external coverage entering at
	remaining radius $r$, $r<0$ an obligation to cover the parent side within distance
	$\bar{d}-|r|$.}
\end{table}


\newcommand{\milpsetup}{%
	\footnotesize
	\allowdisplaybreaks
	\setlength{\jot}{5pt}%
	\setlength{\abovedisplayskip}{6pt}%
	\setlength{\belowdisplayskip}{6pt}%
	\setlength{\abovedisplayshortskip}{3pt}%
	\setlength{\belowdisplayshortskip}{3pt}%
}

\section{Block-level flow formulation with border conditions}
\label{app:flow_border}

This appendix provides the complete block-level flow formulation for computing the DP values $V(B,b,r)$ under the signed-radius border conditions of Section~\ref{sec:dp_framework}.
We consider block $B$ with node set $N_B$, edge set $E_B$, facility set $F_B$, parent AP $a_B$, and child blocks $\mathcal{C}_B$ with child APs $a_C$ for $C \in \mathcal{C}_B$.
Following Section~\ref{sec:flowmip}, the source-node construction is applied to each block individually:
the extended graph $\overline{G}_B$ has node set $\overline{N}_B := N_B \cup \{s\}$ with source node $s$, edge set $\overline{E}_B := E_B \cup \{\{s,j\} : j \in F_B\}$ with $\ell_{\{s,j\}} := 0$ and $c^E_{\{s,j\}} := c^F_j$, and arc set $\overline{A}_B$ containing both directions $(u,v),(v,u)$ of each edge, each inheriting the length of its edge.
Facility openings are the edge decisions $x_{\{s,j\}}$ rather than explicit $y_j$ variables.
For the child states, recall from Sections~\ref{sec:dp_framework} and~\ref{sec:bct_mip} the state sets $\Lambda_C \subseteq \{0,\ldots,\bar{b}\} \times \mathcal{R}_C$ over the relevant child radii $\mathcal{R}_C$ and the selectors $\lambda^C_{b',r'}$.
We write $\mathcal{R}_C^+ := \{r' \in \mathcal{R}_C : r'>0\}$ and $\mathcal{R}_C^- := \{r' \in \mathcal{R}_C : r'<0\}$ for the positive and negative relevant child radii.
As shorthand for the selected child states, the \emph{activations} $\alpha^C_{\text{out}}[r']$ for $r' \in \mathcal{R}_C^+$ and $\alpha^C_{\text{int}}[r']$ for $r' \in \mathcal{R}_C^-$ indicate that the child state with radius $r'$ is selected, and their aggregates $\alpha^C_{\text{out}}$ and $\alpha^C_{\text{int}}$ that some OUT resp.\ INT state of $C$ is selected; they are plain sums of the selectors, defined by Constraints~\eqref{eq:bfg_act_out}--\eqref{eq:bfg_act_agg} of the master formulation.

The construction rests on the following observation: the border conditions largely preserve the flow model.
They modify the \emph{instance} on which it is solved, while the model itself remains unchanged and is only extended to enforce the border conditions and to couple the objective and budget to the precomputed child states.
Recall from Section~\ref{sec:flowmip} that a node is covered if and only if there is an $s$-path of length at most $\bar{d}$ via edges in the solution.
The border conditions extend exactly the two sides of this characterization, following the auxiliary-node interpretation of Section~\ref{sec:bct_border} and Figure~\ref{fig:bc_auxiliary_nodes}: external coverage entering the block adds new origins for coverage paths (each virtual facility is added to the instance as a further facility), and coverage that $B$ must provide to a neighboring block adds new destinations (each virtual demand node is added as a further demand node).
On this auxiliary instance, the flow constraints of Section~\ref{sec:flowmip} apply unchanged; this coupling consists of the child-state selectors, which also add the DP value and child budget terms to the objective and budget, together with the linking constraints that tie the virtual elements to the selected states.

\subsection{The auxiliary instance}
\label{app:flow_border_structure}

Each of the four non-BLANK border conditions of Section~\ref{sec:bct_border} adds one vertex with its connecting arcs to the instance (the BLANK conditions $r = 0$, $r_C = 0$ add nothing; cf.\ Table~\ref{tab:aux_nodes} of the main paper):
\begin{itemize}
	\item \emph{Parent OUT ($r > 0$).}
	The virtual facility at distance $\bar{d} - r$ from $a_B$ is added to the auxiliary instance as a vertex $w_B$, connected like any other facility:
	by the facility arc $(s, w_B)$ of length $0$, and by the connecting arc $(w_B, a_B)$ of length $\bar{d} - r$ that leads into the block.
	Flow through $w_B$ means that a coverage path originates at the parent side and enters $B$ at $a_B$ with remaining radius $r$.
	Since each parent state is solved as its own MILP, $w_B$ is part of the parent-OUT MILPs ($r > 0$) only and open whenever present; using it remains optional.
	Unlike a regular facility, its opening is not a design decision and has no cost.
	\item \emph{Child INT states ($r' \in \mathcal{R}_C^-$).}
	The virtual facility of a child INT state is added likewise, as a vertex $w_C[r']$ with the arcs $(s, w_C[r'])$ of length $0$ and $(w_C[r'], a_C)$ of length $\bar{d} - |r'|$, one per candidate radius.
	Because the child states are chosen inside the MILP, $w_C[r']$ is open only if the INT state with radius $r'$ is selected, i.e., if $\alpha^C_{\text{int}}[r'] = 1$; even then it remains optional.
	\item \emph{Child OUT states ($r' \in \mathcal{R}_C^+$).}
	The virtual demand node of a child OUT state is added as a vertex $d_C$, one per child, with the single connecting arc $(a_C, d_C)$.
	Its length depends on the selected state,
	$\ell_{(a_C, d_C)} := \sum_{r' \in \mathcal{R}_C^+} r'\, \alpha^C_{\text{out}}[r']$,
	i.e., it is the selected radius $r_C$ of the child selection $\lambda^C$ if an OUT state of $C$ is selected and $0$ otherwise.\footnote{Conceptually, each candidate radius $r' \in \mathcal{R}_C^+$ has its own virtual demand node at distance $r'$ from $a_C$. Since exactly one state is selected, they merge into the single node $d_C$ whose arc carries the selected radius as its length: as in Section~\ref{sec:ext_flowmip}, one radius-independent commodity per child suffices.}
	The node $d_C$ must be covered exactly if an OUT state of $C$ is selected; by the covering characterization above, covering it within $\bar{d}$ then means reaching $a_C$ within $\bar{d} - r_C$.
	\item \emph{Parent INT ($r < 0$).}
	The virtual demand node of the parent INT case is added likewise, as a vertex $d_B$ with the connecting arc $(a_B, d_B)$ of length $|r|$.
	It is part of the parent-INT MILPs ($r < 0$) only and must be covered there: an $s$-path of length at most $\bar{d}$ to $d_B$ reaches $a_B$ with remaining radius $|r|$, i.e., an opened facility covers $a_B$ with radius $|r|$ to spare.
\end{itemize}
Virtual demand nodes carry no objective weight; the value of the coverage provided to a neighboring block enters through the DP terms $V(C,b',r')$ instead.

\paragraph{The auxiliary instance.}
Altogether, the auxiliary instance of block $B$ and parent state $(b,r)$ is the graph $\widetilde{G}_B$, obtained from $\overline{G}_B$ as follows:
\begin{itemize}
	\item \emph{Nodes.} The block nodes $N_B$ and the source node $s$, extended by the virtual facilities
	$\mathcal{W}_B := \{w_B\} \cup \{w_C[r'] : C \in \mathcal{C}_B,\ r' \in \mathcal{R}_C^-\}$
	and the virtual demand nodes
	$\mathcal{D}_B := \{d_C : C \in \mathcal{C}_B\} \cup \{d_B\}$.
	\item \emph{Arcs.} The arc set $\widetilde{A}_B$ contains both directions of every edge in $\overline{E}_B$ (block edges and facility edges), extended by the virtual arcs listed above.
	Virtual arcs are directed and have no reverse arcs: flow can pass a virtual facility only from $s$ into the block, and can enter a virtual demand node but not leave it.
	From now on, $\delta^\pm(v)$ are taken with respect to $\widetilde{A}_B$.
	\item \emph{Demands.} The demand set is $\widetilde{N}_B := \bigl(N_B \setminus \{a_B\}\bigr) \cup \mathcal{D}_B$: every block node except the parent AP, whose demand is counted in the parent block, plus the virtual demand nodes.
	Exactly as in Section~\ref{sec:flowmip}, demand $k$ is covered if one unit of its commodity $f^k$ runs from $s$ to $k$ within length $\bar{d}$: the coverage radius is the same for all demands, and the border distances are carried by the virtual arcs.
	Every demand carries a coverage indicator $z_k$, the right-hand side of its conservation constraint; the only distinction is that covering a real demand is a free decision, whereas the virtual demand nodes are forced to be covered exactly as required, by the explicit constraints $z_{d_C} = \alpha^C_{\text{out}}$ and $z_{d_B} = 1$ of Section~\ref{app:flow_border_milp}.
\end{itemize}
At most one of $w_B$ and $d_B$ is present, depending on the sign of $r$; Section~\ref{app:flow_border_milp} states the convention under which a single formulation covers all three cases.

\subsection{The block-level MILP}
\label{app:flow_border_milp}

Rather than stating three near-identical MILPs, we present one \emph{master formulation}, valid for every parent radius $r$, under the following convention:
\begin{quote}
	The parent virtual facility $w_B$ exists only if $r > 0$; the parent virtual demand node $d_B$ exists only if $r < 0$; for $r = 0$ neither exists.
	Elements that do not exist are omitted from all constraints together with their arcs and flow variables (equivalently, fixed to $0$).
\end{quote}
Instantiating the convention yields the three cases of Section~\ref{sec:dp_framework}: for OUT ($r>0$) the virtual demand node $d_B$ is absent together with its arc, its commodity, and Constraint~\eqref{eq:bfg_zdb}; for INT ($r<0$) the virtual facility $w_B$ is absent with its arcs and flow variables; for BLANK ($r=0$) both are absent and only the child-side virtual elements remain.
Section~\ref{app:implementation_notes} describes how the transitions between these cases are realized on a single MILP instance.

\paragraph{Objective, budget, and child-state selection.}
These rows realize the recursion~\eqref{eq:dp_highlevel} and extend the objective and budget of the base formulation by the child terms:
the objective~\eqref{eq:bfg_obj} maximizes the demand covered inside $B$ plus the DP values of the selected child states (the parent AP $a_B$ is excluded; its demand is counted in the parent block);
\eqref{eq:bfg_budget} charges the local design cost plus the budgets of the selected child states against $b$;
\eqref{eq:bfg_select} selects exactly one state per child.
\begingroup
\milpsetup
\begin{align}
\max \quad & \sum_{k \in N_B \setminus \{a_B\}} h_k z_k
+ \sum_{C \in \mathcal{C}_B} \sum_{(b',r') \in \Lambda_C} V(C, b', r') \lambda^C_{b',r'}
\label{eq:bfg_obj} \\
\text{s.t.} \quad
& \sum_{e \in \overline{E}_B} c^E_e x_e
+ \sum_{C \in \mathcal{C}_B} \sum_{(b',r') \in \Lambda_C} b' \lambda^C_{b',r'} \le b
\label{eq:bfg_budget} \\
& \sum_{(b',r') \in \Lambda_C} \lambda^C_{b',r'} = 1
&& \forall C \in \mathcal{C}_B
\label{eq:bfg_select}
\end{align}
\endgroup

\paragraph{Child-state activations.}
Constraints~\eqref{eq:bfg_act_out}--\eqref{eq:bfg_act_agg} define the activations $\alpha^C_{\text{out}}[r']$, $\alpha^C_{\text{int}}[r']$ and their aggregates used throughout Section~\ref{app:flow_border_structure}.
They are sums of selectors and introduce no new decisions; they feed the opening of the child virtual facilities, the coverage indicators of the virtual demand nodes, and the length of the arcs $(a_C, d_C)$.
Together with~\eqref{eq:bfg_select} they satisfy $\alpha^C_0 + \alpha^C_{\text{out}} + \alpha^C_{\text{int}} = 1$, where $\alpha^C_0 := \sum_{b': (b',0) \in \Lambda_C} \lambda^C_{b',0}$ selects the BLANK child state.
\begingroup
\milpsetup
\begin{align}
& \alpha^C_{\text{out}}[r'] = \sum_{b' :\, (b',r') \in \Lambda_C} \lambda^C_{b',r'}
&& \forall C \in \mathcal{C}_B,\ r' \in \mathcal{R}_C^+
\label{eq:bfg_act_out} \\
& \alpha^C_{\text{int}}[r'] = \sum_{b' :\, (b',r') \in \Lambda_C} \lambda^C_{b',r'}
&& \forall C \in \mathcal{C}_B,\ r' \in \mathcal{R}_C^-
\label{eq:bfg_act_int} \\
& \alpha^C_{\text{out}} = \sum_{r' \in \mathcal{R}_C^+} \alpha^C_{\text{out}}[r'], \quad
\alpha^C_{\text{int}} = \sum_{r' \in \mathcal{R}_C^-} \alpha^C_{\text{int}}[r']
&& \forall C \in \mathcal{C}_B
\label{eq:bfg_act_agg}
\end{align}
\endgroup

\paragraph{Coverage of the virtual demand nodes.}
Covering a virtual demand node is not a decision but a consequence of the selected states, imposed by two explicit constraints:
\eqref{eq:bfg_zdc} forces $d_C$ to be covered exactly if an OUT state of $C$ is selected, and \eqref{eq:bfg_zdb} makes the parent obligation unconditional.
With these constraints, the coverage indicators $z_k$, $k \in \widetilde{N}_B$, enter the flow formulation below in the same way for real and for virtual demand nodes.
\begingroup
\milpsetup
\begin{align}
& z_{d_C} = \alpha^C_{\text{out}}
&& \forall C \in \mathcal{C}_B
\label{eq:bfg_zdc} \\
& z_{d_B} = 1
\label{eq:bfg_zdb}
\end{align}
\endgroup

\paragraph{The flow formulation on the auxiliary instance.}
Constraints~\eqref{eq:bfg_cons}--\eqref{eq:bfg_len} are the flow formulation of Section~\ref{sec:flowmip} on the auxiliary instance $\widetilde{G}_B$, with one commodity $f^k$ per demand $k \in \widetilde{N}_B$.
Flow conservation~\eqref{eq:bfg_cons} carries over unchanged: commodity $k$ has supply $z_k$ at the source $s$ and demand $z_k$ at its node $k$, and the third case is standard flow conservation at every other node.
The edge-availability constraints~\eqref{eq:bfg_link} likewise carry over: an edge can carry flow, in either direction, only if it is upgraded, and a facility edge only if its facility is opened.
Virtual facilities have no design variables, so \eqref{eq:bfg_vlink} plays this role for them, opening a child facility $w_C[r']$ exactly if its INT state is selected.
The parent facility $w_B$ is open whenever it exists and requires no such constraint, and neither do the connecting arcs, whose flows coincide with those on the facility arcs.
The length bound~\eqref{eq:bfg_len} is the length bound of the base formulation with all virtual contributions written out:
a path through a virtual facility consumes the offset length of its connecting arc, accounted on its facility-arc flow, and the connecting arc of a virtual demand node contributes the case term, its length times the coverage indicator, which simplifies to $\sum_{r' \in \mathcal{R}_C^+} r'\, \alpha^C_{\text{out}}[r']$ resp.\ $|r|\, z_{d_B}$ by~\eqref{eq:bfg_zdc} resp.~\eqref{eq:bfg_zdb}; every term of~\eqref{eq:bfg_len} is therefore linear.
\begingroup
\milpsetup
\begin{align}
& \sum_{a \in \delta^+(v)} f^k_a - \sum_{a \in \delta^-(v)} f^k_a
= \begin{cases} \hphantom{-}z_k & \text{if } v = s \\ -z_k & \text{if } v = k \\ \hphantom{-}0 & \text{otherwise} \end{cases}
&& \forall k \in \widetilde{N}_B,\ v \in \overline{N}_B \cup \mathcal{W}_B \cup \mathcal{D}_B
\label{eq:bfg_cons} \\
& f^k_{(u,v)} + f^k_{(v,u)} \le x_e
&& \forall k \in \widetilde{N}_B,\ e \in \overline{E}_B
\label{eq:bfg_link} \\
& f^k_{(s, w_C[r'])} \le \alpha^C_{\text{int}}[r']
&& \forall k \in \widetilde{N}_B,\ C \in \mathcal{C}_B,\ r' \in \mathcal{R}_C^-
\label{eq:bfg_vlink} \\
& \sum_{a \in \overline{A}_B} \ell_a f^k_a
+ (\bar{d} - r)\, f^k_{(s, w_B)}
+ \sum_{\substack{C \in \mathcal{C}_B \\ r' \in \mathcal{R}_C^-}} (\bar{d} - |r'|)\, f^k_{(s, w_C[r'])} \notag \\
& \qquad + \begin{cases} \sum_{r' \in \mathcal{R}_C^+} r'\, \alpha^C_{\text{out}}[r'] & \text{if } k = d_C \\ |r|\, z_{d_B} & \text{if } k = d_B \\ \hphantom{-}0 & \text{otherwise} \end{cases}
\;\le\; \bar{d}\, z_k
&& \forall k \in \widetilde{N}_B
\label{eq:bfg_len}
\end{align}
\endgroup

\paragraph{Linking the child OUT commodities.}
One additional family of constraints completes the model: if no OUT state of $C$ is selected, then $z_{d_C} = 0$, and the variable upper bounds~\eqref{eq:bfg_gate} force the entire commodity of $d_C$ to zero.
\begingroup
\milpsetup
\begin{align}
& f^{d_C}_a \le \alpha^C_{\text{out}}
&& \forall C \in \mathcal{C}_B,\ a \in \widetilde{A}_B
\label{eq:bfg_gate}
\end{align}
\endgroup

Reading these constraints per demand type shows how the auxiliary instance realizes the border conditions.
For a real demand nothing changes relative to Section~\ref{sec:flowmip}, except that a coverage path may now originate at a virtual facility, entering the block at the respective AP with exactly the remaining radius that~\eqref{eq:bfg_len} accounts for; in particular, a demand node that is itself a child AP may be covered through its own child's virtual facility.
For the child OUT obligation ($k=d_C$), the case term of~\eqref{eq:bfg_len} equals the selected radius $r_C$, so the facility supplying this coverage must reach $a_C$ with radius $r_C$ to spare; it is opened in $B$, or it is virtual, namely the parent side ($w_B$, if $r>0$) or a sibling placed in an INT state.\footnote{With a virtual facility as source, \eqref{eq:bfg_len} reduces to path length $\le r - r_C$ resp.\ $\le |r_{C'}| - r_C$ for a sibling $C'$ in an INT state: external coverage may continue through $B$ and satisfy the OUT condition of $C$ whenever the connecting path is short enough. The sibling facilities give one flow per ordered pair of children of $B$.}
For the parent INT obligation ($k=d_B$), the case term $|r|$ leaves at most $\bar{d}-|r|$ for the path to $a_B$, which is exactly the covering condition the parent side imposes.

\paragraph{Variable domains.}
All variables are binary; the virtual elements exist as stated by the convention above.
The integrality of all flow variables can be relaxed to $[0,1]$ without loss.
\begingroup
\milpsetup
\begin{align}
& x_e \in \{0,1\} \ \forall e \in \overline{E}_B, \qquad
z_k \in \{0,1\} \ \forall k \in \widetilde{N}_B, \notag \\
& \lambda^C_{b',r'} \in \{0,1\} \ \forall C \in \mathcal{C}_B,\ (b',r') \in \Lambda_C, \quad
\alpha^C_{\text{out}}[r'],\ \alpha^C_{\text{int}}[r'],\ \alpha^C_{\text{out}},\ \alpha^C_{\text{int}} \in \{0,1\}, \notag \\
& f^k_a \in \{0,1\} \ \forall k \in \widetilde{N}_B,\ a \in \widetilde{A}_B \notag
\end{align}
\endgroup

The auxiliary instance is also how the formulation is built in practice; Section~\ref{app:implementation_notes} describes how a single MILP instance per block carries every DP state.

\section{Bound propagation details}
\label{app:bound_propagation}

This appendix collects the formal assumption, statements, proofs, and bound derivations summarized in Section~\ref{sec:propagating_bounds}.

\paragraph{Decomposition assumption.}
The empirical observation from Section~\ref{sec:bct_mip} (that the $\lambda$-allocation is recovered exactly from the incumbent's local design $(\hat{x}, \hat{z})$ via MCKP) motivates the following assumption used in the gap analysis below.

\begin{assumption}\label{ass:exact_child_alloc}
	For the theoretical gap analysis we assume that the child allocation is solved exactly given the local decisions $(\hat{x}, \hat{z})$.
	That is, the primal bound $l(B, b, r)$ returned by a sub-optimal MILP solve is optimal with respect to $\lambda$ given $(\hat{x}, \hat{z})$:
	\[
	l(B,b,r) \;=\; \max_{\lambda} \Bigl( h_B(\hat{x}, \lambda) \;+\; \sum_{C \in \mathcal{C}_B} l(C,\lambda^C) \Bigr),
	\]
	where $h_B(\hat{x}, \lambda)$ denotes the local coverage contribution of the primal solution.
\end{assumption}

\paragraph{Supporting experiment.}
A preliminary experiment motivates the assumption: we re-solved the selector subproblem~\eqref{eq:selector_subproblem} alone from the incumbent $(\hat{x}, \hat{z})$ of 41 sampled block solves, 13 of which were sub-optimal and therefore in the scope of the assumption.
In all 13, the MCKP recovered the incumbent's $\lambda$ exactly and in at most $0.13$\,s, supporting the qualitative property the assumption abstracts, namely that the selector subproblem is easy relative to the network-design part.
The sampling protocol and the per-solve outcomes (\texttt{results/lambda\_experiment\_combined.csv}) are in the online code repository.

\paragraph{Gap propagation setup.}
Consider a block $B$ with children $\mathcal{C}_B$,
where valid bounds $l(C,\lambda^C) \le V(C,\lambda^C) \le u(C,\lambda^C)$ with gap $\Delta(C,\lambda^C) := u(C,\lambda^C) - l(C,\lambda^C)$ are available for each child $C \in \mathcal{C}_B$ and each state $\lambda^C \in \Lambda_C$.
When solving the MILP for $V(B, b, r)$, for lack of the optimal values $V(C,\lambda^C)$, we use child \emph{lower} bounds $l(C,\lambda^C)$ for the child contributions in~\eqref{eq:dp_highlevel}.
For that problem, the solver returns a primal bound $l(B, b, r)$ and a dual bound $\tilde{u}(B, b, r)$:
\[
l(B, b, r) \;\le\;
\max_{x, \, \lambda} \Bigl(
h_B(x,\lambda)\;+\; \sum_{C \in \mathcal{C}_B} l(C,\lambda^C)
\Bigr)
\;\le\; \tilde{u}(B, b, r).
\]
Note that, since that computation did not use the child upper bounds $u(C,\lambda^C)$, $\tilde{u}(B,b,r)$ is not a valid upper bound on $V(B,b,r)$,
however, we can make it one by adding the maximum gap of each child block.

\begin{proposition}[Additive gap propagation]\label{prop:additive_gap}
	Define $\Delta_{\max}(C) := \max_{(b',r') \in \Lambda_C} \Delta(C,b',r')$.
	Then $u(B,b,r) := \tilde{u}(B,b,r) + \sum_{C \in \mathcal{C}_B} \Delta_{\max}(C)$ is a valid upper bound on $V(B,b,r)$.
\end{proposition}
\begin{proof}
	Replacing each child value $V(C,\lambda^C)$ by its upper bound $u(C,\lambda^C) = l(C,\lambda^C) + \Delta(C,\lambda^C)$, and bounding each $\Delta(C,\lambda^C) \le \Delta_{\max}(C)$, gives
	\begin{align*}
		V(B,b,r) \;&\le\; \max_{x, \, \lambda} \Bigl(
		h_B(x,\lambda)\;+\; \sum_{C \in \mathcal{C}_B} u(C,\lambda^C)
		\Bigr)\\
		&\le \max_{x, \, \lambda} \Bigl(
		h_B(x,\lambda)\;+\; \sum_{C \in \mathcal{C}_B} l(C,\lambda^C)
		\Bigr) \;+\; \sum_{C \in \mathcal{C}_B} \Delta_{\max}(C) \\
	&\le\; \tilde{u}(B, b, r) \;+\; \sum_{C \in \mathcal{C}_B} \Delta_{\max}(C).
\end{align*}
\end{proof}

To avoid the conservatism of replacing state-dependent gaps by $\Delta_{\max}$, one can solve an explicit upper bounding problem
$\bar{V}(B,b,r) := \max_{x, \lambda} \bigl( h_B(x,\lambda) + \sum_C u(C,\lambda^C) \bigr)$,
yielding the tighter chain $l(B,b,r) \le V(B,b,r) \le \bar{V}(B,b,r) \le u(B,b,r)$.

\paragraph{Relative gap preservation.}
While additive gaps accumulate, relative gaps are preserved under Assumption~\ref{ass:exact_child_alloc}.
We assume $l(B,\lambda^B) > 0$ and $h_B(\hat{x}, \lambda) > 0$ for all states under consideration; states with zero local coverage can be pruned from the DP table.

\begin{proposition}[Relative gap preservation]\label{prop:relative_gap}
	Under Assumption~\ref{ass:exact_child_alloc}, suppose that for block~$B$ in state~$\lambda^B$, the primal network design~$\hat{x}$ satisfies
	\[
		h_B(\hat{x}, \lambda) \;\geq\; \alpha_B \cdot \max_x\, h_B(x, \lambda)
		\qquad \text{for all child allocations } \lambda,
	\]
	with $\alpha_B \in (0,1]$.
	Suppose inductively that for each child~$C \in \mathcal{C}_B$,
	\[
		l(C,\lambda^C) \;\geq\; \alpha \cdot V(C,\lambda^C)
		\qquad \text{for all } \lambda^C \in \Lambda_C,
	\]
	for some common ratio $\alpha \in (0,1]$.
	Then
	\[
		l(B,\lambda^B) \;\geq\; \min(\alpha_B,\, \alpha) \cdot V(B, \lambda^B).
	\]
\end{proposition}
\begin{proof}
	Let $(x^*, \lambda^*)$ achieve $V(B, \lambda^B) = h_B(x^*, \lambda^*) + \sum_C V(C,\lambda^{C*})$,
	and write $H := h_B(x^*, \lambda^*)$, $S := \sum_C V(C,\lambda^{C*})$.
	By Assumption~\ref{ass:exact_child_alloc}, $l(B,\lambda^B) = \max_\lambda \bigl( h_B(\hat{x}, \lambda) + \sum_C l(C,\lambda^C) \bigr)$.
	Evaluating at the feasible point~$\lambda^*$:
	\[
	\begin{aligned}
	l(B,\lambda^B)
	&\;\geq\; h_B(\hat{x}, \lambda^*) + \sum_C l(C,\lambda^{C*})
	\;\geq\; \alpha_B \cdot H \;+\; \alpha \cdot S \\
	&\;\geq\; \min(\alpha_B, \alpha) \cdot (H + S)
	\;=\; \min(\alpha_B, \alpha) \cdot V(B, \lambda^B).
	\end{aligned}
	\]
\end{proof}
\begin{corollary}[Global approximation ratio]\label{cor:global_ratio}
	Applying Proposition~\ref{prop:relative_gap} inductively from the leaves to the root, the global ratio satisfies
	$l(\textup{root}) \geq \alpha_{\textup{global}} \cdot V(\textup{root})$ with
	$\alpha_{\textup{global}} = \min_{B \in \mathcal{T}} \alpha_B$,
	determined solely by the weakest local solve in the tree.
\end{corollary}

\paragraph{Bounds from neighboring states.}
Under limited time it may be a choice to solve only a subset of all states; we illustrate how upper bounds are derived when the BLANK state $V(B,b,0)$ is computed but the OUT and INT states are not.
Let $a_B$ be the parent articulation point of block $B$, let
$N^B(a_B, \rho) := \{ v \in T_B \mid d(a_B,v) \leq \rho \}$
with total weight $w^B(a_B,\rho) := \sum_{v \in N^B(a_B,\rho)} h_v$.
For $r > 0$, external coverage adds at most the weight of newly reachable nodes, giving the OUT bound $V(B,b,r) \le V(B,b,0) + w^B(a_B,r)$ stated in Section~\ref{sec:propagating_bounds}.
For $r < 0$, the block must open a facility within distance $\bar{d}-|r|$ of $a_B$, which covers nodes up to $\bar{d}$ beyond it, hence within $2\bar{d}-|r|$ of $a_B$.
Letting $c_{\min}(\rho) := \min_{j \in N^B(a_B,\rho) \cap F_B} c^F_j$ and assuming feasibility (i.e., $N^B(a_B,\bar{d}-|r|) \cap F_B \neq \emptyset$ and $b \ge c_{\min}(\bar{d}-|r|)$):
\begin{equation}
	V(B, b, r) \;\le\; V(B, b - c_{\min}(\bar{d}-|r|), 0) + w^B(a_B, 2\bar{d}-|r|) \qquad \text{for } r < 0.
	\label{eq:int_from_blank}
\end{equation}
Removing this cheapest eligible facility from an optimal INT solution, but keeping its edges, frees exactly $c_{\min}(\bar{d}-|r|)$ of budget and loses only coverage within $2\bar{d}-|r|$ of $a_B$; the remaining budget is bounded by the BLANK value.
Combined with monotonicity (Section~\ref{sec:dp_framework}), these bounds allow us to tighten gaps across related states in negligible time in case lower bounds are computed by other heuristic methods.

\section{Implementation choices}
\label{sec:appendix_implementation}
\label{app:implementation}
The BC-tree DP framework involves several implementation choices that affect computational performance.
We summarize the key decisions made after preliminary experiments.

\paragraph{Warm-starting and state ordering.}
We employ two complementary warm-starting strategies.
First, each block MILP is initialized with a greedy Prim-like construction that grows a subgraph from the source node $s$, selecting edges by a profit function that trades coverage gain against upgrade cost and respects the budget; its profit function is documented with the code.
Second, we exploit the monotonicity of DP values across neighboring states: given a solution for state $V(B,b,r)$, the solutions for $V(B,b+1,r)$ and $V(B,b,r+1)$ can only improve, so previous solutions provide valid lower bounds and warm starts.
Of the possible iteration orders, preliminary experiments favored budgets increasing from $0$ to $\bar{b}$ in the outer loop and, for each budget, the signed radius increasing from $-\bar{d}$ to $\bar{d}$ in the inner loop, in line with the monotonicity in $r$.

\paragraph{Incremental model construction.}
\label{app:implementation_notes}
For each block $B$, the flow formulation of Section~\ref{sec:flowmip} restricted to $B$ is built once as a persistent base model, and the border conditions are added as incremental modifications, a feature that modern MILP solvers support directly: the selectors and the virtual elements contribute a moderate number of new columns and rows, while all rows and columns of the base model remain unchanged.
A single model instance per block therefore carries every DP state $(b,r)$.
Consecutive states differ only in the budget $b$ and in a few bound and right-hand-side switches along the convention of Section~\ref{app:flow_border_milp}: for $r \le 0$, the flows through the parent virtual facility are fixed to $0$; for $r \ge 0$, the parent obligation is switched off by setting $z_{d_B} = 0$.
The solver thus keeps the model and its warm-start information across all states of a block.

\paragraph{Root block selection and block merging.}
The BC-tree can be rooted at any block; we default to the largest block by vertex count, which can be expected to be the computationally most challenging one.
Since the root block is solved only for the decoupled condition ($r=0$) at budget $\bar{b}$, placing the root there is the least costly option.
Graphs with many articulation points may decompose into numerous very small biconnected components such as single edges or triangles, for which the overhead of the DP framework may outweigh its benefit; we therefore merge adjacent blocks with at most three vertices, which reduces the number of DP subproblems while increasing the block sizes only marginally.

\paragraph{Relevant budget and distance values.}
A naive implementation iterates over all budgets $b \in \{0,\ldots,\bar{b}\}$ and distances $d \in \{0,\ldots,\bar{d}\}$, but only a subset of these values is achievable: the attainable budgets are the realizable sums of edge costs (a subset sum problem), and the feasible distances are the path lengths that exist in the graph.
We use simple heuristics to identify a superset of relevant values, avoiding enumeration of clearly infeasible states.
This filtering is purely combinatorial; among the remaining states, further states may be dominated and thus not Pareto-optimal, which is detected during the DP iteration, e.g., by obtaining the same solution despite having increased budget or radius.

\paragraph{Time allocation across DP states.}
When a total time limit $T_{\text{total}}$ is imposed, blocks first receive time proportionally to their edge count, $T_B = T_{\text{total}} \cdot |E_B| / \sum_{B'} |E_{B'}|$, of which a small buffer is withheld.
Within each block we then identify a small set of \emph{promising} budgets from the LP relaxation of the BLANK state, using the ratio $\mathrm{UB}_{\mathrm{LP}}(b)/b$ as a measure of marginal coverage per unit budget, and give them a higher weight than the remaining budgets; each budget's weight is shared equally among its associated DP states (one per signed radius $r$).
Time is finally distributed across the $q$ states with a front-loaded weighted schedule: half of the block's time is reserved as a common deadline buffer and the remaining half is split proportionally to the weights $w_i$, so that the $i$-th state receives
\[
T_i \;=\; \tfrac{T}{2} + \frac{\sum_{j \le i} w_j}{\sum_j w_j} \cdot \tfrac{T}{2} \;-\; t_{\text{elapsed}},
\]
and unused time from faster solves automatically accumulates for later states.
States that would receive too little time are skipped, and bounds are constructed from dominating states and heuristics (Section~\ref{sec:propagating_bounds}).
Compared with uniform weighting, the LP-guided scheme concentrates effort on budgets that look most promising in the relaxation, which we found to improve solution quality on hard instances without sacrificing optimality on easy ones.
The exact interval partition, weights, buffer fractions, and skipping threshold are documented with the code in the online repository.

%

\section{Pseudocode}
\label{app:pseudocode}

\SetAlgoSkip{smallskip}
\SetAlgoNlRelativeSize{-2}

This appendix collects pseudocode for the tree DP (Section~\ref{sec:treedp}) and the BC-tree
DP (Section~\ref{sec:bct}), in the unified signed-radius notation of
Section~\ref{sec:dp_framework} and including the practical refinements of
Section~\ref{sec:appendix_implementation}.
The complete pseudocode, with the per-node recursion in full, the helper functions
$\mathcal{A}, \mathcal{B}, \mathcal{C}^1, \mathcal{C}^2$, and the dominance-pruning and
time-allocation procedures, is provided as \texttt{docs/PSEUDOCODE.pdf} in the online code
repository, alongside a code-to-paper notation mapping that points from every algorithm step
to the concrete file and function.

\subsection{Tree DP for the MCNDP}
\label{app:pseudocode_tree}

Algorithm~\ref{alg:treedp} implements the top-level tree DP from Section~\ref{sec:treedp}.
It transforms the input tree into a binary tree, traverses nodes bottom-up, and at every
node fills the table $V(v, b, r)$ for all $(b,r) \in \{0,\ldots,\bar b\} \times \mathcal{R}_v$
from the recursions of Section~\ref{sec:treedp}, which combine the helpers
$\mathcal{A}, \mathcal{B}, \mathcal{C}^1, \mathcal{C}^2$ defined there.
The three radius regimes are distinguished as follows.
At a leaf, $V(v,b,r) = h_v$ if $r \ge \ell_e$ (the parent covers $v$ for free), and otherwise
$h_v$ if $v$ can afford a facility, with the shortfall being $0$ for $r \ge 0$ and $-\infty$ for
$r<0$, where the INT obligation would be violated.
At an inner node, $V(v,b,r)$ is the maximum over opening a facility at $v$ and passing
$\bar d$ downward, forwarding the inherited radius $r - \ell_e$ (only if $r \ge \ell_e$),
leaving $v$ uncovered via $\mathcal{A}$ (only if $0 \le r < \ell_e$), and letting either
subtree cover $v$ and pass a radius outward via $\mathcal{C}^1, \mathcal{C}^2$ (at $0$ for
$r \ge 0$ and at $|r|$ for $r<0$).
Negative budgets and radii outside $\mathcal{R}_v$ are treated as $-\infty$.

{\footnotesize
\begin{algorithm}[H]
\caption{$\textsc{TreeDP}(G, \bar b, \bar d)$: tree DP for the MCNDP.}
\label{alg:treedp}
\SetKwInOut{Input}{Input}\SetKwInOut{Output}{Output}
\Input{Tree $G=(N,E)$ with costs $c^F, c^E$, lengths $\ell$, weights $h$; budget $\bar b$; radius $\bar d$.}
\Output{Optimal objective value, set $Y \subseteq F$ of facilities, set $X \subseteq E$ of upgraded edges.}
Root $G$ at an arbitrary $v_r$ (we use the highest-degree node).\;
Transform $G$ into a binary tree by inserting dummy nodes ($h=0$, $c^F=+\infty$, $c^E=0$) so that every node has at most two children.\;
Compute all-pairs distances $d_G(\cdot,\cdot)$.\;
\ForEach{node $v \in V$}{
  $\mathcal{R}_v \;\leftarrow\; \bigl\{\, \pm(\bar d - d_G(v,u)) \;\bigm|\; u \notin T_v,\, d_G(v,u) < \bar d \,\bigr\} \cup \{0\}$\;
}
\ForEach{$v$ in post-order traversal}{
  \ForEach{$(b, r) \in \{0,\ldots,\bar b\} \times \mathcal{R}_v$}{
    Set $V(v,b,r)$ by the leaf resp.\ inner-node recursion of Section~\ref{sec:treedp}.\;
  }
}
$\mathrm{opt} \;\leftarrow\; V(v_r, \bar b, 0)$\;
$(Y, X) \;\leftarrow\; \textsc{Backtrack}(v_r, \bar b, 0)$\;
\Return{$(\mathrm{opt}, Y, X)$}
\end{algorithm}
}

The DP tables are stored as nested dictionaries indexed by $b$ and the signed radius, and
the helpers iterate over all budget splits $b_1+b_2 \le b$ and the binary edge-upgrade
decisions, which is what the complexity bound $\mathcal{O}(n \bar b^2 R^2)$ of
Section~\ref{sec:treedp} reflects.

\subsection{BC-tree DP for the MCNDP}
\label{app:pseudocode_bct}

Algorithm~\ref{alg:bctdp} extends the framework of Section~\ref{sec:bct} to graphs
with articulation points.
For each block $B$ and each state $(b,r)$ it constructs the parameterized block-level MILP
from Section~\ref{sec:bct_mip} and Section~\ref{app:flow_border}, applies a warm start, and
solves within an allotted time $T_{b,r}$.
The dominance-based pruning and the LP-guided time allocation it invokes are described in
Section~\ref{sec:appendix_implementation}.

{\footnotesize
\begin{algorithm}[H]
\caption{$\textsc{BCTreeDP}(G, \bar b, \bar d, T_{\text{total}})$: BC-tree DP for the MCNDP.}
\label{alg:bctdp}
\SetKwInOut{Input}{Input}\SetKwInOut{Output}{Output}
\Input{Graph $G=(N,E)$ with costs and lengths as before; budget $\bar b$; radius $\bar d$; optional total time limit $T_{\text{total}}$.}
\Output{Best feasible value $l(\text{root})$ with global upper bound $u(\text{root})$; facilities $Y$; upgraded edges $X$.}
Decompose $G$ into blocks $\mathcal{B}$ and articulation points $\mathcal{A}$.\;
(Optional) merge adjacent blocks of size $\le 3$ to reduce DP overhead.\;
Build the block-cut tree $\mathcal{T}$ and root it at the largest block $B_r$.\;
\ForEach{block $B \in \mathcal{B}$}{
  Precompute $\mathcal{R}_B$, the set of relevant budgets, and AP distances.\;
}
Allocate per-block time $T_B$ proportional to $|E_B|$ (with cap on the root block).\;
\ForEach{$B$ in post-order traversal of the block nodes of $\mathcal{T}$}{
  Compute $B.\Delta_{\max} := \sum_{C \in \mathcal{C}_B} \Delta_{\max}(C)$.\tcp*[r]{aggregate child gaps}
  \ForEach{$(b, r)$ in the relevant states of $B$, budget outer, signed radius inner}{
    $T_{b,r} \;\leftarrow\;$ LP-guided share of $T_B$ \tcp*[r]{Section~\ref{sec:appendix_implementation}}
    \uIf{$T_{b,r} < T_{\min}$}{
      $(l, u) \;\leftarrow\; \textsc{SkipState}(B, b, r)$ \tcp*[r]{bounds from neighbor states, e.g.~\eqref{eq:int_from_blank}}
    }
    \Else{
      Build the block MILP for $(b,r)$ (Section~\ref{app:flow_border}), reading the children's
      stored values $V(C,b',r')$ over the non-dominated states $(b',r') \in \Lambda_C$ through
      the child selectors.\;
      Warm-start from the greedy heuristic or the incumbent of a neighboring state, and solve
      within $T_{b,r}$; record primal $l$ and dual $\tilde u$.\;
      $u \;\leftarrow\; \tilde u + B.\Delta_{\max}$ \tcp*[r]{Prop.~\ref{prop:additive_gap}}
    }
    Store $l, u$ in the DP tables for $V(B,b,r)$.\;
  }
  Mark states $(b',r')$ with $u(B,b',r') \le l(B,b,r)$ for some $b \le b'$, $r \le r'$ as
  dominated, exclude them from $\Lambda_B$, and pass $\Lambda_B$ with
  $\Delta_{\max}(B)$ to the parent.\;
}
$\mathrm{opt} \;\leftarrow\; l(B_r, \bar b, 0)$\;
$(Y, X) \;\leftarrow\; \textsc{BacktrackBCT}(B_r, \bar b, 0)$\;
\Return{$(\mathrm{opt},\, u(B_r,\bar b, 0),\, Y, X)$}
\end{algorithm}
}

The block-level MILP is re-used across states of the same block
(Section~\ref{app:implementation_notes}), and the aggregate $B.\Delta_{\max}$ is constructed
once per block, in line with the inductive bound argument in
Section~\ref{app:bound_propagation}.
Skipped states fill the DP table with the conservative bounds of that section, so that the
parent's selectors remain valid even when no full MILP solve has happened.

\section{Instance details}
\label{sec:appendix_instances}
See Table \ref{tab:instances} for a summary of the BC-tree instances used in the experiments of Section~\ref{sec:experiments}.

\begin{table}[htbp]
	\centering
	\caption{Summary of BC-tree instances. $n$: nodes; $m$: edges; Bl: biconnected components (before merging); APs: articulation points; AP$^\circ$: maximum AP degree; Max/Min/Avg: block sizes.}
	\label{tab:instances}
	\scriptsize
	\setlength{\tabcolsep}{2.5pt}
	\begin{tabular}[t]{rlrrrrrrrr}
		\hline
		ID & Instance & $n$ & $m$ & Bl & APs & AP$^\circ$ & Max & Min & Avg \\
		\hline
		1 & mixed\_2full\_sevilla\_sioux\_4b & 152 & 332 & 2 & 1 & 2 & 96 & 57 & 76.5 \\
		2 & mixed\_sevilla2\_grid6x6\_3\_5b & 202 & 418 & 5 & 4 & 2 & 49 & 36 & 41.2 \\
		3 & mixed\_sevilla3\_grid5x5\_2\_grid4x4\_2\_7b & 223 & 485 & 7 & 6 & 2 & 49 & 16 & 32.7 \\
		4 & mixed\_sevilla4\_grid5x5\_2\_6b & 241 & 556 & 6 & 5 & 2 & 49 & 25 & 41.0 \\
		5 & mixed\_sevilla5\_grid5x5\_1\_6b & 265 & 635 & 6 & 5 & 2 & 49 & 25 & 45.0 \\
		6 & sevilla\_5\_5b & 241 & 595 & 5 & 4 & 2 & 49 & 49 & 49.0 \\
		7 & sevilla\_6\_6b & 289 & 714 & 6 & 5 & 2 & 49 & 49 & 49.0 \\
		8 & sioux\_10copies\_8b & 216 & 374 & 6 & 5 & 2 & 57 & 24 & 36.8 \\
		9 & sioux\_11copies\_2b & 251 & 418 & 2 & 1 & 2 & 228 & 24 & 126.0 \\
		10 & sioux\_3subnets\_5b & 51 & 76 & 5 & 4 & 2 & 24 & 2 & 11.0 \\
		11 & sioux\_3subnets\_6b & 66 & 101 & 6 & 5 & 2 & 24 & 2 & 11.8 \\
		12 & sioux\_4subnets\_18b & 78 & 111 & 18 & 15 & 3 & 35 & 2 & 5.3 \\
		13 & sioux\_4subnets\_20b & 62 & 85 & 20 & 16 & 3 & 34 & 2 & 4.0 \\
		14 & sioux\_5copies\_5b & 116 & 190 & 5 & 4 & 2 & 24 & 24 & 24.0 \\
		15 & sioux\_5subnets\_10b & 74 & 108 & 10 & 9 & 2 & 60 & 2 & 8.3 \\
		16 & sioux\_5subnets\_34b & 87 & 111 & 34 & 30 & 3 & 16 & 2 & 3.5 \\
		17 & sioux\_7copies\_2b & 159 & 265 & 2 & 1 & 2 & 136 & 24 & 80.0 \\
		18 & sioux\_8b & 47 & 76 & 4 & 3 & 2 & 43 & 2 & 12.5 \\
		19 & sioux\_8copies\_4b & 170 & 297 & 2 & 1 & 2 & 113 & 58 & 85.5 \\
		20 & sioux\_9b & 56 & 87 & 9 & 7 & 3 & 25 & 2 & 7.1 \\
		21 & sioux\_9copies\_17b & 216 & 350 & 17 & 16 & 2 & 24 & 2 & 13.6 \\
		22 & tree\_balanced\_4b\_118n & 118 & 233 & 4 & 1 & 4 & 49 & 24 & 30.2 \\
		23 & tree\_balanced\_4b\_93n & 93 & 152 & 4 & 1 & 4 & 24 & 24 & 24.0 \\
		24 & tree\_balanced\_7b\_337n & 337 & 833 & 7 & 6 & 2 & 49 & 49 & 49.0 \\
		25 & tree\_random\_3b\_70n & 70 & 114 & 3 & 1 & 3 & 24 & 24 & 24.0 \\
		26 & tree\_random\_7b\_162n & 162 & 266 & 7 & 6 & 2 & 24 & 24 & 24.0 \\
		27 & tree\_random\_7b\_165n & 165 & 272 & 7 & 6 & 2 & 25 & 24 & 24.4 \\
		28 & tree\_star\_5b\_116n\_v2 & 116 & 190 & 5 & 1 & 5 & 24 & 24 & 24.0 \\
		29 & tree\_star\_7b\_162n & 162 & 266 & 7 & 5 & 3 & 24 & 24 & 24.0 \\
		30 & tree\_star\_7b\_162n\_v2 & 162 & 266 & 7 & 1 & 7 & 24 & 24 & 24.0 \\
		31 & tree\_star\_7b\_237n & 237 & 509 & 7 & 5 & 3 & 49 & 24 & 34.7 \\
		32 & tree\_star\_7b\_237n\_v2 & 237 & 509 & 7 & 1 & 7 & 49 & 24 & 34.7 \\
		33 & tree\_star\_7b\_239n & 239 & 513 & 7 & 1 & 7 & 49 & 24 & 35.0 \\
		34 & tree\_star\_8b\_262n & 262 & 551 & 8 & 1 & 8 & 49 & 24 & 33.6 \\
		\hline
	\end{tabular}
\end{table}

\section{Practical solve scenarios: experimental details}
\label{sec:appendix_practical_solve}

This section documents the setup for the two practical solve scenarios discussed in the main text (Figure~\ref{fig:budget_and_time_sensitivity}).
Both experiments are run on the same five instances: the largest networks in our benchmark that did not reach optimality within one hour for either solver, where the choice of approach has the highest practical impact.
All runs use a single thread and the same Gurobi configuration as in the main computational study; the BC-tree DP uses the LP-guided state weighting described in Section~\ref{sec:appendix_implementation}.
Reported coverage values are in both cases the mean over the five instances.

\paragraph{Coverage--budget curve (Figure~\ref{fig:budget_curve}).}
We split the interval $[0,\bar{b}]$ into $20$ equally spaced budget values $b_i = \tfrac{i}{19}\,\bar{b}$ for $i=0,\dots,19$, at coverage radius factor $\gamma = 0.10$.
The BC-tree DP solves all $20$ budget points within a single execution at a one-hour wall-clock limit.
Each FlowMIP variant solves the $20$ points sequentially under the same one-hour total budget; the difference between variants is solely the per-point time allocation, which is proportional to $\log(1+b_i)$ for \emph{FlowMIP-LogTime}, to $b_i$ for \emph{FlowMIP-LinTime}, and to $b_i^2$ for \emph{FlowMIP-QuadTime}, with the weights normalized so that the per-point times sum to the total one-hour budget.

\paragraph{Time-limit sensitivity (Figure~\ref{fig:time_sensitivity}).}
We solve each instance at its default budget for total time limits geometrically spaced between roughly one minute and one hour.
For each time limit, both the BC-tree DP and FlowMIP are run independently under that single limit (no warm-starting across runs) and the best incumbent at the limit is recorded.


\section{Tree DP: additional results}
\label{sec:additional_tree_results}

Figure~\ref{fig:tree_dp_budget_radius} reports execution times on the 20-node tree experiment (324 test cases: $\bar{b}, \bar{d} \in \{10, 20, \ldots, 60\}$, nine random trees per combination) as a function of $\bar{b}$ and $\bar{d}$ separately.
Both algorithms scale with increasing parameter values, but the MILP shows increasing variance and longer tails (up to $11.6$ seconds), while the DP stays below $1.3$ seconds.
Table~\ref{tab:tree_ratio_grid} gives the underlying ratio of mean runtimes per parameter combination.
The MILP is faster on average in 18 of the 36 cells, all of them at low $\bar{b}$ or low $\bar{d}$, but its advantage there is over solves that are already cheap: across those cells the MILP averages $0.11$ seconds and the DP $0.21$ seconds, and no single solve of either algorithm exceeds $1.5$ seconds.
The effect is most pronounced in the $\bar{d} = 10$ column, where the MILP wins all 54 test cases by a median factor of $6.5$.
In the remaining 18 cells the ordering reverses and the absolute times matter more: the MILP averages $1.84$ seconds against the DP's $0.54$ seconds.

\begin{table}[htbp]
	\centering
	\caption{Ratio of mean MILP runtime to mean DP runtime on 20-node trees (nine trees per cell). Values above~1 indicate that the DP is faster.}
	\label{tab:tree_ratio_grid}
	\begin{tabular}{lrrrrrr}
		\toprule
		& \multicolumn{6}{c}{Coverage radius $\bar{d}$} \\
		\cmidrule(l){2-7}
		Budget $\bar{b}$ & 10 & 20 & 30 & 40 & 50 & 60 \\
		\midrule
		10 & 0.28 & 0.40 & 0.53 & 0.62 & 0.70 & 0.74 \\
		20 & 0.17 & 0.67 & 0.73 & 0.91 & 1.13 & 1.10 \\
		30 & 0.23 & 0.83 & 1.46 & 1.94 & 2.02 & 2.24 \\
		40 & 0.30 & 0.87 & 1.65 & 2.51 & 3.57 & 4.18 \\
		50 & 0.18 & 0.77 & 1.68 & 2.95 & 3.96 & 5.19 \\
		60 & 0.16 & 0.58 & 2.29 & 3.97 & 4.66 & 5.29 \\
		\bottomrule
	\end{tabular}
\end{table}

\begin{figure}[htbp]
	\centering
	\begin{subfigure}[b]{0.48\textwidth}
		\centering
		\includegraphics[width=\textwidth]{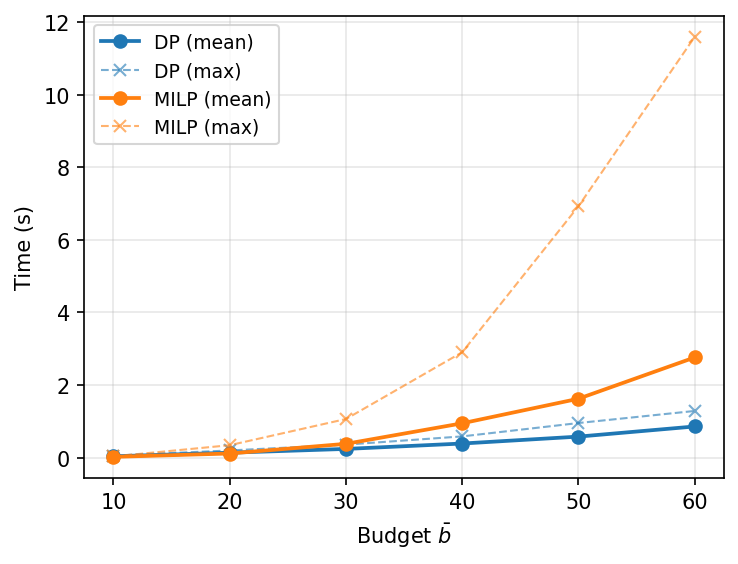}
		\caption{Execution time vs.\ budget $\bar{b}$.}
		\label{fig:tree_dp_budget}
	\end{subfigure}
	\hfill
	\begin{subfigure}[b]{0.48\textwidth}
		\centering
		\includegraphics[width=\textwidth]{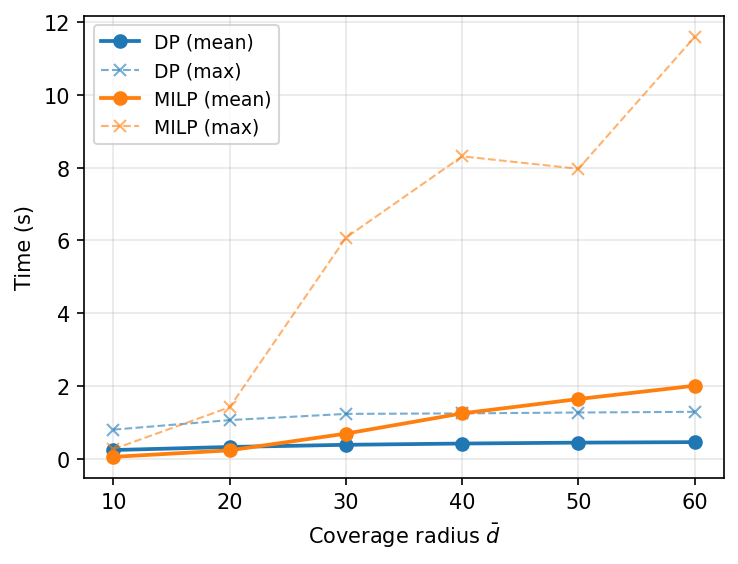}
		\caption{Execution time vs.\ coverage radius $\bar{d}$.}
		\label{fig:tree_dp_radius}
	\end{subfigure}
	\caption{MILP and DP execution times on 20-node random trees across 36 parameter combinations.}
	\label{fig:tree_dp_budget_radius}
\end{figure}

Figure~\ref{fig:tree_dp_n_moderate} shows the scaling behavior at the two lower parameter settings, $\bar{b}=20,\, \bar{d}=40$ and $\bar{b}=30,\, \bar{d}=30$.
Both complement Figure~\ref{fig:tree_dp_n_high} in the main text, which shows the higher-parameter setting $\bar{b}=40,\, \bar{d}=50$ where the DP's margin is far larger.
At $\bar{b}=30,\, \bar{d}=30$ the DP is ahead in mean runtime at every instance size, but only by a factor of $1.1$ to $2.2$, and it wins on 50 of the 64 individual trees.
At $\bar{b}=20,\, \bar{d}=40$ the two algorithms are close: the mean runtime ratio ranges from $0.8$ to $1.9$, the MILP is ahead on average at $n=40$, and the DP wins on 42 of the 64 individual trees.
Across all three scaling settings the MILP's slowest tree of a given size takes $1.3$ to $2.4$ times its mean runtime, whereas the DP stays within $1.1$ to $1.3$ times its own mean, the same spread effect seen at $n = 20$.

\begin{figure}[htbp]
	\centering
	\begin{subfigure}[b]{0.48\textwidth}
		\centering
		\includegraphics[width=\textwidth]{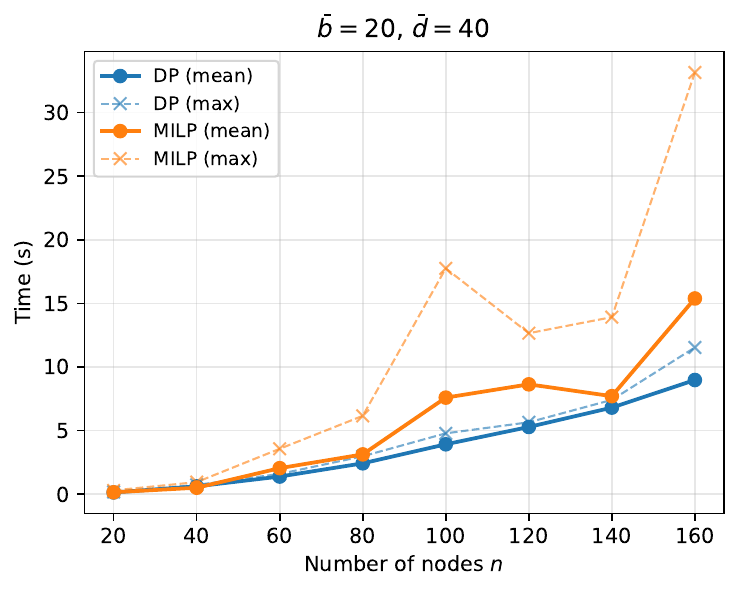}
		\caption{$\bar{b} = 20,\, \bar{d} = 40$.}
		\label{fig:tree_dp_n_b20r40}
	\end{subfigure}
	\hfill
	\begin{subfigure}[b]{0.48\textwidth}
		\centering
		\includegraphics[width=\textwidth]{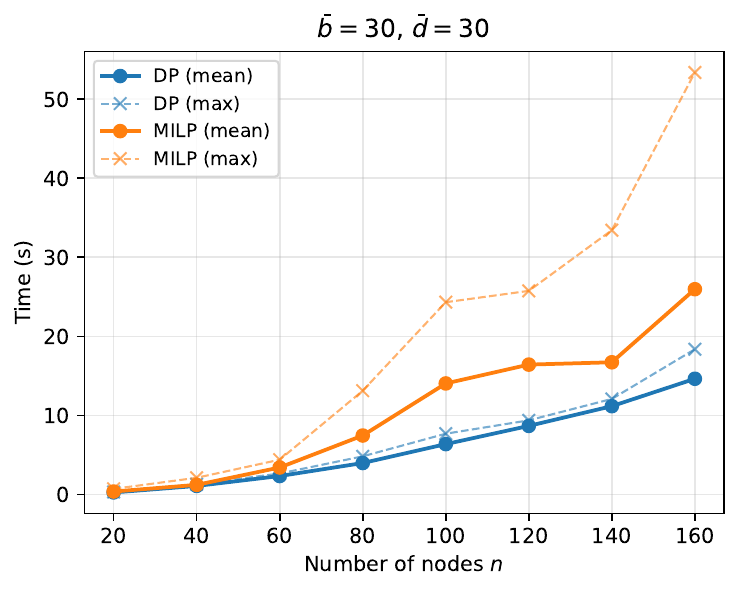}
		\caption{$\bar{b} = 30,\, \bar{d} = 30$.}
		\label{fig:tree_dp_n_b30r30}
	\end{subfigure}
	\caption{Scaling in $n$ at the two lower parameter settings, eight trees per size. Solid lines are means over the trees of a given size, dashed lines the corresponding maxima.}
	\label{fig:tree_dp_n_moderate}
\end{figure}

\section{BC-tree DP: time performance profile}
\label{sec:additional_bct_results}

Figure~\ref{fig:performance_profile_time} reports the absolute time performance profile over the 306 BC-tree test cases,
that is, the fraction of cases each method solves to proven optimality within a given time.
It is the counterpart to the solution quality profile of Figure~\ref{fig:performance_profile} in the main text:
FlowMIP closes more instances outright,
while DP-FlowMIP returns the better solution on the majority of the cases that neither method solves to optimality.

\begin{figure}[htbp]
	\centering
	\includegraphics[width=0.7\textwidth]{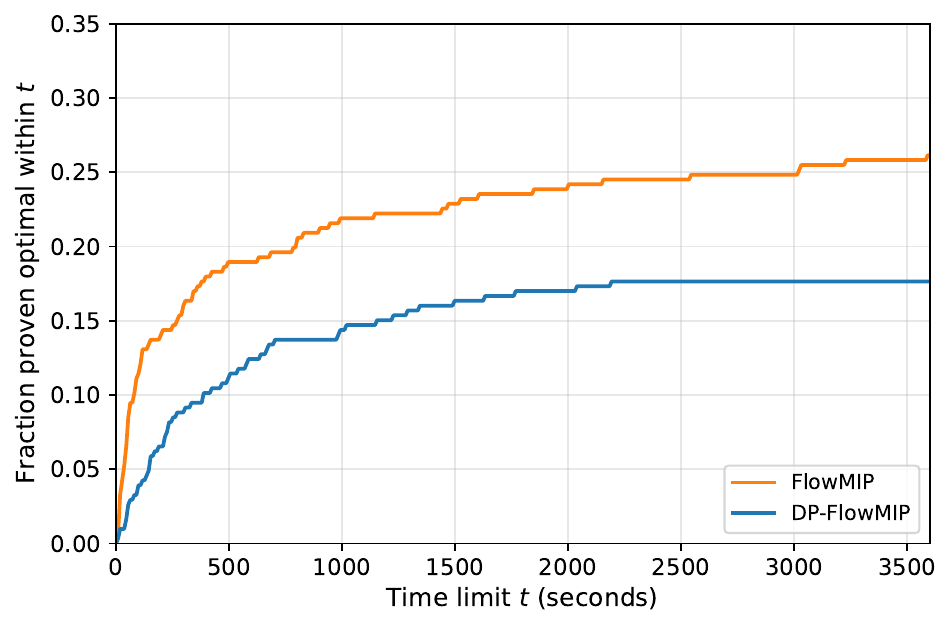}
	\caption{Absolute time performance profile.
		The y-axis shows the fraction of the 306 test cases solved to optimality within time $t$ seconds.
		FlowMIP reaches optimality on more instances (26.1\% at $t=3600$s) than DP-FlowMIP (17.6\%).
		The lazy-cut variants rarely reach optimality within the time limit.
		This complements the solution quality profile in Figure~\ref{fig:performance_profile}: while FlowMIP proves optimality faster, DP-FlowMIP finds better solutions when the time limit is reached.}
	\label{fig:performance_profile_time}
\end{figure}

\section{Complete per-case results}
\label{sec:additional_results}

Table~\ref{tab:full_results} below gives the per-test-case results for all 306 cases (34 instances $\times$ 9 budget/coverage combinations). The same data is also available as a CSV file in the code repository for programmatic access.

\scriptsize
\setlength{\tabcolsep}{3pt}
\begin{longtable}{rrrr|rrrr|rrrr}
\caption{Complete results for all 306 test cases. ID = instance ID (Table~\ref{tab:instances}); LB = best solution found,
UB = dual bound, Opt = optimality proven, Time = seconds.
Budget factor $\beta \in \{0.3, 0.5, 0.7\}$; coverage radius factor $\gamma \in \{0.05, 0.1, 0.15\}$.}
\label{tab:full_results} \\
\toprule
& & & & \multicolumn{4}{c|}{BC-tree DP} & \multicolumn{4}{c}{FlowMIP} \\
ID & $\beta$ & $\gamma$ & $|V|$ & LB & UB & Opt & Time & LB & UB & Opt & Time \\
\midrule
\endfirsthead
\multicolumn{12}{c}{\textit{Table~\ref{tab:full_results} continued}} \\
\toprule
& & & & \multicolumn{4}{c|}{BC-tree DP} & \multicolumn{4}{c}{FlowMIP} \\
ID & $\beta$ & $\gamma$ & $|V|$ & LB & UB & Opt & Time & LB & UB & Opt & Time \\
\midrule
\endhead
\midrule
\multicolumn{12}{r}{\textit{Continued on next page}} \\
\endfoot
\bottomrule
\endlastfoot

1 & 0.3 & 0.05 & 152 & 109 & 227 & No & 3519 & 117 & 151 & No & 3606 \\
1 & 0.3 & 0.1 & 152 & 116 & 230 & No & 3519 & 118 & 151 & No & 3607 \\
1 & 0.3 & 0.15 & 152 & 116 & 231 & No & 3519 & 116 & 150 & No & 3606 \\
1 & 0.5 & 0.05 & 152 & 176 & 292 & No & 3517 & 177 & 225 & No & 3607 \\
1 & 0.5 & 0.1 & 152 & 178 & 296 & No & 3518 & 182 & 222 & No & 3607 \\
1 & 0.5 & 0.15 & 152 & 179 & 299 & No & 3518 & 178 & 224 & No & 3608 \\
1 & 0.7 & 0.05 & 152 & 234 & 351 & No & 3511 & 227 & 279 & No & 3608 \\
1 & 0.7 & 0.1 & 152 & 227 & 352 & No & 3512 & 227 & 278 & No & 3608 \\
1 & 0.7 & 0.15 & 152 & 228 & 356 & No & 3513 & 229 & 277 & No & 3608 \\
\midrule
2 & 0.3 & 0.05 & 202 & 155 & 546 & No & 3568 & 159 & 192 & No & 3611 \\
2 & 0.3 & 0.1 & 202 & 157 & 596 & No & 3393 & 156 & 191 & No & 3611 \\
2 & 0.3 & 0.15 & 202 & 156 & 595 & No & 3174 & 152 & 189 & No & 3611 \\
2 & 0.5 & 0.05 & 202 & 239 & 785 & No & 3566 & 232 & 287 & No & 3612 \\
2 & 0.5 & 0.1 & 202 & 241 & 810 & No & 3567 & 232 & 285 & No & 3612 \\
2 & 0.5 & 0.15 & 202 & 241 & 817 & No & 3568 & 234 & 283 & No & 3612 \\
2 & 0.7 & 0.05 & 202 & 307 & 912 & No & 3565 & 303 & 347 & No & 3613 \\
2 & 0.7 & 0.1 & 202 & 306 & 937 & No & 3567 & 281 & 346 & No & 3614 \\
2 & 0.7 & 0.15 & 202 & 307 & 949 & No & 3568 & 293 & 345 & No & 3614 \\
\midrule
3 & 0.3 & 0.05 & 223 & 178 & 622 & No & 3569 & 176 & 219 & No & 3615 \\
3 & 0.3 & 0.1 & 223 & 177 & 646 & No & 3570 & 171 & 216 & No & 3615 \\
3 & 0.3 & 0.15 & 223 & 176 & 653 & No & 3571 & 172 & 217 & No & 3615 \\
3 & 0.5 & 0.05 & 223 & 264 & 762 & No & 3566 & 263 & 320 & No & 3616 \\
3 & 0.5 & 0.1 & 223 & 265 & 786 & No & 3570 & 255 & 317 & No & 3617 \\
3 & 0.5 & 0.15 & 223 & 266 & 794 & No & 3572 & 256 & 317 & No & 3617 \\
3 & 0.7 & 0.05 & 223 & 339 & 840 & No & 3567 & 325 & 390 & No & 3619 \\
3 & 0.7 & 0.1 & 223 & 340 & 870 & No & 3571 & 331 & 388 & No & 3619 \\
3 & 0.7 & 0.15 & 223 & 333 & 867 & No & 3572 & 322 & 389 & No & 3619 \\
\midrule
4 & 0.3 & 0.05 & 241 & 201 & 682 & No & 3418 & 207 & 245 & No & 3618 \\
4 & 0.3 & 0.1 & 241 & 206 & 715 & No & 3574 & 204 & 247 & No & 3619 \\
4 & 0.3 & 0.15 & 241 & 203 & 719 & No & 3304 & 201 & 246 & No & 3619 \\
4 & 0.5 & 0.05 & 241 & 298 & 870 & No & 3573 & 293 & 351 & No & 3622 \\
4 & 0.5 & 0.1 & 241 & 299 & 880 & No & 3575 & 279 & 349 & No & 3622 \\
4 & 0.5 & 0.15 & 241 & 298 & 891 & No & 3575 & 294 & 351 & No & 3623 \\
4 & 0.7 & 0.05 & 241 & 370 & 949 & No & 3572 & 367 & 429 & No & 3625 \\
4 & 0.7 & 0.1 & 241 & 371 & 967 & No & 3575 & 367 & 425 & No & 3626 \\
4 & 0.7 & 0.15 & 241 & 374 & 977 & No & 3576 & 364 & 425 & No & 3626 \\
\midrule
5 & 0.3 & 0.05 & 265 & 199 & 739 & No & 3556 & 203 & 266 & No & 3625 \\
5 & 0.3 & 0.1 & 265 & 199 & 774 & No & 3558 & 194 & 264 & No & 3626 \\
5 & 0.3 & 0.15 & 265 & 199 & 773 & No & 3558 & 199 & 266 & No & 3626 \\
5 & 0.5 & 0.05 & 265 & 310 & 935 & No & 3557 & 307 & 388 & No & 3629 \\
5 & 0.5 & 0.1 & 265 & 311 & 955 & No & 3559 & 301 & 389 & No & 3630 \\
5 & 0.5 & 0.15 & 265 & 312 & 978 & No & 3559 & 299 & 385 & No & 3630 \\
5 & 0.7 & 0.05 & 265 & 390 & 1021 & No & 3558 & 386 & 477 & No & 3632 \\
5 & 0.7 & 0.1 & 265 & 388 & 1049 & No & 3559 & 384 & 469 & No & 3633 \\
5 & 0.7 & 0.15 & 265 & 388 & 1064 & No & 3560 & 384 & 470 & No & 3635 \\
\midrule
6 & 0.3 & 0.05 & 241 & 193 & 663 & No & 3247 & 185 & 247 & No & 3620 \\
6 & 0.3 & 0.1 & 241 & 194 & 659 & No & 3067 & 188 & 252 & No & 3620 \\
6 & 0.3 & 0.15 & 241 & 196 & 679 & No & 3484 & 186 & 246 & No & 3620 \\
6 & 0.5 & 0.05 & 241 & 297 & 839 & No & 3555 & 282 & 364 & No & 3624 \\
6 & 0.5 & 0.1 & 241 & 299 & 867 & No & 3557 & 275 & 357 & No & 3625 \\
6 & 0.5 & 0.15 & 241 & 298 & 867 & No & 3556 & 284 & 356 & No & 3625 \\
6 & 0.7 & 0.05 & 241 & 367 & 911 & No & 3556 & 367 & 429 & No & 3626 \\
6 & 0.7 & 0.1 & 241 & 365 & 940 & No & 3556 & 361 & 430 & No & 3626 \\
6 & 0.7 & 0.15 & 241 & 377 & 961 & No & 3557 & 363 & 429 & No & 3627 \\
\midrule
7 & 0.3 & 0.05 & 289 & 222 & 1134 & No & 3563 & 233 & 292 & No & 3634 \\
7 & 0.3 & 0.1 & 289 & 228 & 1163 & No & 3562 & 233 & 302 & No & 3634 \\
7 & 0.3 & 0.15 & 289 & 222 & 1170 & No & 3563 & 233 & 296 & No & 3634 \\
7 & 0.5 & 0.05 & 289 & 351 & 1500 & No & 3562 & 350 & 441 & No & 3639 \\
7 & 0.5 & 0.1 & 289 & 354 & 1547 & No & 3563 & 320 & 448 & No & 3641 \\
7 & 0.5 & 0.15 & 289 & 352 & 1547 & No & 3563 & 332 & 447 & No & 3640 \\
7 & 0.7 & 0.05 & 289 & 444 & 1719 & No & 3563 & 448 & 544 & No & 3644 \\
7 & 0.7 & 0.1 & 289 & 436 & 1745 & No & 3563 & 445 & 543 & No & 3645 \\
7 & 0.7 & 0.15 & 289 & 441 & 1740 & No & 3563 & 443 & 541 & No & 3647 \\
\midrule
8 & 0.3 & 0.05 & 216 & 155 & 669 & No & 3469 & 163 & 197 & No & 3611 \\
8 & 0.3 & 0.1 & 216 & 157 & 729 & No & 3434 & 164 & 197 & No & 3612 \\
8 & 0.3 & 0.15 & 216 & 158 & 743 & No & 3350 & 163 & 195 & No & 3612 \\
8 & 0.5 & 0.05 & 216 & 249 & 1011 & No & 3229 & 248 & 299 & No & 3613 \\
8 & 0.5 & 0.1 & 216 & 248 & 1066 & No & 3473 & 245 & 297 & No & 3613 \\
8 & 0.5 & 0.15 & 216 & 249 & 1093 & No & 3256 & 248 & 298 & No & 3613 \\
8 & 0.7 & 0.05 & 216 & 325 & 1203 & No & 3187 & 315 & 377 & No & 3615 \\
8 & 0.7 & 0.1 & 216 & 319 & 1244 & No & 3566 & 314 & 373 & No & 3615 \\
8 & 0.7 & 0.15 & 216 & 313 & 1258 & No & 3568 & 312 & 375 & No & 3615 \\
\midrule
9 & 0.3 & 0.05 & 251 & 176 & 242 & No & 3450 & 173 & 214 & No & 3614 \\
9 & 0.3 & 0.1 & 251 & 175 & 242 & No & 3450 & 172 & 215 & No & 3615 \\
9 & 0.3 & 0.15 & 251 & 176 & 234 & No & 3450 & 173 & 214 & No & 3615 \\
9 & 0.5 & 0.05 & 251 & 262 & 347 & No & 3447 & 263 & 324 & No & 3618 \\
9 & 0.5 & 0.1 & 251 & 261 & 338 & No & 3449 & 255 & 321 & No & 3617 \\
9 & 0.5 & 0.15 & 251 & 257 & 338 & No & 3451 & 256 & 320 & No & 3618 \\
9 & 0.7 & 0.05 & 251 & 344 & 431 & No & 3447 & 336 & 404 & No & 3618 \\
9 & 0.7 & 0.1 & 251 & 336 & 429 & No & 3449 & 333 & 407 & No & 3620 \\
9 & 0.7 & 0.15 & 251 & 329 & 428 & No & 3450 & 329 & 401 & No & 3620 \\
\midrule
10 & 0.3 & 0.05 & 51 & 31 & 31 & Yes & 39 & 31 & 31 & Yes & 2 \\
10 & 0.3 & 0.1 & 51 & 31 & 31 & Yes & 98 & 31 & 31 & Yes & 11 \\
10 & 0.3 & 0.15 & 51 & 28 & 28 & Yes & 133 & 28 & 28 & Yes & 15 \\
10 & 0.5 & 0.05 & 51 & 48 & 48 & Yes & 181 & 48 & 48 & Yes & 52 \\
10 & 0.5 & 0.1 & 51 & 48 & 48 & Yes & 327 & 48 & 48 & Yes & 57 \\
10 & 0.5 & 0.15 & 51 & 48 & 48 & Yes & 496 & 48 & 48 & Yes & 114 \\
10 & 0.7 & 0.05 & 51 & 62 & 62 & Yes & 227 & 62 & 62 & Yes & 76 \\
10 & 0.7 & 0.1 & 51 & 64 & 64 & Yes & 386 & 64 & 64 & Yes & 986 \\
10 & 0.7 & 0.15 & 51 & 64 & 64 & Yes & 583 & 64 & 64 & Yes & 683 \\
\midrule
11 & 0.3 & 0.05 & 66 & 43 & 43 & Yes & 670 & 43 & 43 & Yes & 45 \\
11 & 0.3 & 0.1 & 66 & 43 & 43 & Yes & 1152 & 43 & 43 & Yes & 38 \\
11 & 0.3 & 0.15 & 66 & 43 & 43 & Yes & 1766 & 43 & 43 & Yes & 47 \\
11 & 0.5 & 0.05 & 66 & 67 & 84 & No & 1648 & 67 & 67 & Yes & 192 \\
11 & 0.5 & 0.1 & 66 & 66 & 83 & No & 1775 & 66 & 66 & Yes & 491 \\
11 & 0.5 & 0.15 & 66 & 66 & 90 & No & 1883 & 66 & 66 & Yes & 421 \\
11 & 0.7 & 0.05 & 66 & 86 & 106 & No & 1483 & 86 & 86 & Yes & 1467 \\
11 & 0.7 & 0.1 & 66 & 85 & 107 & No & 1607 & 85 & 90 & No & 3601 \\
11 & 0.7 & 0.15 & 66 & 85 & 106 & No & 1721 & 85 & 86 & No & 3601 \\
\midrule
12 & 0.3 & 0.05 & 78 & 51 & 51 & Yes & 270 & 51 & 51 & Yes & 30 \\
12 & 0.3 & 0.1 & 78 & 50 & 50 & No & 488 & 50 & 50 & Yes & 53 \\
12 & 0.3 & 0.15 & 78 & 50 & 50 & Yes & 696 & 50 & 50 & Yes & 51 \\
12 & 0.5 & 0.05 & 78 & 79 & 79 & Yes & 467 & 79 & 79 & Yes & 304 \\
12 & 0.5 & 0.1 & 78 & 79 & 79 & No & 887 & 79 & 79 & Yes & 471 \\
12 & 0.5 & 0.15 & 78 & 79 & 79 & Yes & 1219 & 79 & 79 & Yes & 292 \\
12 & 0.7 & 0.05 & 78 & 100 & 100 & Yes & 502 & 100 & 100 & Yes & 2540 \\
12 & 0.7 & 0.1 & 78 & 101 & 101 & No & 983 & 101 & 101 & Yes & 3026 \\
12 & 0.7 & 0.15 & 78 & 101 & 101 & Yes & 1288 & 101 & 101 & Yes & 2000 \\
\midrule
13 & 0.3 & 0.05 & 62 & 33 & 33 & Yes & 48 & 33 & 33 & Yes & 12 \\
13 & 0.3 & 0.1 & 62 & 33 & 33 & Yes & 114 & 33 & 33 & Yes & 15 \\
13 & 0.3 & 0.15 & 62 & 32 & 32 & Yes & 149 & 32 & 32 & Yes & 14 \\
13 & 0.5 & 0.05 & 62 & 54 & 54 & Yes & 144 & 54 & 54 & Yes & 39 \\
13 & 0.5 & 0.1 & 62 & 53 & 53 & Yes & 388 & 53 & 53 & Yes & 113 \\
13 & 0.5 & 0.15 & 62 & 52 & 52 & Yes & 540 & 52 & 52 & Yes & 101 \\
13 & 0.7 & 0.05 & 62 & 72 & 72 & Yes & 230 & 72 & 72 & Yes & 341 \\
13 & 0.7 & 0.1 & 62 & 71 & 71 & Yes & 667 & 71 & 71 & Yes & 263 \\
13 & 0.7 & 0.15 & 62 & 71 & 71 & Yes & 979 & 71 & 71 & Yes & 206 \\
\midrule
14 & 0.3 & 0.05 & 116 & 76 & 150 & No & 2623 & 76 & 87 & No & 3603 \\
14 & 0.3 & 0.1 & 116 & 76 & 185 & No & 2638 & 76 & 87 & No & 3602 \\
14 & 0.3 & 0.15 & 116 & 76 & 199 & No & 2634 & 76 & 86 & No & 3603 \\
14 & 0.5 & 0.05 & 116 & 117 & 245 & No & 2495 & 117 & 137 & No & 3603 \\
14 & 0.5 & 0.1 & 116 & 117 & 282 & No & 2498 & 117 & 136 & No & 3603 \\
14 & 0.5 & 0.15 & 116 & 117 & 296 & No & 2496 & 117 & 136 & No & 3603 \\
14 & 0.7 & 0.05 & 116 & 151 & 280 & No & 2388 & 154 & 174 & No & 3603 \\
14 & 0.7 & 0.1 & 116 & 154 & 346 & No & 2402 & 154 & 174 & No & 3603 \\
14 & 0.7 & 0.15 & 116 & 154 & 354 & No & 2409 & 153 & 172 & No & 3603 \\
\midrule
15 & 0.3 & 0.05 & 74 & 44 & 44 & Yes & 43 & 44 & 44 & Yes & 53 \\
15 & 0.3 & 0.1 & 74 & 42 & 42 & Yes & 51 & 42 & 42 & Yes & 89 \\
15 & 0.3 & 0.15 & 74 & 42 & 42 & Yes & 61 & 42 & 42 & Yes & 79 \\
15 & 0.5 & 0.05 & 74 & 66 & 66 & Yes & 300 & 66 & 66 & Yes & 372 \\
15 & 0.5 & 0.1 & 74 & 65 & 65 & Yes & 1015 & 65 & 65 & Yes & 1601 \\
15 & 0.5 & 0.15 & 74 & 64 & 64 & Yes & 1626 & 64 & 64 & Yes & 901 \\
15 & 0.7 & 0.05 & 74 & 85 & 85 & Yes & 2031 & 85 & 85 & Yes & 1442 \\
15 & 0.7 & 0.1 & 74 & 84 & 87 & No & 3420 & 84 & 88 & No & 3601 \\
15 & 0.7 & 0.15 & 74 & 84 & 87 & No & 3421 & 84 & 88 & No & 3601 \\
\midrule
16 & 0.3 & 0.05 & 87 & 48 & 48 & Yes & 226 & 48 & 48 & Yes & 19 \\
16 & 0.3 & 0.1 & 87 & 47 & 47 & No & 598 & 47 & 47 & Yes & 34 \\
16 & 0.3 & 0.15 & 87 & 47 & 47 & Yes & 986 & 47 & 47 & Yes & 38 \\
16 & 0.5 & 0.05 & 87 & 74 & 74 & Yes & 634 & 74 & 74 & Yes & 115 \\
16 & 0.5 & 0.1 & 87 & 73 & 110 & No & 1660 & 73 & 73 & Yes & 293 \\
16 & 0.5 & 0.15 & 87 & 72 & 121 & No & 2013 & 72 & 72 & Yes & 338 \\
16 & 0.7 & 0.05 & 87 & 97 & 97 & Yes & 1335 & 97 & 97 & Yes & 798 \\
16 & 0.7 & 0.1 & 87 & 96 & 157 & No & 1988 & 96 & 96 & Yes & 3018 \\
16 & 0.7 & 0.15 & 87 & 95 & 164 & No & 2132 & 95 & 95 & Yes & 3226 \\
\midrule
17 & 0.3 & 0.05 & 159 & 123 & 136 & No & 3460 & 122 & 138 & No & 3605 \\
17 & 0.3 & 0.1 & 159 & 123 & 166 & No & 3467 & 123 & 138 & No & 3604 \\
17 & 0.3 & 0.15 & 159 & 123 & 143 & No & 3467 & 123 & 138 & No & 3605 \\
17 & 0.5 & 0.05 & 159 & 182 & 205 & No & 3461 & 180 & 207 & No & 3605 \\
17 & 0.5 & 0.1 & 159 & 183 & 229 & No & 3462 & 180 & 204 & No & 3605 \\
17 & 0.5 & 0.15 & 159 & 182 & 230 & No & 3462 & 181 & 205 & No & 3605 \\
17 & 0.7 & 0.05 & 159 & 231 & 279 & No & 3458 & 229 & 255 & No & 3606 \\
17 & 0.7 & 0.1 & 159 & 231 & 279 & No & 3459 & 228 & 255 & No & 3606 \\
17 & 0.7 & 0.15 & 159 & 229 & 282 & No & 3461 & 225 & 253 & No & 3606 \\
\midrule
18 & 0.3 & 0.05 & 47 & 38 & 38 & Yes & 4 & 38 & 38 & Yes & 4 \\
18 & 0.3 & 0.1 & 47 & 37 & 37 & No & 13 & 37 & 37 & Yes & 13 \\
18 & 0.3 & 0.15 & 47 & 37 & 37 & Yes & 9 & 37 & 37 & Yes & 13 \\
18 & 0.5 & 0.05 & 47 & 57 & 57 & Yes & 14 & 57 & 57 & Yes & 37 \\
18 & 0.5 & 0.1 & 47 & 56 & 56 & Yes & 214 & 56 & 56 & Yes & 142 \\
18 & 0.5 & 0.15 & 47 & 56 & 56 & Yes & 143 & 56 & 56 & Yes & 153 \\
18 & 0.7 & 0.05 & 47 & 71 & 71 & Yes & 51 & 71 & 71 & Yes & 58 \\
18 & 0.7 & 0.1 & 47 & 71 & 71 & Yes & 424 & 71 & 71 & Yes & 392 \\
18 & 0.7 & 0.15 & 47 & 70 & 70 & Yes & 574 & 70 & 70 & Yes & 627 \\
\midrule
19 & 0.3 & 0.05 & 170 & 130 & 221 & No & 3526 & 128 & 152 & No & 3607 \\
19 & 0.3 & 0.1 & 170 & 129 & 229 & No & 3527 & 129 & 151 & No & 3607 \\
19 & 0.3 & 0.15 & 170 & 129 & 228 & No & 3527 & 125 & 152 & No & 3607 \\
19 & 0.5 & 0.05 & 170 & 199 & 296 & No & 3524 & 197 & 232 & No & 3608 \\
19 & 0.5 & 0.1 & 170 & 198 & 298 & No & 3525 & 198 & 230 & No & 3608 \\
19 & 0.5 & 0.15 & 170 & 199 & 301 & No & 3524 & 195 & 227 & No & 3608 \\
19 & 0.7 & 0.05 & 170 & 254 & 349 & No & 3516 & 244 & 290 & No & 3608 \\
19 & 0.7 & 0.1 & 170 & 253 & 356 & No & 3519 & 247 & 290 & No & 3608 \\
19 & 0.7 & 0.15 & 170 & 250 & 355 & No & 3520 & 247 & 289 & No & 3608 \\
\midrule
20 & 0.3 & 0.05 & 56 & 44 & 44 & No & 33 & 44 & 44 & Yes & 15 \\
20 & 0.3 & 0.1 & 56 & 44 & 44 & Yes & 75 & 44 & 44 & Yes & 23 \\
20 & 0.3 & 0.15 & 56 & 42 & 42 & Yes & 95 & 42 & 42 & Yes & 26 \\
20 & 0.5 & 0.05 & 56 & 69 & 69 & No & 68 & 69 & 69 & Yes & 29 \\
20 & 0.5 & 0.1 & 56 & 67 & 67 & Yes & 146 & 67 & 67 & Yes & 84 \\
20 & 0.5 & 0.15 & 56 & 67 & 67 & Yes & 210 & 67 & 67 & Yes & 90 \\
20 & 0.7 & 0.05 & 56 & 86 & 86 & No & 80 & 86 & 86 & Yes & 60 \\
20 & 0.7 & 0.1 & 56 & 86 & 86 & Yes & 165 & 86 & 86 & Yes & 274 \\
20 & 0.7 & 0.15 & 56 & 86 & 86 & Yes & 249 & 86 & 86 & Yes & 360 \\
\midrule
21 & 0.3 & 0.05 & 216 & 130 & 592 & No & 2865 & 134 & 160 & No & 3610 \\
21 & 0.3 & 0.1 & 216 & 130 & 661 & No & 2868 & 135 & 162 & No & 3610 \\
21 & 0.3 & 0.15 & 216 & 127 & 698 & No & 2888 & 134 & 157 & No & 3610 \\
21 & 0.5 & 0.05 & 216 & 207 & 851 & No & 2811 & 209 & 250 & No & 3612 \\
21 & 0.5 & 0.1 & 216 & 207 & 892 & No & 2852 & 208 & 249 & No & 3612 \\
21 & 0.5 & 0.15 & 216 & 203 & 947 & No & 2882 & 205 & 246 & No & 3612 \\
21 & 0.7 & 0.05 & 216 & 272 & 988 & No & 2817 & 274 & 321 & No & 3613 \\
21 & 0.7 & 0.1 & 216 & 273 & 1056 & No & 2853 & 274 & 326 & No & 3613 \\
21 & 0.7 & 0.15 & 216 & 271 & 1086 & No & 2909 & 272 & 325 & No & 3613 \\
\midrule
22 & 0.3 & 0.05 & 118 & 90 & 187 & No & 2804 & 91 & 96 & No & 3603 \\
22 & 0.3 & 0.1 & 118 & 90 & 200 & No & 2799 & 91 & 95 & No & 3603 \\
22 & 0.3 & 0.15 & 118 & 91 & 191 & No & 2799 & 91 & 96 & No & 3603 \\
22 & 0.5 & 0.05 & 118 & 129 & 258 & No & 2629 & 134 & 143 & No & 3603 \\
22 & 0.5 & 0.1 & 118 & 132 & 258 & No & 2646 & 132 & 142 & No & 3603 \\
22 & 0.5 & 0.15 & 118 & 132 & 255 & No & 2648 & 132 & 142 & No & 3603 \\
22 & 0.7 & 0.05 & 118 & 165 & 290 & No & 2492 & 170 & 182 & No & 3604 \\
22 & 0.7 & 0.1 & 118 & 166 & 292 & No & 2533 & 168 & 180 & No & 3604 \\
22 & 0.7 & 0.15 & 118 & 166 & 285 & No & 2481 & 166 & 180 & No & 3604 \\
\midrule
23 & 0.3 & 0.05 & 93 & 66 & 114 & No & 3045 & 66 & 66 & Yes & 781 \\
23 & 0.3 & 0.1 & 93 & 64 & 122 & No & 3032 & 64 & 64 & Yes & 800 \\
23 & 0.3 & 0.15 & 93 & 64 & 111 & No & 3033 & 64 & 64 & Yes & 942 \\
23 & 0.5 & 0.05 & 93 & 98 & 170 & No & 2639 & 98 & 103 & No & 3602 \\
23 & 0.5 & 0.1 & 93 & 97 & 176 & No & 2655 & 97 & 102 & No & 3602 \\
23 & 0.5 & 0.15 & 93 & 96 & 173 & No & 2664 & 96 & 101 & No & 3602 \\
23 & 0.7 & 0.05 & 93 & 122 & 204 & No & 2369 & 121 & 131 & No & 3602 \\
23 & 0.7 & 0.1 & 93 & 121 & 203 & No & 2387 & 121 & 129 & No & 3602 \\
23 & 0.7 & 0.15 & 93 & 121 & 200 & No & 2391 & 121 & 129 & No & 3602 \\
\midrule
24 & 0.3 & 0.05 & 337 & 257 & 1287 & No & 3568 & 258 & 343 & No & 3646 \\
24 & 0.3 & 0.1 & 337 & 247 & 1301 & No & 3568 & 256 & 347 & No & 3647 \\
24 & 0.3 & 0.15 & 337 & 252 & 1297 & No & 3568 & 260 & 343 & No & 3648 \\
24 & 0.5 & 0.05 & 337 & 405 & 1700 & No & 3568 & 390 & 516 & No & 3657 \\
24 & 0.5 & 0.1 & 337 & 407 & 1735 & No & 3568 & 381 & 516 & No & 3659 \\
24 & 0.5 & 0.15 & 337 & 404 & 1720 & No & 3569 & 389 & 512 & No & 3657 \\
24 & 0.7 & 0.05 & 337 & 500 & 1950 & No & 3568 & 512 & 638 & No & 3666 \\
24 & 0.7 & 0.1 & 337 & 514 & 162665 & No & 3568 & 511 & 637 & No & 3666 \\
24 & 0.7 & 0.15 & 337 & 503 & 163539 & No & 3569 & 511 & 637 & No & 3667 \\
\midrule
25 & 0.3 & 0.05 & 70 & 41 & 41 & Yes & 1490 & 41 & 41 & Yes & 92 \\
25 & 0.3 & 0.1 & 70 & 40 & 40 & Yes & 2185 & 40 & 40 & Yes & 246 \\
25 & 0.3 & 0.15 & 70 & 40 & 54 & No & 2885 & 40 & 40 & Yes & 99 \\
25 & 0.5 & 0.05 & 70 & 64 & 113 & No & 2829 & 64 & 64 & Yes & 1139 \\
25 & 0.5 & 0.1 & 70 & 65 & 121 & No & 2819 & 65 & 65 & Yes & 1517 \\
25 & 0.5 & 0.15 & 70 & 64 & 118 & No & 2830 & 64 & 64 & Yes & 827 \\
25 & 0.7 & 0.05 & 70 & 83 & 118 & No & 2492 & 83 & 88 & No & 3601 \\
25 & 0.7 & 0.1 & 70 & 85 & 124 & No & 2507 & 85 & 89 & No & 3601 \\
25 & 0.7 & 0.15 & 70 & 85 & 138 & No & 2498 & 85 & 88 & No & 3601 \\
\midrule
26 & 0.3 & 0.05 & 162 & 106 & 295 & No & 2768 & 107 & 122 & No & 3605 \\
26 & 0.3 & 0.1 & 162 & 105 & 405 & No & 2761 & 107 & 121 & No & 3605 \\
26 & 0.3 & 0.15 & 162 & 104 & 453 & No & 2762 & 107 & 121 & No & 3605 \\
26 & 0.5 & 0.05 & 162 & 167 & 535 & No & 2678 & 170 & 193 & No & 3606 \\
26 & 0.5 & 0.1 & 162 & 166 & 616 & No & 2672 & 170 & 192 & No & 3606 \\
26 & 0.5 & 0.15 & 162 & 164 & 656 & No & 2693 & 168 & 191 & No & 3606 \\
26 & 0.7 & 0.05 & 162 & 220 & 638 & No & 2636 & 223 & 250 & No & 3606 \\
26 & 0.7 & 0.1 & 162 & 221 & 807 & No & 2657 & 222 & 246 & No & 3606 \\
26 & 0.7 & 0.15 & 162 & 220 & 812 & No & 2677 & 220 & 246 & No & 3606 \\
\midrule
27 & 0.3 & 0.05 & 165 & 100 & 341 & No & 2769 & 106 & 123 & No & 3605 \\
27 & 0.3 & 0.1 & 165 & 100 & 385 & No & 2762 & 105 & 120 & No & 3606 \\
27 & 0.3 & 0.15 & 165 & 100 & 392 & No & 2771 & 104 & 120 & No & 3606 \\
27 & 0.5 & 0.05 & 165 & 161 & 504 & No & 2683 & 165 & 192 & No & 3606 \\
27 & 0.5 & 0.1 & 165 & 161 & 581 & No & 2701 & 162 & 190 & No & 3606 \\
27 & 0.5 & 0.15 & 165 & 160 & 587 & No & 2686 & 163 & 190 & No & 3606 \\
27 & 0.7 & 0.05 & 165 & 216 & 610 & No & 2640 & 216 & 248 & No & 3606 \\
27 & 0.7 & 0.1 & 165 & 212 & 680 & No & 2687 & 214 & 248 & No & 3607 \\
27 & 0.7 & 0.15 & 165 & 211 & 693 & No & 2689 & 212 & 246 & No & 3607 \\
\midrule
28 & 0.3 & 0.05 & 116 & 85 & 162 & No & 2638 & 85 & 85 & Yes & 3589 \\
28 & 0.3 & 0.1 & 116 & 85 & 172 & No & 2639 & 85 & 85 & Yes & 1843 \\
28 & 0.3 & 0.15 & 116 & 84 & 177 & No & 2604 & 84 & 84 & Yes & 2155 \\
28 & 0.5 & 0.05 & 116 & 124 & 228 & No & 2253 & 124 & 132 & No & 3603 \\
28 & 0.5 & 0.1 & 116 & 123 & 214 & No & 2281 & 123 & 131 & No & 3603 \\
28 & 0.5 & 0.15 & 116 & 122 & 227 & No & 2243 & 122 & 129 & No & 3603 \\
28 & 0.7 & 0.05 & 116 & 155 & 258 & No & 2049 & 154 & 168 & No & 3603 \\
28 & 0.7 & 0.1 & 116 & 155 & 256 & No & 2074 & 154 & 166 & No & 3603 \\
28 & 0.7 & 0.15 & 116 & 154 & 259 & No & 2050 & 154 & 164 & No & 3603 \\
\midrule
29 & 0.3 & 0.05 & 162 & 111 & 313 & No & 2652 & 111 & 126 & No & 3605 \\
29 & 0.3 & 0.1 & 162 & 111 & 340 & No & 2655 & 111 & 128 & No & 3605 \\
29 & 0.3 & 0.15 & 162 & 111 & 341 & No & 2601 & 112 & 126 & No & 3605 \\
29 & 0.5 & 0.05 & 162 & 173 & 448 & No & 2359 & 172 & 196 & No & 3606 \\
29 & 0.5 & 0.1 & 162 & 171 & 468 & No & 2392 & 172 & 194 & No & 3606 \\
29 & 0.5 & 0.15 & 162 & 171 & 474 & No & 2351 & 173 & 197 & No & 3606 \\
29 & 0.7 & 0.05 & 162 & 226 & 546 & No & 2237 & 226 & 252 & No & 3607 \\
29 & 0.7 & 0.1 & 162 & 224 & 561 & No & 2282 & 224 & 251 & No & 3607 \\
29 & 0.7 & 0.15 & 162 & 222 & 562 & No & 2258 & 223 & 253 & No & 3607 \\
\midrule
30 & 0.3 & 0.05 & 162 & 135 & 285 & No & 2553 & 135 & 139 & No & 3605 \\
30 & 0.3 & 0.1 & 162 & 133 & 273 & No & 2528 & 133 & 137 & No & 3605 \\
30 & 0.3 & 0.15 & 162 & 133 & 292 & No & 2519 & 133 & 137 & No & 3605 \\
30 & 0.5 & 0.05 & 162 & 194 & 354 & No & 2229 & 194 & 203 & No & 3606 \\
30 & 0.5 & 0.1 & 162 & 193 & 353 & No & 2236 & 193 & 203 & No & 3606 \\
30 & 0.5 & 0.15 & 162 & 193 & 365 & No & 2212 & 193 & 202 & No & 3606 \\
30 & 0.7 & 0.05 & 162 & 241 & 404 & No & 2104 & 243 & 253 & No & 3607 \\
30 & 0.7 & 0.1 & 162 & 240 & 400 & No & 2121 & 241 & 250 & No & 3607 \\
30 & 0.7 & 0.15 & 162 & 241 & 417 & No & 2114 & 240 & 251 & No & 3607 \\
\midrule
31 & 0.3 & 0.05 & 237 & 188 & 479 & No & 3563 & 185 & 225 & No & 3617 \\
31 & 0.3 & 0.1 & 237 & 187 & 475 & No & 3565 & 182 & 224 & No & 3617 \\
31 & 0.3 & 0.15 & 237 & 188 & 473 & No & 3564 & 176 & 223 & No & 3617 \\
31 & 0.5 & 0.05 & 237 & 273 & 581 & No & 3560 & 270 & 321 & No & 3619 \\
31 & 0.5 & 0.1 & 237 & 273 & 583 & No & 3563 & 265 & 322 & No & 3621 \\
31 & 0.5 & 0.15 & 237 & 275 & 587 & No & 3563 & 266 & 321 & No & 3620 \\
31 & 0.7 & 0.05 & 237 & 341 & 646 & No & 3560 & 345 & 400 & No & 3623 \\
31 & 0.7 & 0.1 & 237 & 342 & 661 & No & 3566 & 334 & 398 & No & 3623 \\
31 & 0.7 & 0.15 & 237 & 342 & 654 & No & 3565 & 330 & 396 & No & 3624 \\
\midrule
32 & 0.3 & 0.05 & 237 & 209 & 459 & No & 3563 & 207 & 230 & No & 3617 \\
32 & 0.3 & 0.1 & 237 & 209 & 477 & No & 3564 & 200 & 229 & No & 3617 \\
32 & 0.3 & 0.15 & 237 & 209 & 468 & No & 3564 & 205 & 230 & No & 3617 \\
32 & 0.5 & 0.05 & 237 & 303 & 551 & No & 3560 & 298 & 327 & No & 3620 \\
32 & 0.5 & 0.1 & 237 & 301 & 572 & No & 3562 & 299 & 327 & No & 3620 \\
32 & 0.5 & 0.15 & 237 & 302 & 576 & No & 3560 & 294 & 326 & No & 3619 \\
32 & 0.7 & 0.05 & 237 & 367 & 628 & No & 3559 & 369 & 401 & No & 3622 \\
32 & 0.7 & 0.1 & 237 & 367 & 637 & No & 3564 & 366 & 396 & No & 3623 \\
32 & 0.7 & 0.15 & 237 & 367 & 635 & No & 3564 & 347 & 396 & No & 3622 \\
\midrule
33 & 0.3 & 0.05 & 239 & 189 & 503 & No & 3018 & 190 & 217 & No & 3617 \\
33 & 0.3 & 0.1 & 239 & 188 & 503 & No & 3004 & 189 & 215 & No & 3618 \\
33 & 0.3 & 0.15 & 239 & 188 & 499 & No & 3004 & 189 & 216 & No & 3618 \\
33 & 0.5 & 0.05 & 239 & 284 & 616 & No & 2991 & 283 & 318 & No & 3620 \\
33 & 0.5 & 0.1 & 239 & 283 & 617 & No & 2976 & 279 & 319 & No & 3620 \\
33 & 0.5 & 0.15 & 239 & 282 & 594 & No & 2986 & 285 & 316 & No & 3620 \\
33 & 0.7 & 0.05 & 239 & 357 & 693 & No & 3043 & 352 & 399 & No & 3622 \\
33 & 0.7 & 0.1 & 239 & 359 & 696 & No & 3009 & 344 & 396 & No & 3623 \\
33 & 0.7 & 0.15 & 239 & 359 & 694 & No & 3003 & 353 & 398 & No & 3623 \\
\midrule
34 & 0.3 & 0.05 & 262 & 213 & 554 & No & 2968 & 202 & 243 & No & 3622 \\
34 & 0.3 & 0.1 & 262 & 214 & 568 & No & 2977 & 204 & 243 & No & 3622 \\
34 & 0.3 & 0.15 & 262 & 212 & 559 & No & 2979 & 202 & 241 & No & 3622 \\
34 & 0.5 & 0.05 & 262 & 310 & 665 & No & 2919 & 311 & 347 & No & 3625 \\
34 & 0.5 & 0.1 & 262 & 311 & 659 & No & 2929 & 311 & 347 & No & 3625 \\
34 & 0.5 & 0.15 & 262 & 308 & 682 & No & 2924 & 312 & 346 & No & 3625 \\
34 & 0.7 & 0.05 & 262 & 386 & 742 & No & 2966 & 388 & 428 & No & 3628 \\
34 & 0.7 & 0.1 & 262 & 390 & 744 & No & 2977 & 386 & 427 & No & 3628 \\
34 & 0.7 & 0.15 & 262 & 388 & 754 & No & 2980 & 387 & 424 & No & 3629 \\
\end{longtable}

\end{document}